\documentclass[letterpaper, 10pt,reqno]{amsart}

\usepackage[dvipsnames]{xcolor}
\usepackage[portrait,margin=3cm]{geometry}
\usepackage{mathrsfs}

\usepackage{amssymb,amsthm,amsfonts,amsbsy,amstext,latexsym}
\usepackage[reqno]{amsmath}
\usepackage[normalem]{ulem}

\usepackage{caption}
\usepackage{subcaption}

\usepackage{color}
\usepackage{bbm}
\usepackage{comment}
\usepackage{algorithm}
\usepackage{algpseudocode}

\usepackage{a4wide,graphicx, listings,lscape, epstopdf, units, textcomp, fancyhdr, url, soul, color}
\usepackage{tikz}
\usetikzlibrary{shapes.geometric}
\usepackage[textsize=small]{todonotes}
\usepackage{enumitem}

\usepackage{mathtools}
\mathtoolsset{showonlyrefs}

\usepackage{xifthen}% provides \isempty test

\definecolor{MyDarkblue}{rgb}{0,0.08,0.50}
\definecolor{Brickred}{rgb}{0.65,0.08,0}

\usepackage{hyperref,upref}
\hypersetup{
	colorlinks=true,       % false: boxed links; true: colored links
	linkcolor=MyDarkblue,          % color of internal links
	citecolor=Brickred,        % color of links to bibliography
	filecolor=red,      % color of file links
	urlcolor=cyan           % color of external links
}

\newtheorem*{theorem*}{Theorem}
\newtheorem{theorem}{Theorem}[section]
\newtheorem{lemma}[theorem]{Lemma}
\newtheorem{proposition}[theorem]{Proposition}
\newtheorem{corollary}[theorem]{Corollary}

\newtheorem{claim}[theorem]{Claim}

\theoremstyle{definition}
\newtheorem{definition}[theorem]{Definition}

\newtheorem{remark}[theorem]{Remark}

\newcommand{\eps}{\varepsilon}

\newcommand{\cA}{\mathcal{A}}

\newcommand{\cG}{\mathcal{G}}

\newcommand{\sss}{\scriptscriptstyle}

\DeclareSymbolFont{extraup}{U}{zavm}{m}{n}
\DeclareMathSymbol{\varheart}{\mathalpha}{extraup}{86}
\DeclareMathSymbol{\vardiamond}{\mathalpha}{extraup}{87}

\newcommand{\R}{\mathbb{R}}

\newcommand{\Z}{\mathbb{Z}}

\newcommand{\cov}{{\rm cov}}

\newcommand*{\be}{\begin{equation}}
	\newcommand*{\ee}{\end{equation}}
\newcommand*{\ba}{\begin{aligned}}
	\newcommand*{\ea}{\end{aligned}}
\newcommand*{\barr}{\begin{array}{c}}
	\newcommand*{\earr}{\end{array}}

\makeatletter
\def\namedlabel#1#2{\begingroup
	#2%
	\def\@currentlabel{#2}%
	\phantomsection\label{#1}\endgroup
}

\makeatother
\newcommand{\bes}{\begin{equation*}}
	\newcommand{\ees}{\end{equation*}}

\newcommand{\cm}[1]{{\color{orange} #1}}
\newcommand{\cmc}[1]{{\color{orange}{ \bf [~Carlos:\ }\emph{#1}\textbf{~]}}}

\numberwithin{equation}{section}

\newcommand{\invisible}[2]{%
    \ifthenelse{\isempty{#1}}
    {}% if #1 is empty
    {#2}% if #1 is not empty
}

\newcommand{\D}[1]{\D_{#1}^{\sss (u,v)}}
\newcommand{\dm}[1]{{\color{red} #1}}
\newcommand{\dmc}[1]{{\color{red}{\bf [~Dieter:\ } \color{red}{\em #1}\color{red}{\bf~]}}}
\newcommand{\revisor}[3]{{\color{blue}{\bf [~Revisor:\ } \color{blue}{\em #1}\color{blue}{\bf~]}}}

\makeatletter
\newcommand{\leqnomode}{\tagsleft@true\let\veqno\@@leqno}
\newcommand{\reqnomode}{\tagsleft@false\let\veqno\@@eqno}
\makeatother
\newlength{\tagmarginsep} % Distance required
\tikzstyle{vertex}=[circle,fill=orange!60,minimum size=10pt,inner sep=0pt]
\tikzstyle{tedge} = [draw,ultra thick,->,>=stealth, orange]
\tikzstyle{esq}=[circle,fill=white,minimum size=10pt,inner sep=0pt]
\tikzstyle{up}=[<-,>=stealth]
\title{On the jump of the cover time in random geometric graphs}
\author[Carlos Martinez-Arevalo]{Carlos Martinez-Arevalo}
\address[Carlos Martinez]{IAM, University of Bonn, Germany}
\email{cmartine@uni-bonn.de}

\author[Dieter Mitsche]{Dieter Mitsche}
\address[Dieter Mitsche]{IMC, Pontificia Universidad Católica, Avda.\ Vicu\~na Mackenna 4860, Santiago, Chile}
\email{dieter.mitsche@uc.cl}
\thanks{Dieter Mitsche has been supported by Fondecyt grant 1261336. Carlos Martinez-Arevalo has been supported by the Deutsche Forschungsgemeinschaft (DFG, German Research Foundation) – Projektnummer 552316285.}

\begin{document}

\begin{abstract}
    In this paper we study the cover time of the simple random walk on the giant component of supercritical $d$-dimensional random geometric graphs on $\mathrm{Poi}(n)$ vertices. We show that the cover time undergoes a jump at the connectivity threshold radius $r_c$: with $r_g$ denoting the threshold for having a giant component, we show that if the radius $r$ satisfies $(1+\eps)r_g \le r \le (1-\eps)r_c$ for $\eps > 0$ arbitrarily small, the cover time of the giant component is asymptotically almost surely $\Theta(n \log^2 n$). On the other hand, we show that for $r \ge (1+\eps)r_c$, the cover time of the graph is asymptotically almost surely $\Theta(n \log n)$ (which was known for $d=2$ only for a radius larger by a constant factor). Our proofs also shed some light onto the behavior around $r_c$.
\end{abstract}
\maketitle
%{\scshape Carlos Martinez} \hfill {\scshape \large Cover times rgg} \hfill {}
%\smallskip
%\hrule
\noindent

MSC Class: 05C80, 60D05, 05C81
\section{Introduction}

%\cm{Reviewers only had small comments in this part. Maybe we benefit if we rewrite this. I am not sure and don't have a strong opinion. In case we rewrite this part, I would prefer to include less technical details until we reach the subsection of main results. Let me know what you prefer so I can make the modifications.}
%\dmc{We could make $\mathcal{G}_n^{r,d} $ (or simply $\mathcal{G}_n$ when $r,d$ are clear) for both a graph and the distribution and then talk about $\mathcal{G}_n$ just instead of this $G \in \mathcal{G}_n$.}
%We denote by $\cG_n^{r,d}$ the model of random geometric graphs defined as follows: let $d$ and $n$ be positive integers, and let $r$ be a positive real number, possibly depending on $n$. The vertices of the graph are the points of a Poisson Point Process (PPP) of intensity $1$ in $\Lambda_n=\left[-\frac{n^{1/d}}{2}, \frac{n^{1/d}}{2}\right]^d$, that is, the $d$-dimensional axis-parallel cube of volume $n$ centered at the origin, and two vertices are connected by an edge if their Euclidean distance is at most $r$.

Random geometric graphs were originally introduced by Gilbert~\cite{Gil61} as a model for telecommunication networks, and since then they received a lot of attention both from a theoretical point of view (see for example the monograph of Penrose~\cite{Pen03}, or also~\cite{Pen16}) as well as from an applied point of view (see for example~\cite{Aky02, Nek07, FLMPSSSvL19}).

%\cmc{I feel that before the preentation of this part was more natural.}\dmc{lo puse un póco mas tarde}

\subsection{Model and main results}%\dmc{all rewritten}
The model of \emph{random geometric graphs (RGGs)} is defined as follows:
For $r>0$ and $d, n\in\mathbb{N}$ we denote by $\mathcal{G}_n^{r,d}$ (or rather $\mathcal{G}_n$ when the dependence on $r$ and $d$ is clear)  the RGG of radius $r$ inside the $d$-dimensional cube $\Lambda_n = \left[ \frac{-n^{1/d}}{2}, \frac{n^{1/d}}{2}\right]^d$ of volume $n$ whose vertex set $\mathcal{V}_n$ is the set of points of a Poisson Point Process (PPP) of unit intensity in $\Lambda_n$ and edge set 
\begin{equation}
    \mathcal{E}_n^r = \{\{x,y\}:x,y \in \mathcal{V}_n, \lVert x-y \rVert_2 \leq r\}.
\end{equation} 
We define the infinite counterpart $\mathcal{G}^{r,d}$ similarly, but replace $\Lambda_n$ by $\mathbb{R}^d$ with one artificial vertex added at the origin. We define then the \emph{thermodynamical threshold $r_g$} as %\dmc{por favor no estas frases tan largas, ya se quejaron una vez. Corregi la definicion tambien para hacerla más final}
\begin{equation}\label{eq:definition_RG}
\begin{aligned}
r_g=r_g(d)=\inf\{ r: \mathbb{P}(\text{the origin is in an infinite component of $\cG$} )>0\}.
\end{aligned}
\end{equation}
It is well 
known that $0<r_g<\infty$ for every $d\ge 2$, although its exact value is not known (see~\cite{Pen03} for estimates in $d=2$). When restricting $\mathbb{R}^d$ to $\Lambda_n$, denote  by $L_1(\mathcal{G}_n)$ or simply, $L_1$, the largest  component of $\mathcal{G}_n$ (ties broken arbitrarily; note that in this case the largest component not necessarily contains the origin). When $r > r_g$, we will refer to this component also as the \emph{giant component} or simply giant.
 Define also \emph{the connectivity threshold $r_c=r_c(d)$} as follows: for any $\eps \ge 0$,
 %\dmc{Why is it ambiguous? I changed it a bit}
\begin{equation}\label{eq:defintion_RC}
\begin{aligned}
r_c^n(d)=\inf\{ r: \mathbb{P}(\cG_n^{r,d} \text{ is connected for $r \ge (1+\eps)r_c$}) \ge \frac12\}
\end{aligned} 
\end{equation}
We will drop $n$ from the notation above when the setup is clear. It is well known that $r_c$ exists and that $r_c=(\log n /V_d)^{1/d}$ (see~\cite{GuptaKumar}), where $V_d$ is the volume of the unit ball in $d$ dimensions. 
 
 %$cG_n$Define also $r_g=r_g(n,d)$ as the smallest radius such that a.a.s. $G$ with such a radius has a giant component, that is, a component of size linear in $n$. \cmc{I would like to improve both the definitions of $r_c$ and $r_g$, otherwise define formally in another section and say that here the definitions are informal and formal definitions are given later.} It is well known that $r_g=\Theta(1)$\cmc{We should define $\Theta,O,etc$ before} (the existence of such a threshold is well known, the exact value is only conjectured, even in $2$ dimensions). For rigorous lower and upper bounds on $r_g$ see~\cite{Pen03}. 

%\dmc{Puse aca la primera notacion}

On the other hand, the simple random walk on the vertices of a  finite and connected graph $G=(V,E)$ is a discrete-time stochastic process with random variables $X_0, X_1, \ldots$ such that $X_0=v$ for some $v \in V$, and for $i \ge 0$, $X_{i+1}$ is a vertex chosen uniformly at random from the neighbors of $X_i$. Due to its numerous applications in computer science~\cite{AF}, one of the most studied properties of random walks is the \emph{cover time}, defined as 
$$
\tau_{cov}(G)=\max_{v \in V, X_0=v}\mathbb{E}[ \min\{ t \ge 0: \bigcup_{s=0}^t \{X_s\}=V \}].
$$
In words, the cover time is the expected time it takes to visit all vertices, when starting from a worst-case vertex.

In this paper we consider the cover time of the simple random walk in $L_1$. In order to state our results, we use standard asymptotic notation: specifically, if $(a_n)_{n}$, $(b_n)_{n}$ are sequences of real numbers, we write
$a_n = O(b_n)$ if for some positive constant $C > 0$ and non-negative integer $n_0$ it holds that~$|a_n| \le C|b_n|$
for all $n > n_0$. Also, we write $a_n = \Omega(b_n)$ if $b_n = O(a_n)$, and $a_n=\Theta(b_n)$ if~$a_n = O(b_n)$ and
$a_n=\Omega(b_n)$. We write $a_n = o(b_n)$ if~$a_n/b_n\to 0$ as $n\to\infty$ and $a_n=\omega(b_n)$ if~$b_n=o(a_n)$. 
Previously, Avin and Ercal (see~\cite{covertimergg}) showed that for $G \in \mathcal{G}_n^{r,2}$ with $r^2 > 8 \log n$, that is, a factor $8\pi$ above the connectivity threshold, $\tau_{cov}(G)$ is of the order $n\log n$ a.a.s., where here and in the following, for a sequence $\{\cA_n\}_{n}$ of events, we say that $\cA_n$ holds 
\begin{center} 
\emph{asymptotically almost surely (a.a.s.)}, if $\Pr[\cA_n] \to 1$, as $n \to \infty$. 
\end{center}
Also, Cooper and Frieze~\cite{cooper2011cover} showed that for $d \ge 3$ and $\cG_n$ with $r=(c \log n / V_d)^{1/d}$, with $c > 1$ being arbitrary, a.a.s.\ $\tau_{cov}=c \log \left(\frac{c}{c-1}\right) n \log n$. Note that their result does not apply for $d=2$. In this paper, we look at sparser graphs not dealt with by the two papers. In particular, we may assume throughout the paper that $r^d=O(\log n)$.
Our first main result then is the following:
\begin{theorem}\label{ref:main1}
    Let $\eps > 0$ be arbitrarily small, and let $d \ge 2$ %\dmc{Ok. No hablamos de $d=1$ entonces}
    . Let  $(1+\eps)r_g \le r \le (1-\eps)r_c$, and consider $\cG_n^{r,d}$. Then, a.a.s., $$\tau_{cov}(L_1)=\Theta(n \log^2 n),$$ where the hidden constants depend only on $\eps,d$. 
    %\cmc{We say $G\in \mathcal{G}$, I think we should improve notation here, as the cover time should also include edges. Maybe just say $\mathcal{G}$ instead of introducing $G$ too?}\dmc{As I suggested before, maybe $\mathcal{G}_n$?} 
\end{theorem}
%The proof is done in two separate parts, the lower bound is proved in Theorem~\ref{thm:lower_bound_cov_r_not_fixed} and the upper bound in Theorem~\ref{thm:upper_bound_cov_non_fixed_r}.
Our second main result complements the previous result in $d=2$ and closes the gap of previous papers above the connectivity threshold:
\begin{theorem}\label{ref:main2}
    Let $\eps > 0$ be an arbitrarily small constant. Let $d=2$. Let $r \ge (1+\eps)r_c$, and consider $\cG_n^{r,d}$. Then, a.a.s.\,
    $$\tau_{cov}(G)=\Theta(n \log n),$$ where the hidden constant depends only on $\varepsilon$.
\end{theorem}
%\cmc{We have to make the same change here regarding the constants.}
\begin{remark}
  %  \begin{enumerate}
    %    \item \cm{We remark the importance of $\eps$ in the above results. In fact, if $\epsilon\to \infty$ \dmc{you mean $\eps \to 0$?} we don't expect a similar result to hold, in the sense that the leading constant of the cover time will depend on $\eps$ and will probably go to 0. Our proofs shed some light onto the behavior of $\tau_{cov}$ right at $r_c$, and we comment on this in the concluding remarks (we do not know the exact behavior, though).}
       Theorems~\ref{ref:main1} and~\ref{ref:main2}, together with known results for $d \ge 3$, show that for $\cG_n$, $\tau_{cov}$ undergoes a jump at $r_c$, decreasing from $\Theta(n \log^2 n)$ to $\Theta(n \log n)$. A more detailed look into this jump is part of future work (see also concluding remarks below). 
        %\item \cmc{Maybe mention that the %order is the same as Bernoulli percolation when r is fixed?}\dmc{Please not oversell this, because this is not really a selling point. I'd not put it here. }
   % \end{enumerate}
\end{remark}

\subsection{Further related work} Apart from the two aforementioned papers on the cover time of RGGs, the cover time is a classic parameter in the study of random walks that has been analyzed in many contexts and on different graphs: in a geometric setup that is  the most similar to the current work, the cover time of the $d$-dimensional toroidal grid of side length $n$ (and thus with $n^d$ vertices) was studied in a series of papers: first, for $d=2$, an upper bound of the cover time of the correct order was first found by Aldous~\cite{Aldous89}, and a non-matching first order lower bound by Lawler~\cite{Lawler92}. Finally, the correct first order was obtained in the landmark paper of Dembo, Peres, Rosen and
Zeitouni~\cite{DPRZ04}: they showed that the cover time of the $2$-dimensional torus satisfies
$$
\tau_{cov} / (n^2 \log^2 n ) \to 4/\pi,
$$
as $n \to \infty$, where the convergence is in probability. Next, Ding~\cite{Ding} improved the result by showing that $\sqrt{\tau_{cov}/n^2} - \frac{2}{\sqrt{\pi}}\log n$ is of order $\log \log n $ with probability tending to $1$ as $n \to \infty$. Abe~\cite{Abe2} then found the second order term: he showed that there exists $c \in (0,1)$, such that
$$
2 \log n - \log \log n - (\log \log n)^c \le \frac{\tau_{cov}}{\frac{2}{\pi}n^2 \log n} \le 2 \log n - \log \log n + (\log \log n)^c
$$ with probability tending to $1$ as $n \to \infty$, similar to a previous bound in the continuous setup of Brownian motion by Belius and Kistler~\cite{BeliusKistler}. 
Recently, using a comparison with the extremal landscape of the discrete
Gaussian free field, Louidor and Saglietti~\cite{LS24} went even further: they showed that the square root of the cover time normalized by the size of the subset is tight around $\frac{1}{\sqrt{\pi}}\log n-\frac{1}{4\sqrt{\pi}}\log \log n$. Very recently, the same authors~\cite{LS26} improved this and found the scaling limit. The latter approach was inspired by an analysis of the cover time on the binary tree of depth $n$: in this case, Cortines, Louidor and Saglietti showed in~\cite{CLS21} that the cover time, normalized by $2^{n+1}n$, and then centered by $(\log 2)n-\log n$, admits a weak limit as $n \to \infty$. The limiting distribution in this case obtained was a Gumbel random variable with rate one, shifted randomly
by the logarithm of the sum of the limits of the derivative
martingales associated with two negatively correlated discrete
Gaussian free fields on the infinite version of the tree.
For $d \ge 3$, it is known for a long time  that the order of the cover time is $\Theta(n^d \log n)$ (see for example~\cite{AF}). Belius in~\cite{Belius} then showed that, when normalized by the volume of the torus and then centered by a constant times $\log n^d$, the cover time converges in law to a Gumbel distribution.

In the context of a percolated grid (of side length $n$), when edges are independently present with probability $p < 1$, Abe in~\cite{aberesistance} showed that for any $d \ge 2$, the cover time of the giant component of supercritical Bernoulli percolation restricted to a grid of side length $n$ in $\mathbb{Z}^d$ is a.a.s.\ of the order $n^d \log^2 n$, %\dmc{the referee was right here. it is $n^d$} 
thereby showing a quantitative difference of unpercolated and percolated grid for $d\ge 3$.

 In the context of random hyperbolic graphs, recently, Kiwi, Schepers and Sylvester studied the cover time of the giant component of random hyperbolic graphs on $n$ vertices. They showed that the cover time of the giant component of these graphs is $\Theta(n \log^2 n)$ (see~\cite{Kiwi}).
 
In another context, on the random graph $\mathcal{G}(n,p)$, a series of papers of Cooper and Frieze analyzed the cover time of the largest component of $\mathcal{G}(n,p)$ for different values of $p$. Frieze, Pegden and Tkocz (see~\cite{FPT22}) obtained more precise estimates for the cover time of the emerging giant component, improving on previous results of Barlow, Ding, Nachmias and Peres~\cite{BDNP} and Cooper, Frieze, Lubetzky~\cite{CFL}.

Finally, for general graphs, Feige showed in~\cite{feige1995lower, feige1995upper} that for any connected graph on $n$ vertices, the cover time is a.a.s.\ at least $(1-o(1)) n \log n $ and at most $(4/27)n^3+o(n^3)$, and he also gave graphs attaining the asymptotic bounds. 
%\cm{In the same paper, they related the cover time to the maximum of the Gaussian Free Field,...}
%\cmc{I would like to write a theorem saying something about the maximum of the GFF}\dmc{mentioning is fine. But why writing a theorem? Do we have extra results or you just use the known link between max and cover time?}\cmc{It would be only a direct application. We can mention it or add as a corollary.} 
From an algorithmic point of view, for general graphs, Kahn, Kim, Lóvasz and Vu in~\cite{kahn2000cover} first established a $(\log \log n)^2$-polynomial time approximation algorithm for the cover time. Later this was improved by Ding, Lee and Peres, who in a breakthrough result in~\cite{DLP12} found a polynomial-time approximation algorithm approximating the cover time up to a constant factor (later this result was refined by Ding~\cite{DingAlg} and Zhai~\cite{ZhaiAlg}).

\subsection{Outline of proofs}
The approach we follow is based on the idea of seeing the graph as an electric network. The main focus of the paper is therefore to obtain sharp bounds on the effective resistance.

The lower bound of the proof of Theorem~\ref{ref:main1} is comparably simpler: when $(1+\eps)r_g\le r\le cr_c$ for some $c$ small enough, we prove that there are polynomially many pending paths of order $\log n /r^d$ (as in the paper of Abe~\cite{aberesistance}). In the case $cr_c\le r\le  (1-\eps)r_c$, we will prove that there are many vertices with degree $1$ that belong to the giant component. 
 The upper bound is inspired by Abe's result in~\cite{aberesistance} on $\tau_{cov}$ of the supercritical percolation cluster on $\mathbb{Z}^d$: we first need to adapt his backbone (which can be understood as a large chunk of $L_1$ that is almost everywhere) construction by using a suitable renormalization scheme on the random geometric graph for fixed radius $r$. We then use this structure to build a flow. Extra care is needed, as there are additional obstacles that do not appear in Bernoulli percolation: crossing may be arbitrarily close to one another, and we may have to extract induced paths from paths in densely populated areas. When the typical degree is tending to infinity, one extra splitting of the flow is introduced in order to bound the effective resistance between any two vertices of a certain backbone by $O((\log n) / r^{2d})$. 
  %\dmc{Maybe take out: "We remark that if we were interested in showing Theorem~\ref{ref:main1} only for $r=\Theta(1)$ and $r \ge C r_g$ for $C$ being a large constant, the theorem would directly follow as a direct application of~\cite{aberesistance}, using a tessellation of cubes of side length equal to $r/\sqrt{5}$ (for $d=2$, and changing the constant for other dimensions) and a coupling with Bernoulli percolation. However, since we show it  for $r \ge (1+\eps) r_g$, we need the already mentioned renormalization scheme for the upper bound."}\cmc{I agree with this}
 
 The proof of Theorem~\ref{ref:main2} elaborates on the idea of the upper bound of Theorem~\ref{ref:main1} by an additional ingredient: on the one hand, since the typical degree is high, we need again the extra splitting from the case when the radius tends to infinity. On the other hand, we then show that other vertices outside this backbone cannot be separated from the backbone by taking out too few vertices: we will find many short internally vertex-disjoint paths between any vertex outside and some vertex inside the backbone, on which we can split the flow again evenly.
 
\textbf{Further notation.} For $\mathrm{j}\in \mathbb{Z}^d$, we denote by $\mathrm{j}_i$ the $i$-th coordinate. We denote by $e_i$ the $i$-th standard basis vector in $d$ dimensions. For a set $A$, we denote its cardinality by $|A|$. For $G=(V,E)$ and $x,y \in V$, we write $x \sim y$ for $\{x,y\} \in E$. We denote by $d_G(x,y)$ the graph distance between $x$ and $y$, that is, the number of edges on a shortest path between $x$ and $y$. If $x$ and $y$ are identified with points, we also denote by $d_E(x,y)$ the Euclidean distance between $x$ and $y$. We say that two paths of vertices $u_0:=x, u_1,\ldots, u_m, u_{m+1}:=y$ and $v_0:=x, v_1, \ldots, v_n, v_{n+1}:=y$  between $x$ and $y$ are internally vertex-disjoint, if $\{u_0,\ldots,u_{m+1}\} \cap \{v_0,\ldots,v_{n+1}\}=\{x,y\}$. A path of vertices $v_1,\ldots, v_k$ in $G=(V,E)$ is induced, if $\{v_i,v_j\} \in E $ implies $j=i+1$ for some $i \in \{1,\ldots, k-1\}$. Note that in an induced path of $\cG_n$ for every vertex of the path except for the first and last there are exactly two other vertices in the path at Euclidean distance at most $r$, and for the first and last vertex there is only one other vertex in the path at Euclidean distance at most $r$. Observe also, by taking out vertices if needed, one may obtain an induced path from a path that is not, and that has the same endpoints as the original path. Finally, in order to make notation more concise, we associate $\cG_n$ with its vertex set, so that we will write $x \in \cG_n$ instead of $x \in V(\cG_n)$.

\textbf{Plan of paper}. In Section~\ref{sec:Prelim} we recall basic facts on random walks, in particular on the effective resistance, as well as facts on random geometric graphs and Bernoulli percolation. In Section~\ref{sec:renormalization} we then introduce a renormalization scheme that will turn out to be very useful in our proofs, and in Section~\ref{sec:randompaths} we build on this scheme to construct random paths of good cubes. Section~\ref{sec:lower_bound_fixed_radius} is dedicated to the proof of the lower bound of Theorem~\ref{ref:main1}. In Section~\ref{sec:upperboundfixedradius} we then show the corresponding upper bound of Theorem~\ref{ref:main1}  for $r=\Theta(1)$. In Section~\ref{sec:upper_bound_covTime_rgg_nonFixed_r} we extend this to the case $r=\omega(1)$. Finally, in Section~\ref{sec:upper_bound_above_connectivity} we prove Theorem~\ref{ref:main2}, and we mention some open problems in Section~\ref{sec:conclusion}.

\section{Preliminaries}\label{sec:Prelim}
In this section we recall basics on different topics, and we prove preliminary results.

\subsection{Markov chains and random walks}
Fix throughout this section a connected finite graph $G=(V,E)$. A network is a connected graph $G=(V,E)$ together with a set of positive real values $(\mathfrak{r}(e))_{e\in E}$, the \emph{resistances} of the edges. We denote the network simply as $G$ and make no distinction with the underlying graph. We define the conductance $\mathfrak{c}$ of the edge $e$ as $\mathfrak{c}(e)=\frac{1}{\mathfrak{r}(e)}$. Given a network, we define the associated transition probabilities by
\begin{equation}\label{eq:transition_probabilities}
    P(x,y)=\frac{\mathfrak{c}({x,y})}{\mathfrak{c}(x)},
\end{equation}
where $\mathfrak{c}(x)=\sum_{y\sim x} \mathfrak{c}(x,y)$. For $\{x,y\} \in E$, denote by $\overrightarrow{e}=\overrightarrow{xy}$ an oriented edge, oriented from $x$ to $y$, and by $-\overrightarrow{e}=\overrightarrow{yx}$ the oriented edge, oriented from $y$ to $x$. For $x,y \in V$, a \emph{unit flow} from $x$ to $y$ is a function $\theta$ on oriented edges such that $\theta(\overrightarrow{xy})=-\theta(\overrightarrow{yx})$, $\sum_{a, x\sim a}\theta(\overrightarrow{xa})=1$, $\sum_{a, a\sim y}\theta(\overrightarrow{ay})=1$ and
\begin{equation}
    \sum_{a, z\sim a} \theta(\overrightarrow{za}) = 0 \text{   for all $z\in V \setminus \{x,y\}$.}
\end{equation} 
The next definition is central for our investigations:
\begin{definition}\label{def:effectiveresistance}
 The \emph{effective resistance} between $x,y\in V$ is defined as
\begin{equation}
    \mathcal{R}(x\longleftrightarrow y,G)= \inf \{ \sum_{e \in E} \theta^2(e)\mathfrak{r}(e): \text{ $\theta$ is a unit flow from $x$ to $y$}\}.
\end{equation}   
\end{definition}

\begin{remark}\label{rem:effective_resistance}
    \begin{enumerate}
    %\cm{
        %\item Although the definition of a flow uses directed edges, equation \eqref{eq:definition_flow} makes sense by anti-symmetry.
    %}\dmc{sacaria}\cmc{Ok}
    By definition, we obtain upper bounds on the effective resistance by specifying a unit flow between $x$ and $y$. A trivial upper bound for $\mathcal{R}(x\longleftrightarrow y,G)$ is therefore given by the graph distance between the vertices (define the unit flow over the shortest path with all edges having flow equal to $1$).
    \end{enumerate}
    
\end{remark}
 The effective resistance can be interpreted as the value that appears when replacing the entire network with only one edge between $x$ and $y$ that mimics the resistance between them in the original network. A \emph{subnetwork} $G'$ of $G$ is a connected subgraph of $G$ with the same values of resistances on the remaining edges. We now collect a few results that will turn very useful for us. First, effective resistance can not decrease when we delete edges:
\begin{theorem} [Rayleigh, Theorem 9.12 of~\cite{levinperes}]\label{thm:rayleigh}
    Let $G'$ be a subnetwork of $G$. Then, for $x,y \in V(G')$ we have
    \begin{equation}
        \mathcal{R}(x\longleftrightarrow y,G) \leq \mathcal{R} (x\longleftrightarrow y,G').
    \end{equation}
\end{theorem}
When the underlying network $G$ is clear, we simply denote the effective resistance by $\mathcal{R}(x\longleftrightarrow y)$. It is well known that the triangle inequality holds for the effective resistance:
\begin{lemma} [Corollary 10.8 of~\cite{levinperes}] \label{lem:triangleinequality}
    Let $G=(V,E)$ and let $x,y,z \in V$. Then
    \begin{equation}
        \mathcal{R} (x\longleftrightarrow z) \leq \mathcal{R} (x\longleftrightarrow y) + \mathcal{R} (y \longleftrightarrow z).
    \end{equation}
\end{lemma}

For $x,y \in V$, we say that $\Pi \subseteq E$ is an \emph{edge-cutset} separating $x$ from $y$ if every path from $x$ to $y$ contains at least one edge in $\Pi$. In other words, if we remove $\Pi$ from $E$ then $x$ is isolated from $y$. We will use the following Nash-Williams inequality to obtain lower bounds for effective resistances.
\begin{lemma} [Proposition 9.16 of~\cite{levinperes}]\label{lem:Nash-Williamas}
    Let $\{\Pi_k\}_k$ be a pairwise edge-disjoint family of edge-cutsets that separate $x$ from $y$. Then
    \begin{equation}
        \mathcal{R}(x\longleftrightarrow y) \geq \sum_{k}\left(\sum_{e\in \Pi_k} \mathfrak{c}(e) \right)^{-1}.
    \end{equation}
\end{lemma}
The simple random walk in $G$ admits a natural network representation when the set of resistances is given by $\mathfrak{r}(e)=1$ for all $e\in E$. The following bounds on $\tau_{cov}$ in terms of the effective resistance will turn out to be useful for us:
\begin{theorem}[Theorem 2.4 of \cite{chandra1989electrical}] \label{thm:upperboundcov_chandra}
    Let $G=(V,E)$ be a connected graph. Then there exists a universal constant $\gamma_1>0$ such that   
    \begin{equation}
    \tau_{cov} \leq \gamma_1|E|  \log(|V|) \left(\max_{x,y \in V} \mathcal{R}(x\longleftrightarrow y) \right).
    \end{equation}
\end{theorem}
\begin{theorem}[Theorem 1.3 of \cite{kahn2000cover}]\label{thm:lowerboundcov}
    Let $G=(V,E)$ be a connected graph. Given $V'\subseteq V$, there exists a universal constant $\gamma_2 > 0$ such that
    \begin{equation}
        \tau_{cov}\geq \gamma_2|E|  \log(|V'|) \left(\min_{\substack{x \neq y \\ x, y \in V'}} \mathcal{R}(x \longleftrightarrow y)\right).
    \end{equation}
\end{theorem}

%\dmc{para que este remark? Lo sacaria. Por ahora en comentarios}\cmc{Ok}
%\begin{remark}
 %   \cm{The bounds given in Theorems \ref{thm:upperboundcov_chandra} and \ref{thm:lowerboundcov} do not match in general.}
%\end{remark}

\subsection{Random geometric graphs}
We next state and prove useful facts about random geometric graphs. 
The first lemma is an upper bound on the number of edges in each parallelepiped.

\begin{lemma}\label{lem:edgedeviation}
    Let $d\ge 2$, and let $\cG_n$ %\dmc{please correctly put the macro throughout}\cmc{Fixed, thanks} 
    with $r=\Theta(1)$. A.a.s., every axis-parallelepiped with integer coordinates (inside $\Lambda_n$) of side lengths $Mr \log n$ for some constant $M > 0$ in 2 dimensions and $Mr$ in the remaining $d-2$ dimensions contains $O(\log^2 n)$ vertices and has $O(\log^2 n)$ edges with at least one endpoint in it. 
\end{lemma}
\begin{proof}
First note that by standard Chernoff bounds for Poisson random variables, together with a union bound over all $O(n)$ parallelepipeds, each such parallelepiped contains a.a.s. $O(\log^2 n)$ vertices, showing the first part of the lemma. Now, for the second part, tessellate such parallelepiped into axisparallel boxes of sidelength $2r$. Tessellate it first so that the left upper corner is aligned with the tessellation, and include also one extra row and one column possibly outside the parallelepiped. Note that there are $C \log^2 n$ such boxes for some $C > 0$. Let $i_0$ be a large constant. For $i_0 \le i \le C_1\log n$ with $C_1$ large enough, the probability that in a fixed box the number of vertices is within $[2^i, 2^{i+1})$ is at most
$$
p_i := e^{- (2^i-1) (2r)^d/3},
$$
where we used the fact that $\mathbb{P}(\mathrm{Po}(\lambda)\ge k) \le e^{-(k-1)\lambda/3}$ for $k$ large enough. Since the number of vertices in different boxes is independent, the number of boxes having cells with a number of vertices within $[2^i, 2^{i+1})$ is dominated by a binomial distribution with parameters $C\log^2 n$ and $p_i$.  We distinguish two cases: if $2^i < C_1 \log \log n$, then the probability to have at least $\log^2 n / 2^{3i}$ many boxes whose number of vertices is within $[2^i, 2^{i+1})$ is at most
$$
\binom{C \log^2 n}{\log^2 n / 2^{3i}}p_i^{\log^2 n / 2^{3i}} \le (Ce 2^{3i} p_i)^{\log^2 n / 2^{3i}} \le n^{-2}.
$$
On the other hand, if $2^i \ge C_1\log \log n$ for $C_1$ large enough, then, by Chernoff bounds, the probability to have at least $C_2 \log n/ 2^i$ (for $C_2$ large) boxes whose number of vertices is within $[2^i, 2^{i+1})$ is at most
$$
\binom{C\log^2 n}{C_2 \log n/2^i}p_i^{C_2\log n/2^i} \le (C_3 e(\log n) p_i 2^i)^{C_2 \log n/2^i} \le n^{-2}.
$$
Now, assuming the worst case that all pairs of vertices inside the same cell are incident to each other, by combining both cases, with probability at least say $1-n^{-3/2}$, 
the total number of edges inside cells is at most
$$
O(\log^2 n)+\sum_{i, 2^i < C_1 \log \log n } (\log^2 n / 2^{3i}) (2^{i+1})^2 +\sum_{i, 2^i \ge C_1 \log \log n, 2^i \le C_3 \log n} (C_2 \log n / 2^i ) (2^{i+1})^2 = O(\log^2 n),
$$
where we chose $C_3$ large enough so that with probability at most $n^{-2}$ there exists a cell with more than $C_3 \log n$ many vertices, and where the sum is only over all $i$ that are powers of $2$.
Hence, with probability at least $1-n^{-3/2}$, the number of edges contained inside some cell of the given tessellation is $O(\log^2 n)$. We may shift the tessellation in each coordinate by $r$ (taking all $2^d$ combinations of shifts or not shifts in each coordinate), and by a union bound, we see that in all tessellations the total number of edges, is still with probability $1-O(n^{-3/2})$, $O(\log^2 n)$. Observe that any edge of length $r$ must be fully contained in at least one cell of one tessellation, and hence the statement holds for a fixed parallelepiped. Taking a union bound over all $O(n)$ parallelepipeds, the lemma follows.
\end{proof}

The following result will be used several times in the sequel. Since we were unable to find a reference, we give a proof for completeness.  Recall the definition of $r_g$ in \eqref{eq:definition_RG}.

\begin{lemma}\label{lem:sizergg}
    Consider $\cG_n^{r,d}$ with $d\ge 2$ and $r \ge (1+\eps) r_g$ for arbitrarily small $\eps > 0$. A.a.s., the number of edges of $L_1$ is $\Theta(n r^d)$.
\end{lemma}
%\cmc{I think it is worth to write the full proof in detail.}\dmc{I wrote a bit more. Feel free to add if you wish}
\begin{proof}
The upper bound on the number of edges follows trivially, as a.a.s.\ the total number of edges in $\cG_n$ is $\Theta(n r^d)$: %\dmc{Ok, I put more. Feel free to extend if it is not clear still}
indeed, consider a tessellation of $\Lambda_n$ into boxes of side length $2r$ together with all possible combinations of respective shifts of this tessellation by $r$ in each dimension (so there are $2^d$ such tessellations). It is clear that the number of boxes is $\Theta(n/r^d)$ for every shift (boundary issues do not matter here). For one of such shifts, the number of edges with both endvertices inside a box is bounded from above by $X^2$, where $X \sim \mathrm{Poi}(r^{2d})$. Since different boxes are independent, by standard concentration results, a.a.s. the total number of edges in one tessellation is $O(nr^{d})$. The desired result follows by a union bound over all shifts, and the fact that any edge has to be completely contained in one box of one tessellation.

For the lower bound, first sprinkle a Poisson point process of intensity $1-\xi$ where $\xi > 0$ is small enough so that the resulting graph has a giant component of size $\Theta(n)$ a.a.s. (see for example Theorem 10.9 of~\cite{Pen03}). Now, each  vertex of this giant has in expectation $\Theta(\xi r^d)$ neighbors among vertices of the second PPP of intensity $\xi$, and so the total number of edges emanating from the giant component of the first PPP to the second PPP is in expectation $\Theta(\xi n r^d)$. Since we may assume $r \le r_c$, and the giant component a.a.s.\ contains in each ball of radius $r$ at most $O(\log n)$ many vertices, the number of edges emanating from one fixed vertex of the first PPP is also at most $O(\log n)$. Since the total number of edges added is $\Omega(n)$, by a standard second moment method, concentration follows.  
\end{proof}
We will make use of the following proposition that relates the Euclidean distance $d_E$ of two vertices of $L_1$ to the graph distance $d_{G}$.
 By Theorem 3 and Remark 5 of \cite{friedrich2013diameter}, we have:
\begin{theorem}\label{thm:chemical_distance_friedrich}
    Consider $\cG_n^{r,d}$ with $d\geq 2$ and $r_g < r$. A.a.s.\ the following is true: there exist positive constants $C=C(r,d)$ and $C'=C'(r,d)$, such that for any two vertices $x,y \in L_1$ with $d_E(x,y)\geq C \frac{\log n}{r^{d-1}}$, we have $d_{G}(x,y)\leq C'd_E(x,y)/r$.
\end{theorem}
Recall the definition of $r_c$ in \eqref{eq:defintion_RC}. We get the following corollary: 
\begin{corollary}\label{cor:chemical_distance_rgg}
  Consider $\cG_n^{r,d}$ with $d\geq 2$ and $r_g < r =O(r_c)$. Let $V'\subseteq L_1$ be such that for any vertex $x\in L_1$ we have that  $d_E(x,V')< C\log n/r^{d-1}$ for the same constant $C$ as in Theorem~\ref{thm:chemical_distance_friedrich}. Then, a.a.s., there exists a constant $c'>0$ such that $d_{G}(x,V')\leq c'\log n/r^d$ for any vertex $x\in L_1$.
\end{corollary}

\begin{proof}
%\dmc{New proof. I didn't understand the old proof. If you have a simpler proof please go ahead. But at least it seems correct now}
By hypothesis, there exists $y \in V'$ with $d_E(x,y) < C \log n/r^{d-1}$. If $d_G(x,y) \le c'\log n/r^d$, we are done, so suppose not. Since $x \in L_1$, a.a.s., there exists $v \in L_1$ with $d_E(x,v) \ge n^{\eps}$ for some $\eps > 0$ (since $|L_1|=\Theta(n)$ and inside each region of area $n^\eps r^d$ there are a.a.s.\ at most $O(n^{\eps}r^d)$ many vertices, such $v$ must exist a.a.s.). We claim that therefore exists $z \in L_1$ such that $d_E(x,z) \ge C \log n /r^{d-1}$ and $d_E(x,z) =O(\log n / r^{d-1})$: indeed, consider a shortest path $x=v_0, v_1, \ldots, v_k=v$ for some $k \in \mathbb{N}$: since the Euclidean distance between two consecutive vertices is at most $r$, there exist vertices $v_i$ and $v_{i+1}$ on this path such that $d_E(x,v_i) < 3C\log n/r^{d-1}$ and $d_E(x, v_{i+1}) \ge 3C\log n/r^{d-1}$. Also, at the same time $d_E(x,v_{i+1}) \le 3C\log n/r^{d-1}+r =O(\log n/r^{d-1})$, where the equality holds from our assumption $r=O(r_c)$. Hence, $v_{i+1}$ is the desired vertex $z \in L_1$.
Therefore, by Theorem~\ref{thm:chemical_distance_friedrich}, $d_G(x,v_{i+1}) \le C'd_E(x,v_{i+1})/r \le C''\log n/r^d$ for some $C',C'' > 0$. If $v_{i+1}\in V'$, we are done. Otherwise, note first that by the triangle inequality, $d_E(v_{i+1},y)=O(\log n/r^{d-1})$. Next, we must also have $d_E(v_{i+1},y)\ge C\log n/r^{d-1}$: indeed, if not, we would conclude $d_E(x,y) \ge  d_E(x,v_{i+1}) - d_E(v_{i+1},y) \ge 3C\log n/r^{d-1} - C\log n/r^{d-1}=2C\log n/r^{d-1}$, contradicting the assumption. But 
$d_E(v_{i+1},y)\ge C\log n/r^{d-1}$ implies that by Theorem~\ref{thm:chemical_distance_friedrich}, a.a.s.\ $d_G(v_{i+1},y) \le c'\log n/r^d$ for some $c' > 0$. Once again, by the triangle inequality, $d_G(x,y) \le d_G(x,v_{i+1})+d_G(v_{i+1},y) \le 2c'\log n$, and we are done.
\end{proof}

\begin{comment}
\dmc{This is the old proof.  I left it for now, but I don't know why in the original part $v_{i+1} $ was in $V'$}
On the one hand, for every $x \in L_1$ there exists $y \in V'$ with $d_E(x,y) < C \log n/r^{d-1}$, but on the other hand, since $x \in L_1$, a.a.s., there exists $v \in L_1$ with $d_E(x,v) \ge n^{\eps}$ for some $\eps > 0$ (since $|L_1|=\Theta(n)$ and inside each region of area $n^\eps r^d$ there are a.a.s.\ at most $O(n^{\eps}r^d)$ many vertices, such $v$ must exist a.a.s.). We claim that therefore exists $z \in L_1$ such that $d_E(x,z) \ge C \log n /r^{d-1}$ and $d_E(x,z) =O(\log n / r^{d-1})$: indeed, consider a shortest path $x=v_0, v_1, \ldots, v_k=v$ for some $k \in \mathbb{N}$: since the Euclidean distance between two consecutive vertices is at most $r$, there exist vertices $v_i$ and $v_{i+1}$ on this path such that $d_E(x,v_i) < C\log n/r^{d-1}$ and $d_E(x, v_{i+1}) \ge C\log n/r^{d-1}$. Also, at the same time $d_E(x,v_{i+1}) \le C\log n/r^{d-1}+r =O(\log n/r^{d-1})$, where the equality holds from our assumption $r=O(r_c)$. Hence, $v_{i+1}$ is the desired vertex $z \in L_1$. Therefore, by Theorem~\ref{thm:chemical_distance_friedrich}, a.a.s.\ $d_G(y,v_{i+1}) \le C'd_E(x,v_{i+1})/r \le C''\log n/r^d$ for some $C',C'' > 0$. Since $v_{i+1}\in V'$, the statement follows. \dmc{Here we claimed $v_{i+1} \in V'$, but it is not direct to me. I didn't see this}
\end{comment}
Next, we also use the following result (see for example~\cite{Pen03}, Theorem 10.18):
\begin{theorem}\label{thm:giantcomponents}
    Let $d\geq 2$, let $r\ge (1+\eps) r_g$ for some $\eps >0$, and consider $\cG_n^{r,d}$. A.a.s., there exists $c > 0$ such that the second largest connected component of $\cG_n$ is of size at most $c \log^2 n$.
\end{theorem}
\subsection{Bernoulli percolation}
We also need results from Bernoulli percolation. Consider an i.i.d.\ family of independent Bernoulli variables with mean $p$, indexed by the edges of $\mathbb{Z}^d$, called bonds. We say that the bond $e$ is open if its associated variable is equal to 1. For a given realization, we generate a subgraph of $\mathbb{Z}^d$ with edges given by the open bonds, this random graph model is known as Bernoulli bond percolation. For $d\geq 2$, it is known that this model undergoes a non-trivial phase transition: there exists $0 < p_c^{\mathrm{bond}} < 1$ such that if $p>p_c^{\mathrm{bond}}$ then a.s.\ there is a unique (maximal) infinite connected component $\mathcal{C}$ (see for example~\cite{bollobaspercolation}), which we call the infinite cluster. 
If one considers variables indexed by the vertices of $\mathbb{Z}^d$, called a site, instead of bonds we have Bernoulli site percolation. We say that a site is open if its associated variable is 1. We say that a path of sites is open if all its sites are open. As in Bernoulli bond percolation, for $d\geq 2$ there exists $0 < p_c^{\mathrm{site}} < 1$ such that if $p>p_c^{\mathrm{site}}$ then there is a.s.\ a unique infinite cluster.

We denote by $\mathbb{P}^{\mathrm{site}}_p$ the law of Bernoulli site percolation with parameter $p$, and by $\mathbb{P}^{\mathrm{bond}}_p$ the law of Bernoulli bond percolation with parameter $p$, respectively. We denote by $\mathcal{C}^{\mathrm{bond}}_i(n)$ the $i$-th largest connected component of Bernoulli bond percolation on $\mathbb{Z}^d$ restricted to the box $\left[-n^{1/d}/2,n^{1/d}/2\right]^{d}$. Similarly, we denote by $\mathcal{C}_i^{\mathrm{site}}(n)$ the $i$-th largest connected component of Bernoulli site percolation on the same box. The following result is well known (for instance, see Theorems 1.2 and 3.1 of \cite{pisztora1996surface} for the result in the framework of bond percolation; the modification to site percolation is straightforward): 
\begin{proposition}\label{prop:largest_component_bernoulli_percolation}
    Consider Bernoulli site percolation on $\mathbb{Z}^d$ with $d\ge 2$, and let  $p\in (p_c^{\mathrm{site}},1]$. Then, a.a.s.
    $$
    |\mathcal{C}_1^{\mathrm{site}}(n)| =\Theta(n) \mbox{ and } |\mathcal{C}_2^{\mathrm{site}}(n)| = O(\log^\delta n),
    $$
    for some $\delta>0$ that depends only on $d,p$.
\end{proposition}

\begin{comment}
We now identify $\mathcal{C}_1^{\mathrm{bond}}(n)$ with a network where all edges have unit resistance. The following estimate for the effective resistance of the resulting network will be useful for us:  \cmc{I didn't find any reference of this result later. Maybe remove it?}

\begin{theorem} [Theorem 1.1 of~\cite{aberesistance}]
    Consider Bernoulli bond percolation on $\mathbb{Z}^d$ with $d \ge 2$, and let $p\in (p_c^{\mathrm{bond}},1)$. There exist constants $c_1,c_2>0$ such that a.a.s.,
    \begin{equation}
        c_1\log n \leq \max_{x,y \in \mathcal{C}_1^{\mathrm{bond}}(n)} \mathcal{R}\left(x\longleftrightarrow y\right) \leq c_2\log n.
    \end{equation}
\end{theorem}
\end{comment}

We also make use of the following percolation estimations of crossings in boxes for 2-dimensional Bernoulli site percolation. Given $[0, m]\times [0, n] \subseteq \mathbb{Z}^2$, a horizontal crossing is an open self-avoiding path entirely contained within the box such that the first vertex of the path has its $x$-coordinate equal to $0$ and the final vertex has its $x$-coordinate equal to $m$. We define a vertical crossing in a similar way. A \emph{long} crossing of a rectangle $[a, b]\times [c, d]$ is an open self-avoiding path along the longer dimension of the rectangle (ties broken arbitrarily). We will use the following result:

\begin{lemma} [Theorem 11.1 of~\cite{kesten}]\label{lem:crossingpercolation}
    Consider Bernoulli site percolation on $\mathbb{Z}^2$ with parameter $p>p_c^{\mathrm{site}}$. There exist positive constants $c_1,c_2,c_3$ such that
    \begin{equation}
    \begin{split}
        \mathbb{P}^{\mathrm{site}}_p\big(&\text{There exist at least $c_1n$ disjoint horizontal crossings} 
        \\ &\text{ in $([0, m]\times [0, n])\cap \mathbb{Z}^2$}\big)\geq 1-c_2me^{-c_3n}.
    \end{split}
    \end{equation}
\end{lemma}
Finally, we need a result relating the graph distance to the Euclidean distance between two vertices of the largest connected component in Bernoulli site percolation (see Lemma 2.3 of ~\cite{aberesistance}) on $\mathbb{Z}^d$ restricted to $[-n^{1/d}/2, n^{1/d}/2]^d$:
\begin{theorem}\label{them:chemical_distance_percolation}
 Consider Bernoulli site percolation on $\mathbb{Z}^d$ with $d\ge 2$, and let $p> p_c^{\mathrm{site}}$. There exist constants $\kappa_0, c$ such that a.a.s.\ the following holds: for any $x,y \in \mathcal{C}_1^{\mathrm{site}}(n)$, and any $\kappa >\kappa_0$, if $d_E(x,y) \leq \kappa \log n$, then
$d_{G}(x,y) \leq c \kappa  \log n.$
\end{theorem}

We say that two sites of $\mathbb Z^d$ are diagonal neighbors if they are at distance $1$ with respect to the $\ell^{\infty}$-norm. A finite set $S\subseteq \mathbb Z^d$ is said to be \emph{diagonally connected} if any two sites can be connected by a path contained in $S$ whose consecutive sites are diagonal neighbors. For a finite connected set $S \subseteq \mathbb{Z}^d$, its \emph{interior boundary} $\partial^{\mathrm{int}}S$ is defined as $\{x \in S: \exists y \notin S, x \text{ adjacent to } y\}$, and similarly, its \emph{external boundary} $\partial^{\mathrm{int}}S$ is defined as $\{x \notin S: \exists y in S, x \text{ adjacent to } y\}$.

\begin{lemma}[Lemma 2 of \cite{timar2013boundary}]\label{lem_Timar}
    Let $S$ be a finite connected subset of $\mathbb Z^d$. Then, the external boundary of $S$ is diagonally connected. In particular, any finite connected subset of $\mathbb Z^d$ has a diagonally connected external $\mathbb Z^d$-boundary,
and if $S \subseteq \mathbb Z^d \cap [-n^{1/d}/2,n^{1/d}/2]^d$, the outer $\mathbb Z^d$-boundary of $S$ in any component of $\mathbb Z^d \cap [-n^{1/d}/2,n^{1/d}/2]^d \setminus S$ is diagonally connected. Equivalently, for every maximally connected component $S' \subseteq  (\mathbb{Z}^d \cap [-n^{1/d}/2,n^{1/d}/2]^d) \setminus S$, the set $\partial^{\mathrm{int}}S' \cap [-n^{1/d}/2,n^{1/d}/2]^d$ is diagonally connected. 
\end{lemma}

\begin{comment}
In fact, we need its extension to finite subsets. Denoting by $\partial^{\mathrm{int}}S$ the internal vertex boundary of a subset $S \subseteq \mathbb{Z}^d$, we have:
\begin{lemma}\label{lem_Timar2}
    Let $S$ be a connected set inside of $[-n^{1/d}/2,n^{1/d}/2]^d$. Then, for every maximally connected component $S' \subseteq  [-n^{1/d}/2,n^{1/d}/2]^d \setminus S$, the set $\partial^{\mathrm{int}}S' \cap [-n^{1/d}/2,n^{1/d}/2]^d$ is diagonally connected. 
\end{lemma}
\begin{proof}
    First, note that $[-n^{1/d}/2,n^{1/d}/2]^d\setminus S'$ must be connected. Indeed, $S\subseteq [-n^{1/d}/2,n^{1/d}/2]^d\setminus S'$, and if $S''$ is another maximally connected component of $ [-n^{1/d}/2,n^{1/d}/2]^d \setminus S$, then $S$ must be connected to $S''$ (otherwise, $S''$ is not maximal). Then, we may apply Lemma~\ref{lem_Timar} to $[-n^{1/d}/2,n^{1/d}/2]^d\setminus S'$, which gives the desired result.  
\end{proof}
\end{comment}
\section{Renormalization scheme}\label{sec:renormalization}
In this section we explain the renormalization scheme that we use in order to obtain the upper bound of Theorem~\ref{ref:main1}. The objective of the renormalization is, similar to the one used by Abe in ~\cite{aberesistance}: to prove that there is a backbone of crossings of $L_1$. In our setup we encounter several technical challenges proper of random geometric graphs that we have to consider. From now on, we will use the term vertices exclusively for vertices of $\mathcal{G}_n$. Fix throughout this section $M > 0$ large enough. For each site $\mathrm{i} \in \mathbb{Z}^d$ we define the cube %\dmc{maybe say why we center at $-Mr/2, Mr/2$? Or at least say it will become clear later}\cmc{I think there is not exactly a good reason, it could be $+[0,M]^d$, this would need some modifications on other cubes.}
\begin{equation}
    \mathrm{C_i}=Mr\mathrm{i}+[-Mr/2, Mr/2)^{d}.
\end{equation}
In words, $\mathrm{C_i}$ is an axis-parallel cube with center $Mr\mathrm{i}$ and side length $Mr$. We say it has \emph{index} $\mathrm{i}$. We say that two cubes $\mathrm{C_i}$ and $\mathrm{C_j}$ are neighbors if $\mathrm{i}\sim \mathrm{j}$ (on the renormalized lattice resulting from tessellation of $\mathbb{R}^d$ into cubes of sidelength $Mr$). Denote by $\mathcal{K}_n$ the set of cubes that are totally contained within $\Lambda_n= \left[ -\frac{n^{1/d}}{2},\frac{n^{1/d}}{2}\right]^{d}$. 
%The choice of $M$ depends on the section: in Section~\ref{sec:lower_bound_fixed_radius} the choice of $M$ depends on another renormalization scheme that is introduced there. In the other sections, we may assume that by an appropriate choice $M$, all cubes $\mathrm{C_i}$ are either completely contained in $\Lambda_n$ or they share at most some faces with $\Lambda_n$.

Now, let $\mu=\mu(d)$ be the largest constant so that any two vertices of $\mathcal{G}_n$ in adjacent cubes, that is, sharing a $(d-1)-$dimensional face, of side length $r/\mu$ are connected by an edge (for example, in $d=2$, $\mu=\sqrt{5}$). We will consider such a tessellation of side length $r/\mu$ later on as well, for now we use it to define crossings.
For $a_1 < a_2$, $b_1 < b_2$, let $U=[a_1,a_2]\times [b_1,b_2]\times U_{d-2}$ be a parallelepiped in $\mathbb{R}^d$, where $U_{d-2}$ a parallelepiped in $\mathbb{R}^{d-2}$.
We say that the parallelepiped $U$ has a \emph{crossing component}, if there exists a connected component of $\mathcal G_n$ within the parallelepiped that contains vertices at Euclidean distance at most $\frac{r}{\mu}$ from each of the $2d$ faces of the parallelepiped. Similarly as we did for $\Z^d$, we define crossings. A \emph{horizontal crossing of $U$} is a self-avoiding path
\[
w=(w_0,w_1,\ldots,w_m), 
\]
of vertices of $\cG_n$ such that $w_j\in U$ for all $0\le j\le m$, such that the first vertex satisfies $(w_0)_1\in (a_1,a_1+r/\mu)$, and the last vertex satisfies $(w_m)_1\in (b_1-r/\mu,b_1)$, where $(z)_i$ denotes the $i$-th coordinate of $z\in\R^d$.
 A vertical crossing is defined analogously, exchanging the first coordinate by the second coordinate and $a_1,a_2$ by $b_1,b_2$, respectively. As before, we use also the definition of a \emph{long crossing} for a crossing  in the longer dimension, that is, whose length is $\max\{a_2-a_1, b_2-b_1\}$. We remark that while a horizontal (or vertical) crossing is a self-avoiding path, a crossing component is in general not a path (but it contains such a path by definition). Given a connected component $G'\subseteq \mathcal{G}_n$, the diameter of $G'$ is defined as the maximum Euclidean distance between any two vertices of $G'$.

We say that the cube $\mathrm{C_i} \in \mathcal{K}_n$ is \textit{good} if the following conditions hold (see Figure \ref{fig:goodcube}):
\begin{enumerate}
    %\item $\mathrm{C_i}$ contains a crossing component.
    \item $\mathrm{C_i}\cup \mathrm{C_j}$ contains a crossing component for each neighbor $\mathrm{j}\sim \mathrm{i}$ if $\mathrm{C_j}\in\mathcal{K}_n$.
    \item $\mathrm{C_i}$ has only one connected component with diameter larger than $Mr/5$.
\end{enumerate}

%\cm{
%\begin{remark}

    %Item (a) of the above definition is added to avoid problems with boundary cubes, since if $\mathrm{C_i}$ belongs to the boundary of $\mathcal{K}_n$, conditions (b) and (c) do not necessarily imply condition (a). 
  %  \dmc{No entiendo eso. En figura 1 por ejemplo supongamos que existen solo los cubos que hacen el azul y el verde. Pero entonces automaticamente $C_i$ tambien contiene un crossing component (interseccion de azul y verde en $C_i$). Sacaria item a. Y pondria el remark que eso implica que automaticamente que $C_i$ tiene una crossing component}\cmc{Tienes razon, ya lo borre. No escribiria un remark tampoco, creo que esta claro.}
%\end{remark}
%}

\begin{figure}[htbp!]
\centering
\begin{minipage}{0.48\textwidth}
    \centering
    \begin{tikzpicture}
        % Size of each square
        \def\squaresize{2}
        
        % Central square
        \draw[thick] (0, 0) rectangle (\squaresize, \squaresize);
        \node at (0.2, 1.8) {\footnotesize C$_i$};

        % Surrounding squares
        \draw[thick] (0, \squaresize) rectangle (\squaresize, 2*\squaresize); % top
        \draw[thick] (0, -\squaresize) rectangle (\squaresize, 0); % bottom
        \draw[thick] (-\squaresize, 0) rectangle (0, \squaresize); % left
        \draw[thick] (\squaresize, 0) rectangle (2*\squaresize, \squaresize); % right

        % Smooth, self-avoiding paths
        \draw[red, thick] (-\squaresize, 0.5*\squaresize+0.1) 
        .. controls (-0.5*\squaresize, 0.5*\squaresize) 
        and (0.5*\squaresize, 1.5*\squaresize) .. (\squaresize , 0.3*\squaresize); % Left to top
        
        \draw[blue, thick] (0.1*\squaresize, 2*\squaresize) 
        .. controls (0.9*\squaresize, 0.8*\squaresize) 
        and (0.8*\squaresize, 0.5*\squaresize) .. (0.7*\squaresize, 0); % Top to right
        
        \draw[green, thick] (2*\squaresize, 0.5*\squaresize) 
        .. controls (1.5*\squaresize, 0.5*\squaresize) 
        and (0.5*\squaresize, -0.5*\squaresize) .. (0, 0.5*\squaresize); % Right to bottom
        
        \draw[orange, thick] (0.5*\squaresize, -\squaresize) 
        .. controls (0.5*\squaresize, -0.5*\squaresize) 
        and (0.2*\squaresize, 0.5*\squaresize) .. (0.3*\squaresize,\squaresize); % Bottom to left
    \end{tikzpicture}
    \caption{The central cube \(\mathrm{C_i}\) and its surrounding paths. The colored lines represent vertex paths of \(\cG_n\).}
    \label{fig:goodcube}
\end{minipage}%
\hfill
\begin{minipage}{0.48\textwidth}
    \centering
    \begin{tikzpicture}
        % The big cube \Lambda_n
        \draw[black, thick] (0,0,0) -- (4,0,0) -- (4,4,0) -- (0,4,0) -- cycle; % bottom face
        \draw[black, thick] (0,0,4) -- (4,0,4) -- (4,4,4) -- (0,4,4) -- cycle; % top face
        \draw[black, thick] (0,0,0) -- (0,0,4); % left edge
        \draw[black, thick] (4,0,0) -- (4,0,4); % right edge
        \draw[black, thick] (0,4,0) -- (0,4,4); % back left edge
        \draw[black, thick] (4,4,0) -- (4,4,4); % back right edge

        \draw (1, 2, 4) -- (4, 2, 4);
        \draw (1, 2.2, 4) -- (4, 2.2, 4);
        \draw (4, 2, 4) -- (4, 2, 0);
        \draw (4, 2.2, 4) -- (4, 2.2, 0);

        \draw[red] (0, 2, 4) -- (1, 2, 4);
        \draw[red] (0, 2.2, 4) -- (1, 2.2, 4);
        
        \foreach \x in {5,...,20} {
            \pgfmathsetmacro{\xDiv}{\x/5} 
            \draw (\xDiv, 2, 4) -- (\xDiv, 2.2, 4);
            \draw (\xDiv, 2.2, 4) -- (\xDiv, 2.2, 0);
        }

        \foreach \x in {0,...,5} {
            \pgfmathsetmacro{\xDiv}{\x/5} 
            \draw[red] (\xDiv, 2, 4) -- (\xDiv, 2.2, 4);
            \draw[red] (\xDiv, 2.2, 4) -- (\xDiv, 2.2, 0);
        }

        \foreach \x in {0,...,20} {
            \pgfmathsetmacro{\xDiv}{\x/5} 
            \draw (1,2.2,\xDiv) -- (4,2.2,\xDiv);
            \draw (4,2,\xDiv) -- (4,2.2,\xDiv);
        }

        \foreach \x in {0,...,20} {
            \pgfmathsetmacro{\xDiv}{\x/5} 
            \draw[red] (0,2.2,\xDiv) -- (1,2.2,\xDiv);
        }
        
        \node at (3.7, 0.3, 4) {\(\Lambda_n\)};
        \node at (4.3, 2, 0) {\(\mathcal{K}_n^2\)};
    \end{tikzpicture}
    \caption{A \(2\)-dimensional slice \(\mathcal{K}_n^2\), with a red \(\alpha\)-logarithmic strip.}
    \label{fig:bigcube}
\end{minipage}
\end{figure}
By a straightforward modification of a result of Penrose and Pisztora (Theorem 2 of~\cite{pisztorapenrose})  (in the original definition of Penrose and Pisztora, the crossing component is at distance at most $r$ from all the faces of the parallelepiped, but the proof works in exactly the same way for $r/\mu$) we may state the following lemma:
\begin{lemma}\label{lem:Pisztorapenrose}
    Let $d\geq 2$, let $M$ be a large enough constant, and let $r \ge (1+\eps)r_g$ for arbitrarily small $\eps > 0$. Fix $\mathrm{C_i} \in \mathcal{K}_n$. If $\mathrm{i}\sim \mathrm{j}$, there exists a constant  $\gamma>0$ such that 
    \begin{equation}
    \begin{split}
        \mathbb{P}&(\mathrm{C_i}\cup \mathrm{C_j} \text{ has a crossing component}) \geq 1-e^{-\gamma Mr^d},
        \\ \mathbb{P}&(\mathrm{C_i} \text{ has only one connected component with diameter larger than $Mr/5$})\geq 1-e^{-\gamma Mr^d}.
    \end{split}
    \end{equation}
\end{lemma}

For $\mathrm{C_i} \in \mathcal{K}_n$, define now $U_{\mathrm{i}}$ to be the indicator random variable that the cube $\mathrm{C_i}$ is good.
We have the following important corollary as a straightforward application of Lemma~\ref{lem:Pisztorapenrose}.

\begin{corollary}\label{cor:domination_good_cubes}
    Let $d\ge 2$ and $r\ge (1+\eps)r_g$. For any $0<p<1$, there exists $M=M(p,\eps)$ large enough, such that the good cubes in $\mathcal{K}_n$ stochastically dominate a Bernoulli site percolation process with parameter $p$.
\end{corollary}
\begin{proof}
By Lemma~\ref{lem:Pisztorapenrose}, for $M$ large enough there is a constant $\gamma'>0$ such that for any $\mathrm{C_i}
\in \mathcal{K}_n$ we have $\mathbb{P} (U_{\mathrm{i}}=1) \geq 1- e^{-\gamma' M r^d}$. Note that the indicator events $U_{\mathrm{i}}$ and $U_{\mathrm{j}}$ are independent if $\mathrm{i}$ and $\mathrm{j}$ are at graph distance in $\mathbb{Z}^d$ of at least 3. By Theorem 0.0 of~\cite{liggettdomination}, for $M$ large enough, $(U_{\mathrm{i}})_{\mathrm{i} \in \mathcal{K}_n}$ stochastically dominates a family $(Y_{\mathrm{i}})_{\mathrm{i} \in \mathcal{K}_n}$ of i.i.d.\ Bernoulli variables with parameter $p'$, where $p' \to 1$ as $M \to \infty$. 
\end{proof}

\begin{remark}\label{rem:domination_largest_component_good_cubes}

    \begin{itemize}
        \item  If $p>p_c^{\mathrm{site}}$  in Corollary \ref{cor:domination_good_cubes}, by Proposition \ref{prop:largest_component_bernoulli_percolation}, the set of good cubes has a connected component of size $\Theta(n/M^d)$. Moreover, if a cube $\mathrm{C_i}$ is good and $\mathrm{C_i}$ belongs to this connected component, then observe that any vertex $x \in \mathrm{C_i}\cap \mathcal{G}_n$ belonging to the largest component of $\cG_n[\mathrm{C_i}]$ (the subgraph of $\cG_n$ induced by all vertices inside  $\mathrm{C_i}$) belongs a.a.s.\ to $L_1$: indeed, since each good cube contains at least one vertex, and the giant component of the graph of good cubes contains $\Theta(n/M^d)$ cubes, by Theorem~\ref{thm:giantcomponents}, a.a.s.\  it is the giant component of $G$.
        \item 
        The importance of good cubes is that they allow us to construct paths of $\mathcal G_n$: Given two good cubes $\mathrm{C_i},\mathrm{C_j}$ that are neighbors, and vertices $x\in \mathrm{C_i}, y\in \mathrm{C_j}$ that belong to the (unique) giant components of $\cG_n[\mathrm{C_i}]$ and $\cG_n[\mathrm{C_j}]$, respectively, there exists a path connecting $x,y$ that is totally contained within $\mathrm{C_i}\cup \mathrm{C_j}$: Indeed, by definition, $\mathrm{C_i} \cup \mathrm{C_j}$ has a crossing component which contains $x$ and $y$. More generally, the same holds if $\mathrm{C_i}$ and $\mathrm{C_j}$ belong to the same connected component in the set of good cubes; one may just extract a path of good cubes, such that $x$ belongs to the first cube  on the path, $y$ belongs to the last cube on the path, and $x,y$ are connected by a path of vertices of $\cG_n$  totally contained within the path of good cubes. 
    \end{itemize}
\end{remark}

A 2-dimensional slice of $\mathcal{K}_n$ is a subset of cubes with index $\mathrm{i}$ such that all but $2$ coordinates are fixed (see Figure~\ref{fig:bigcube} for an example). We denote such a slice by $\mathcal{K}_n^{2}$. Since every $2$-dimensional slice is isomorphic to $\mathbb{Z}^2$ restricted to a box of a certain side length, by abuse of notation we can thus also define both vertical and horizontal orientations for each slice. Thus, w.l.o.g., we may suppose in the following that $\mathcal{K}^{2}_{n}$ is the $2$-dimensional slice made out of cubes $\mathrm{C_i}$ where only the first two coordinates are not equal to 0.  

Recall that for a cube $\mathrm{C_j}$, we denote by $\mathrm{j}_1,\ldots, \mathrm{j}_d$ its $d$ coordinates. Fix $\alpha>0$. Partition the slice $\mathcal{K}_n^{2}$ into disjoint vertical strips, each formed by
$\lceil \alpha\log n\rceil$ consecutive columns of cubes (see Figure ~\ref{fig:bigcube}). Formally, for $s \in \mathbb{Z}$, denote the $s$-th $\alpha$-logarithmic vertical strip by 
    \begin{equation}\label{eq:definition_stips}
          \mathrm{V}_s:= \{\mathrm{C_j} \in \mathcal{K}_n^2: \lceil \alpha \log n \rceil-s \leq \mathrm{j}_1< s+\lceil \alpha \log n \rceil \}.
    \end{equation}
    Recall also that strips at the boundary of the slice might have a smaller width, in which case we say that the strip is on the boundary of $\Lambda_n$. In a similar way, we define the $s$-th $\alpha$-logarithmic horizontal strip as \begin{equation}
        \mathrm{H}_s:= \{\mathrm{C_j} \in \mathcal{K}_n^2: \lceil \alpha \log n \rceil-s \leq \mathrm{j}_2< s+\lceil \alpha \log n \rceil \}.
    \end{equation}

    We prove the following result regarding long crossings of good cubes inside of $\Lambda_n$: 
\begin{lemma}\label{lem:crossingsstrips}
    Let $d \ge 2$ and $\eps>0$, and assume $(1+\eps)r_g \le r =O(r_c)$ and assume $M$ large enough. Let $\alpha$ be a large enough positive constant.
    Then, there exists $c>0$ such that a.a.s., every $\alpha$-logarithmic strip that is not on the boundary of $\Lambda_n$ has at least $c \log n$ disjoint long crossings of good cubes.
    Then, a.a.s.\ there exists $c > 0$ such that each $\alpha$-logarithmic strip that is not on the boundary of $\Lambda_n$ has at least $c \log n$ disjoint long crossings of good cubes.
\end{lemma}
\begin{proof}
    By Remark~\ref{rem:domination_largest_component_good_cubes}, for $M$ large enough, the good cubes of the tessellation stochastically dominate a supercritical Bernoulli site percolation process. Observe that the number of $2$-dimensional slices of $\mathcal{K}_n$ is $O(n^{1/d})$, and also the number of $\alpha$-logarithmic strips inside such a slice is $O(n^{1/d})$. Hence, for $\alpha$ large enough, by Lemma~\ref{lem:crossingpercolation} applied to every $\alpha$-logarithmic strip, together with a union bound, the probability that at least one strip does not have $c\log n$ long crossings of good cubes is at most
    $O(n^{2/d})c_2n^{1/d} e^{-c_3 \alpha \log n}=o(1)$. The lemma follows.
\end{proof}

\section{Random paths of good cubes}\label{sec:randompaths}

The framework described in this section is going to be important to define a flow, which in turn will be useful to get our final upper bounds on the cover time through the effective resistance.
We first introduce a structure of random paths between good cubes of a 2-dimensional slice. This construction was first introduced by Abe ~\cite{aberesistance}.
We focus here on the cubes $\mathrm{C}$ introduced in Section~\ref{sec:renormalization}, although the construction extends more generally to any collection of good cubes with suitable connectivity properties dominating a supercritical Bernoulli site percolation process.

Our first step will then be to show that for any vertices $x, y$ of $\cG_n$ belonging to good cubes $\mathrm{C_i}$ and $\mathrm{C_j}$, respectively, and belonging to the largest components of these cubes (and hence in particular to $L_1$), the effective resistance between $x$ and $y$ can be suitably bounded from above. We thus focus on good cubes only. In this section we will allow $M$ to change slightly with $n$ so that a cube $\mathrm{C_i}$ is either totally contained within $\Lambda_n$ or shares at most some lower-dimensional faces with $\Lambda_n$, and so $\Lambda_n$ is completely covered by the tessellation.

Let \(\mathcal L_1\) and \(\mathcal L_2\) be two self-avoiding paths of cubes, each endowed with the orientation induced by the ordering of its cubes, and let \(\mathrm C \in \mathcal L_1 \cap \mathcal L_2\). We define the \emph{concatenation of \(\mathcal L_1\) and \(\mathcal L_2\) at \(\mathrm C\)} as the path obtained by following \(\mathcal L_1\) up to \(\mathrm C\), and then continuing along \(\mathcal L_2\) from \(\mathrm C\) onward. If \(\mathrm C\) is the last cube in \(\mathcal L_1 \cap \mathcal L_2\) with respect to this ordering in $\mathcal L_2$, then this concatenated path is self-avoiding.

In the same way we define the concatenation of two self-avoiding paths $\mathcal{J}_1$ and $\mathcal{J}_2$ of $\cG_n$ at a vertex $x$ that belongs to both paths, as the path that results from starting at $\mathcal{J}_1$ until hitting $x$ and after that continuing with $\mathcal{J}_2$. 

Recall the definitions of $\mathrm H_s,\mathrm{ V}_s$, see \eqref{eq:definition_stips}. Recall that by Lemma ~\ref{lem:crossingsstrips}, for $\alpha$ and $M$ large enough, a.a.s.\ there are at least $\lfloor c\log n \rfloor$ long crossings of good cubes in each $\alpha$-logarithmic strip (that is not on the boundary) of every 2-dimensional slice of $\Lambda_n$. In what follows we condition on this event.

Without loss of generality suppose in the following that $\mathcal{K}^{2}_{n}$ is the 2-dimensional slice made out of cubes $\mathrm{C_i}$ with only the first two coordinates not equal to 0. Note that strips at the boundary of the slice might have a smaller width.

Select $\lfloor c\log n \rfloor$ disjoint long crossings of good cubes for each strip $\mathrm{H}_s,\mathrm{V}_s$ that is not on the boundary. We enumerate the vertical crossings of good cubes of the vertical strip $\mathrm{V}_s$, counting from right to left and denote them by $\mathrm{V}_s^{\ell}$ for $1\leq \ell \leq \lfloor c\log n \rfloor$ (this way of enumerating the vertical crossings is going to be important later). Similarly, we enumerate the horizontal crossings of $\mathrm{H}_s$, this time counting from bottom to top and denoting them by $\mathrm{H}_s^{\ell}$, for $1 \le \ell \le \lfloor c\log n \rfloor$.    
    
    For $r,s \in \mathbb{Z}$, denote by $\mathrm{T}_{s,r}$ the smallest $\emph{parallelepiped}$ that contains all cubes $\mathrm{C_i}$ from $\mathrm{V}_s\cap \mathrm{H}_r$. Note that the parallelepiped $\mathrm{T}_{s,r}$ that is not at the boundary of the 2-dimensional slice has Euclidean side length $Mr\lceil \alpha \log n \rceil$ in the first 2 dimensions and $Mr$ in the others, as the side length of a cube $\mathrm{C_i}$ is $Mr$. We consider in the construction that follows only parallelepipeds that are not at the boundary, as this will be sufficient for our purposes. Also note that any vertical crossing of good cubes $\mathrm{V}^{\ell}_{s}$ of $\mathrm{V}_s$ and any horizontal crossing of good cubes $\mathrm{H}^{\ell}_r$ of $\mathrm{H}_r$ intersect in at least one good cube $\mathrm{C}$ that lies inside of $\mathrm{T}_{s,r}$, as the slice is a 2-dimensional object. Let now $\mathcal{L}$ be a self-avoiding path of parallelepipeds $\mathrm{T}$ between $\mathrm{T}_{s_1,r_1}$ and $\mathrm{T}_{s_2,r_2}$ (that is, a sequence of adjacent parallelepipeds of the same slice such that no parallelepiped appears more than once, see Figure ~\ref{fig:path_mathcal_L}). We then construct an ordered set $(\mathrm{T}_{a_1,b_1},\mathrm{T}_{a_2,b_2},...,\mathrm{T}_{a_k,b_k})$ of parallelepipeds that belong to $\mathcal{L}$ and that we use to identify the change of its direction in the following way: 
    \begin{itemize}
    
    \item Set $\mathrm{T}_{a_1,b_1}=\mathrm{T}_{s_1,r_1}$. Let $\mathrm{T}_{a_\ell,b_\ell} \in \mathcal{L}$, and suppose without loss of generality that the next step of the path $\mathcal{L}$ is to the right. $\mathrm{T}_{a_{\ell+1},b_{\ell+1}}$ is then defined as the last parallelepiped of $\mathcal{L}$ that is visited by going only to the right of $\mathrm{T}_{a_\ell,b_\ell}$. The construction is analogous if the path $\mathcal{L}$ after $\mathrm{T}_{a_\ell,b_\ell}$ goes to the left, up or down, respectively. We finish at $\mathrm{T}_{a_k,b_k}=\mathrm{T}_{s_2,r_2}$. If $(a_{i+1},b_{i+1})-(a_i,b_i)=me_1$ for some $m$, we say that $\mathcal{L}$ has horizontal orientation at step $(a_i,b_i)$, and analogously, we say that it has vertical orientation if instead of $e_1$ we have $e_2$. 
    \end{itemize}
    
    Thus, $(\mathrm T_{a_\ell,b_\ell})^k_{\ell=1}$ records the endpoints of the maximal straight segments of $\mathcal L$.
    See Figure ~\ref{fig:path_mathcal_L} for an example of the construction.
\begin{figure}[htb]
\begin{tikzpicture}
% Define square size
\def\squarex{1} % Square width
\def\squarey{1} % Square height

% Number of squares to the middle
\def\middle{4} % Number of squares to reach the middle from the first level

% Coordinates for the first square (bottom-left of the staircase)
\coordinate (start_x) at (0, 0);

\draw[thick] (0,0) rectangle (1,1);
\draw[thick] (1,0) rectangle (2,1);
\draw[thick] (2,0) rectangle (3,1);
\draw[thick] (3,0) rectangle (4,1);
\draw[thick] (4,0) rectangle (5,1);
\draw[thick] (5,0) rectangle (6,1);

\node[anchor=north west, scale=0.7] at (0,1) {$T_{a_1, b_1}$};
\draw[->] (1.5,0.5) -- (3.5,0.5);

\draw[thick] (5,-1) rectangle (6,0);
\draw[thick] (5,-2) rectangle (6,-1);
\draw[thick] (5,-3) rectangle (6,-2);
\node[anchor=north west, scale=0.7] at (5,1) {$\mathrm{T}_{a_2, b_2}$};

\draw[->] (5.5,-0.3) -- (5.5,-1.7);

\draw[thick] (6,-3) rectangle (7,-2);
\draw[thick] (7,-3) rectangle (8,-2);
\node[anchor=north west, scale=0.7] at (5,-2) {$\mathrm{T}_{a_3, b_3}$};

\draw[->] (5.8,-2.5) -- (7.2,-2.5);

\draw[thick] (7,-3) rectangle (8,-2);
\draw[thick] (7,-2) rectangle (8,-1);
\draw[thick] (7,-1) rectangle (8,0);
\draw[thick] (7,0) rectangle (8,1);
\draw[thick] (7,1) rectangle (8,2);
\draw[thick] (7,2) rectangle (8,3);
\draw[thick] (7,2) rectangle (8,3);
\node[anchor=north west, scale=0.7] at (7,-2) {$\mathrm{T}_{a_4, b_4}$};

\draw[->] (7.5,-1.2) -- (7.5,1.2);

\draw[thick] (7,2) rectangle (8,3);
\draw[thick] (8,2) rectangle (9,3);
\draw[thick] (9,2) rectangle (10,3);
\node[anchor=north west, scale=0.7] at (7,3) {$\mathrm{T}_{a_5, b_5}$};

\node[anchor=north west, scale=1.6] at (1,2.5) {$\mathcal{L}$};

\draw[->] (7.8,2.5) -- (9.2,2.5);

\node at (10.3,2.5) {...};

\end{tikzpicture}
\captionof{figure}{Path $\mathcal{L}$ of parallelepipeds $\mathrm{T}$. The path changes direction when visiting the parallelepipeds $\mathrm{T}_{a_i,b_i}$ } \label{fig:path_mathcal_L}
\end{figure}
    Assume in this subsection that $r_1=r_2$ for the parallelepipeds $\mathrm{T}_{s_1,r_1}$ and $\mathrm{T}_{s_2,r_2}$, and assume that $s_2>s_1$, so that $\mathrm{T}_{s_2,r_2}$ is to the right of $\mathrm{T}_{s_1,r_1}$. We will also throughout this subsection require $s_2-s_1$ to be odd, so that the parallelepipeds $\mathrm{T}_{s_1,r_1}$ and $\mathrm{T}_{s_2,r_2}$ are separated by an even number of parallelepipeds. We now construct a self-avoiding random path $\mathcal{L}_X$ (based on the straight path $\mathcal{L}$) of parallelepipeds between $\mathrm{T}_{s_1,r_1}$ and $\mathrm{T}_{s_2,r_2}$ as follows: 
    \begin{itemize}
    \item Let $X$ be a continuous random variable chosen uniformly in $[0,\frac{s_2-s_1+1}{2}]$, and let $z$ be the point in the left bottom corner of $\mathrm{T}_{\frac{s_2+s_1+1}{2},r_1}$. We define $y_X \in \Lambda_n$ to be $y_X := z+(0,Mr\lceil\alpha \log n\rceil X,0,...,0)$  (the $x$-coordinate of this point is the midpoint of $\mathrm{T}_{s_1,r_1}$ and $\mathrm{T}_{s_2,r_2}$, see Figure ~\ref{fig:middle_point_y_x}). 
    \end{itemize}
    Note that we assume here that $y_{X}$ is inside of $\Lambda_n$; if this were not the case we define $y'_{X}=z-(0,Mr\lceil \alpha \log n \rceil X,0,...,0)$, and the construction that follows is reflected vertically. Note that at least one of $y_X,y_X'$ belongs to $\Lambda_n$. We will thus w.l.o.g.\ follow our construction assuming that $y_X \in \Lambda_n$.
    
    We denote by $\mathcal{L}_X$ the a.s.\ unique random path of parallelepipeds intersecting the segment between the bottom left corner of $\mathrm{T}_{s_1,r_1}$ and the point $y_X$, and the segment between the bottom right corner of $\mathrm{T}_{s_2,r_2}$ and $y_X$ (see Figure~\ref{fig:random_path_of_supercubes}).
\begin{figure}[htb]
\begin{minipage}{0.5\textwidth}
     \begin{tikzpicture}[scale=0.8]

% Define smaller cube size
\def\cubex{1} % Cube width
\def\cubey{1} % Cube height
\def\cubez{1} % Cube depth

% Coordinates for the first cube
\coordinate (A1) at (0, 0, 0);
\coordinate (B1) at (\cubex, 0, 0);
\coordinate (C1) at (\cubex, \cubey, 0);
\coordinate (D1) at (0, \cubey, 0);
\coordinate (E1) at (0, 0, -\cubez);
\coordinate (F1) at (\cubex, 0, -\cubez);
\coordinate (G1) at (\cubex, \cubey, -\cubez);
\coordinate (H1) at (0, \cubey, -\cubez);

% Draw the first cube (T_{s_1, r_1})
\draw[thick] (A1) -- (B1) -- (C1) -- (D1) -- cycle;
\draw[thick] (A1) -- (E1) -- (H1) -- (D1) -- cycle;
\draw[thick] (E1) -- (F1) -- (G1) -- (H1) -- cycle;
\draw[thick] (B1) -- (F1);
\draw[thick] (C1) -- (G1);

% Label the first cube
\node[anchor=north west, scale=0.7] at (D1) {$T_{s_1, r_1}$};

% Add a red node at the bottom-left corner of the first cube
\node[red, circle, fill=red, scale=0.5] at (A1) {};

% Move to the second cube to the right (T_{s_2, r_2})
\coordinate (A2) at (2.5, 0, 0);
\coordinate (B2) at (2.5+\cubex, 0, 0);
\coordinate (C2) at (2.5+\cubex, \cubey, 0);
\coordinate (D2) at (2.5, \cubey, 0);
\coordinate (E2) at (2.5, 0, -\cubez);
\coordinate (F2) at (2.5+\cubex, 0, -\cubez);
\coordinate (G2) at (2.5+\cubex, \cubey, -\cubez);
\coordinate (H2) at (2.5, \cubey, -\cubez);

% Draw the second cube (T_{s_2, r_2})
\draw[thick] (A2) -- (B2) -- (C2) -- (D2) -- cycle;
\draw[thick] (A2) -- (E2) -- (H2) -- (D2) -- cycle;
\draw[thick] (E2) -- (F2) -- (G2) -- (H2) -- cycle;
\draw[thick] (B2) -- (F2);
\draw[thick] (C2) -- (G2);

% Label the second cube
\node[anchor=north west, scale=0.7] at (D2) {$T_{ \frac{s_2+s_1+1}{2},r_1}$};

\node[anchor=north west, scale=0.7] at (2.3,0) {$z$};

\node[red, circle, fill=red, scale=0.45] at (2.5,0) {};

% Middle cube (T_{s_1, (r_1 + r_2 + 1)/2})
% Coordinates for the middle cube
\coordinate (A3) at (5, 0, 0);
\coordinate (B3) at (5+\cubex, 0, 0);
\coordinate (C3) at (5+\cubex, \cubey, 0);
\coordinate (D3) at (5, \cubey, 0);
\coordinate (E3) at (5, 0, -\cubez);
\coordinate (F3) at (5+\cubex, 0, -\cubez);
\coordinate (G3) at (5+\cubex, \cubey, -\cubez);
\coordinate (H3) at (5, \cubey, -\cubez);

% Draw the middle cube
\draw[thick] (A3) -- (B3) -- (C3) -- (D3) -- cycle;
\draw[thick] (A3) -- (E3) -- (H3) -- (D3) -- cycle;
\draw[thick] (E3) -- (F3) -- (G3) -- (H3) -- cycle;
\draw[thick] (B3) -- (F3);
\draw[thick] (C3) -- (G3);

% Label the middle cube
\node[anchor=north west, scale=0.7] at (D3) {$T_{s_2, r_2}$};

% Add a red node at the bottom-right corner of the last cube
\node[red, circle, fill=red, scale=0.5] at (B3) {};

% Draw a dashed line from the top-right vertical edge of the middle cube to a red node above
\coordinate (top_right_edge) at (A2);
\coordinate (red_node) at (2.5, \cubey + 3.5, 0); % Red node above

\draw[dashed] (A2) -- (red_node);

% Red node labeled y_X
\node[red, circle, fill=red, scale=0.5, label=above:{$y_X$}] at (red_node) {};

% Draw two red segments from y_X to A1 and B3
\draw[red, thick] (red_node) -- (A1); % Segment to bottom-left corner of the first cube
\draw[red, thick] (red_node) -- (B3); % Segment to bottom-right corner of the last cube

\node[scale=1.5] at (1.5, 0.5, 0) {$\cdot$};
\node[scale=1.5] at (1.75, 0.5, 0) {$\cdot$};
\node[scale=1.5] at (2, 0.5, 0) {$\cdot$};

\node[scale=1.5] at (4, 0.5, 0) {$\cdot$};
\node[scale=1.5] at (4.25, 0.5, 0) {$\cdot$};
\node[scale=1.5] at (4.5, 0.5, 0) {$\cdot$};
\end{tikzpicture}
\captionof{figure}{Position of $y_X$ relative to the parallelepipeds}
    \label{fig:middle_point_y_x}
\end{minipage}%
\begin{minipage}{0.5\textwidth}
\begin{tikzpicture}[scale=0.8]

% Define square size
\def\squarex{1} % Square width
\def\squarey{1} % Square height
\draw[draw=transparent, fill=red!30, thick] (0*\squarex, 0*\squarey) rectangle (1*\squarex, 1*\squarey); 
\draw[draw=transparent, fill=red!30, thick] (1*\squarex, 1*\squarey) rectangle (2*\squarex , 2*\squarey);
\draw[draw=transparent, fill=red!30, thick] (1*\squarex, 0*\squarey) rectangle (2*\squarex, 1*\squarey); 
\draw[draw=transparent, fill=red!30, thick] (2*\squarex, 1*\squarey) rectangle (3*\squarex, 2*\squarey); 
\draw[draw=transparent, fill=red!30, thick] (3*\squarex, 1*\squarey) rectangle (4*\squarex , 2*\squarey); 
\draw[draw=transparent, fill=red!30, thick] (3*\squarex, 2*\squarey) rectangle (4*\squarex , 3*\squarey);
 \draw[draw=transparent, fill=red!30, thick] (4*\squarex, 2*\squarey) rectangle (5*\squarex , 3*\squarey);
\draw[draw=transparent, fill=red!30, thick] (5*\squarex, 2*\squarey) rectangle (6*\squarex , 3*\squarey);
\draw[draw=transparent, fill=red!30, thick] (6*\squarex, 2*\squarey) rectangle (7*\squarex , 3*\squarey);
\draw[draw=transparent, fill=red!30, thick] (6*\squarex, 1*\squarey) rectangle (7*\squarex , 2*\squarey);
\draw[draw=transparent, fill=red!30, thick] (7*\squarex, 1*\squarey) rectangle (8*\squarex , 2*\squarey);
\draw[draw=transparent, fill=red!30, thick] (8*\squarex, 1*\squarey) rectangle (8*\squarex , 2*\squarey);
\draw[draw=transparent, fill=red!30, thick] (8*\squarex, 0*\squarey) rectangle (9*\squarex , 1*\squarey);
\draw[draw=transparent, fill=red!30, thick] (9*\squarex, 0*\squarey) rectangle (10*\squarex , 1*\squarey);
\draw[draw=transparent, fill=red!30, ultra thin] (8*\squarex, 1*\squarey) rectangle (9*\squarex , 2*\squarey);

% Number of squares to the middle
\def\middle{4} % Number of squares to reach the middle from the first level

% Coordinates for the first square (bottom-left of the staircase)
\coordinate (start_x) at (0, 0);

% Draw the increasing part of the staircase
\foreach \level in {0,...,4} {
    \foreach \i in {0,...,\level} {
        \draw[thick] (\level*\squarex, \i*\squarey) rectangle (\level*\squarex + \squarex, \i*\squarey + \squarey);
    }
}
\draw[thick] (5*\squarex, 0*\squarey) rectangle (5*\squarex + \squarex, 0*\squarey + \squarey);
\draw[thick] (5*\squarex, 1*\squarey) rectangle (5*\squarex + \squarex, 1*\squarey + \squarey);
\draw[thick] (5*\squarex, 2*\squarey) rectangle (5*\squarex + \squarex, 2*\squarey + \squarey);
\draw[thick] (5*\squarex, 3*\squarey) rectangle (5*\squarex + \squarex, 3*\squarey + \squarey);
\draw[thick] (5*\squarex, 4*\squarey) rectangle (5*\squarex + \squarex, 4*\squarey + \squarey);

\draw[thick] (6*\squarex, 0*\squarey) rectangle (6*\squarex + \squarex, 0*\squarey + \squarey);
\draw[thick] (6*\squarex, 1*\squarey) rectangle (6*\squarex + \squarex, 1*\squarey + \squarey);
\draw[thick] (6*\squarex, 2*\squarey) rectangle (6*\squarex + \squarex, 2*\squarey + \squarey);
\draw[thick] (6*\squarex, 3*\squarey) rectangle (6*\squarex + \squarex, 3*\squarey + \squarey);

\draw[thick] (7*\squarex, 0*\squarey) rectangle (7*\squarex + \squarex, 0*\squarey + \squarey);
\draw[thick] (7*\squarex, 1*\squarey) rectangle (7*\squarex + \squarex, 1*\squarey + \squarey);
\draw[thick] (7*\squarex, 2*\squarey) rectangle (7*\squarex + \squarex, 2*\squarey + \squarey);

\draw[thick] (8*\squarex, 0*\squarey) rectangle (8*\squarex + \squarex, 0*\squarey + \squarey);
\draw[thick] (8*\squarex, 1*\squarey) rectangle (8*\squarex + \squarex, 1*\squarey + \squarey);

\draw[thick] (9*\squarex, 0*\squarey) rectangle (9*\squarex + \squarex, 0*\squarey + \squarey);

% First square at the bottom-left
\coordinate (A1) at (0, 0);
\coordinate (B1) at (\squarex, 0);
\coordinate (C1) at (\squarex, \squarey);
\coordinate (D1) at (0, \squarey);
% Label the first square
\node[anchor=north west, scale=0.5] at (D1) {$T_{s_1, r_1}$};
% Second labeled square (at the middle of the stair)
\coordinate (A2) at (4*\squarex, 0);
\coordinate (B2) at (4*\squarex + \squarex, 0);
\coordinate (C2) at (4*\squarex + \squarex, \squarey);
\coordinate (D2) at (4*\squarex, \squarey);
% Label the second square
\node[anchor=north west, scale=0.5] at (D2) {$T_{ \frac{s_2+s_1-1}{2},r_1}$};
% Third labeled square (at the far right)
\coordinate (A3) at (8*\squarex, 0);
\coordinate (B3) at (9*\squarex + \squarex, 0);
\coordinate (C3) at (8*\squarex + \squarex, \squarey);
\coordinate (D3) at (9*\squarex, \squarey);

% Label the third square
\node[anchor=north west, scale=0.5] at (D3) {$T_{s_2, r_2}$};

\coordinate (red_node) at (5*\squarex, 2.8);
\node[red, circle, fill=red, scale=0.5, label=above:{$y_X$}] at (red_node) {}; 

\draw[red, thick] (red_node) -- (A1); % Segment to bottom-left corner of the first cube
\draw[red, thick] (red_node) -- (B3); % Segment to bottom-right corner of the last cube

% Fill the squares intersected by the red lines with translucent red

\end{tikzpicture}
    \captionof{figure}{The red cubes form the path $\mathcal{L}_X$}
    \label{fig:random_path_of_supercubes}
\end{minipage}
\end{figure}
    Let $\mathrm{C_i} \in \mathrm{T}_{s_1,r_1}$ and $\mathrm{C_j} \in \mathrm{T}_{s_2,r_2}$ be any good cubes such that both of them belong to some vertical crossings of the strips $\mathrm{V}_{s_1}$ and $\mathrm{V}_{s_2}$, respectively (see the leftmost part of Figure~\ref{fig:paths_from_c_i} for an example of such crossings for $\mathrm{T}_{s_1,r_1}$). Either such crossings have length $O(\log n)$, or otherwise, by Theorem ~\ref{them:chemical_distance_percolation}, a.a.s. we may select crossings with this property. We denote the selected vertical crossings by $\mathrm{V}(\mathrm{C_i})$ and $\mathrm{V}(\mathrm{C_j})$. From the vertical crossing of good cubes $\mathrm{V}(\mathrm{C_i})$ that contains $\mathrm{C_i}$ we then extract a bottom-to-top crossing of $\mathrm{T}_{s_1,r_1}$ made from good cubes which we denote by $\mathrm{V}'(\mathrm{C_i})$ (see yellow cubes in the middle picture of Figure~\ref{fig:paths_from_c_i}). Similarly, we extract a bottom-to-top-crossing of $\mathrm{T}_{s_2,r_2}$ made from good cubes, from the vertical crossing $\mathrm{V}(\mathrm{C_j})$ that contains $\mathrm{C_j}$, and we denote it by $\mathrm{V}{'} (\mathrm{C_j})$. 
    
    Now, given $1\leq \ell \leq \lfloor c\log n \rfloor $, our goal is to construct a self-avoiding path of good cubes between $\mathrm{C_i}$ and $\mathrm{C_j}$ that is entirely contained within $\mathcal{L}_X$ and has index $\ell$.
    
    Since $\mathrm{V}{'}(\mathrm{C_i})$ is self-avoiding, we may split it into two paths of good cubes with intersection only at $\mathrm{C_i}$: these paths start at $\mathrm{C_i}$ and go to the left and down, following the induced orientation, see the rightmost picture of Figure~\ref{fig:paths_from_c_i}. Denote the paths by $\mathrm{V}_1'(\mathrm{C_i})$ and $\mathrm{V}_2'(\mathrm{C_i})$ and do the same for $\mathrm{V}'(\mathrm{C_j})$. At least one of these paths of good cubes that start at $\mathrm{C_i}$ intersects $\mathrm{H}_{r_1}^{\ell}$, since the bottom-to-top crossing (of $\mathrm{T}_{s_1,r_1}$) $\mathrm{V}'(\mathrm{C_i})$ made from good cubes intersects the horizontal crossing of good cubes (of the strip $\mathrm{H}_{r_1}$) $\mathrm{H}^{\ell}_{r_1}$ in at least one good cube. Suppose without loss of generality that $\mathrm{V}_1'(\mathrm{C_i})$ intersects this horizontal crossing. Concatenate then $\mathrm{V}_1'(\mathrm{C_i})$ with $\mathrm{H}^{\ell}_{r_1}$ at the last good cube of $\mathrm{H}^{\ell}_{r_1}$ that intersects $\mathrm{V}_1'(\mathrm{C_i})$; we denote the resulting path by $\mathcal{Q}_1$ and note that it is non-intersecting (see  Figure~\ref{fig:concatenation_of_paths}, the bottom picture shows the final $\mathcal{Q}_1$).
    Now, observe that $(a_3,b_3)-(a_2,b_2)$ must be vertically oriented. We concatenate $\mathcal{Q}_1$ to the vertical crossing $\mathrm{V}_{a_3}^\ell$ (seen as path that goes from bottom to top) at the last cube of $\mathrm{V}_{a_3}^{\ell}$ that intersects $\mathcal{Q}_1$; we denote this path by $\mathcal{Q}_2$.

\begin{figure}[htb]
\begin{minipage}{0.49\textwidth}
\begin{tikzpicture}[scale=0.8]

% Define square size
\def\squarex{1.5} % Square width
\def\squarey{1.5} % Square height
\def\squarexx{1} % Square width
\def\squareyy{1} % Square height

\draw[thick] (6*\squarex, 0*\squarey) rectangle (6*\squarex + \squarex, 0*\squarey + \squarey);
\draw[thick] (6*\squarex, 1*\squarey) rectangle (6*\squarex + \squarex, 1*\squarey + \squarey);
\draw[thick] (6*\squarex, -1*\squarey) rectangle (6*\squarex + \squarex, -1*\squarey + \squarey);
\draw[thick] (6*\squarex, -2*\squarey) rectangle (6*\squarex + \squarex, -2*\squarey + \squarey);
\draw[thick] (6*\squarex, 2*\squarey) rectangle (6*\squarex + \squarex, 2*\squarey + \squarey);

\foreach \y in {2.2, 2.25, 2.3, ..., 3} {
    \draw[draw=black, ultra thin] 
        (6.1*\squarex, \y*\squarey) rectangle (6.15*\squarex, \y*\squarey + 0.05*\squarey); % Small cubes (0.05 x 0.05 squares)
}

\foreach \y in {6.15, 6.2, 6.25, ..., 6.65} {
    \draw[draw=black, ultra thin] 
        (\y*\squarex, 2.2*\squarey) rectangle (\y*\squarex+0.05*\squarex, 2.25*\squarey); % Small cubes (0.05 x 0.05 squares)
}

\foreach \y in {0.9, 0.95, 1, ..., 2.25} {
    \draw[draw=black, ultra thin] 
        (6.65*\squarex, \y*\squarey) rectangle (6.7*\squarex, \y*\squarey + 0.05*\squarey); % Small cubes (0.05 x 0.05 squares)
}

\foreach \y in {6.45, 6.45, 6.5, ..., 6.65} {
    \draw[draw=black, ultra thin] 
        (\y*\squarex, 0.9*\squarey) rectangle (\y*\squarex+0.05*\squarex, 0.95*\squarey); % Small cubes (0.05 x 0.05 squares)
}

\foreach \y in {0.3, 0.35, 0.4, ..., 0.9} {
    \draw[draw=black, ultra thin] 
        (6.45*\squarex, \y*\squarey) rectangle (6.5*\squarex, \y*\squarey + 0.05*\squarey); % Small cubes (0.05 x 0.05 squares)
}

\draw[draw=black,fill=red, ultra thin] 
        (6.45*\squarex, 0.5*\squarey) rectangle (6.5*\squarex, 0.5*\squarey + 0.05*\squarey);

\foreach \y in {6.2, 6.25, 6.3, ..., 6.45} {
    \draw[draw=black, ultra thin] 
        (\y*\squarex, 0.3*\squarey) rectangle (\y*\squarex+0.05*\squarex, 0.35*\squarey); % Small cubes (0.05 x 0.05 squares)
}

\foreach \y in {-0.2, -0.15, -0.1, ..., 0.3} {
    \draw[draw=black, ultra thin] 
        (6.2*\squarex, \y*\squarey) rectangle (6.25*\squarex, \y*\squarey + 0.05*\squarey); % Small cubes (0.05 x 0.05 squares)
}

\draw[draw=black, ultra thin] 
        (6.25*\squarex, -0.2*\squarey) rectangle (6.3*\squarex, -0.2*\squarey + 0.05*\squarey); % Small cubes (0.05 x 0.05 squares)
\draw[draw=black, ultra thin] 
        (6.3*\squarex, -0.2*\squarey) rectangle (6.35*\squarex, -0.2*\squarey + 0.05*\squarey); % Small cubes (0.05 x 0.05 squares)

\foreach \y in {-1.3, -1.25, -1.2, ..., -0.2} {
    \draw[draw=black, ultra thin] 
        (6.3*\squarex, \y*\squarey) rectangle (6.35*\squarex, \y*\squarey + 0.05*\squarey); % Small cubes (0.05 x 0.05 squares)
}

\foreach \y in {6.3, 6.35, 6.4, ..., 6.85} {
    \draw[draw=black, ultra thin] 
        (\y*\squarex, -1.3*\squarey) rectangle (\y*\squarex+0.05*\squarex, -1.35*\squarey); % Small cubes (0.05 x 0.05 squares)
}

\foreach \y in {-2, -1.95, -1.9, ..., -1.35} {
    \draw[draw=black, ultra thin] 
        (6.85*\squarex, \y*\squarey) rectangle (6.9*\squarex, \y*\squarey + 0.05*\squarey); % Small cubes (0.05 x 0.05 squares)
}

\node at (6*\squarex + 0.5*\squarex, 3.5*\squarey) {\vdots};
\node at (6*\squarex + 0.5*\squarex, -2.5*\squarey) {\vdots};
% Add label T_{s_1, r_1} at the corner of the center square
\node[anchor=north east, scale=0.6] at (6.5*\squarex, 0*\squarey + \squarey) {\( \mathrm{T}_{s_1, r_1} \)};

% Draw the main rectangles (shifted by 2*\squarex)
\draw[thick] (8*\squarex, 0*\squarey) rectangle (8*\squarex + \squarex, 0*\squarey + \squarey);
\draw[thick] (8*\squarex, 1*\squarey) rectangle (8*\squarex + \squarex, 1*\squarey + \squarey);
\draw[thick] (8*\squarex, -1*\squarey) rectangle (8*\squarex + \squarex, -1*\squarey + \squarey);
\draw[thick] (8*\squarex, -2*\squarey) rectangle (8*\squarex + \squarex, -2*\squarey + \squarey);
\draw[thick] (8*\squarex, 2*\squarey) rectangle (8*\squarex + \squarex, 2*\squarey + \squarey);

% Add vertical path with small rectangles (0.05 x 0.05 squares), no fill, and ultra-thin lines
\foreach \y in {2.2, 2.25, 2.3, ..., 3} {
    \draw[draw=black, ultra thin] 
        (8.1*\squarex, \y*\squarey) rectangle (8.15*\squarex, \y*\squarey + 0.05*\squarey);
}

\foreach \y in {8.15, 8.2, 8.25, ..., 8.65} {
    \draw[draw=black, ultra thin] 
        (\y*\squarex, 2.2*\squarey) rectangle (\y*\squarex+0.05*\squarex, 2.25*\squarey);
}

\foreach \y in {0.9, 0.95, 1, ..., 2.25} {
    \draw[draw=black,ultra thin] 
        (8.65*\squarex, \y*\squarey) rectangle (8.7*\squarex, \y*\squarey + 0.05*\squarey);
}

\draw[draw=black, fill=yellow,ultra thin] 
        (8.65*\squarex, 0.9*\squarey) rectangle (8.65*\squarex+0.05*\squarex, 0.95*\squarey);

\draw[draw=black, fill=yellow,ultra thin] 
        (8.65*\squarex, 0.95*\squarey) rectangle (8.65*\squarex+0.05*\squarex, 1*\squarey);

\foreach \y in {8.45, 8.5, 8.55, ..., 8.65} {
    \draw[draw=black, fill=yellow,ultra thin] 
        (\y*\squarex, 0.9*\squarey) rectangle (\y*\squarex+0.05*\squarex, 0.95*\squarey);
}

\foreach \y in {0.3, 0.35, 0.4, ..., 0.9} {
    \draw[draw=black, fill=yellow, ultra thin] 
        (8.45*\squarex, \y*\squarey) rectangle (8.5*\squarex, \y*\squarey + 0.05*\squarey);
}

\draw[draw=black, fill=red, ultra thin] 
        (8.45*\squarex, 0.5*\squarey) rectangle (8.5*\squarex, 0.5*\squarey + 0.05*\squarey);

\foreach \y in {8.2, 8.25, 8.3, ..., 8.45} {
    \draw[draw=black,fill=yellow ,ultra thin] 
        (\y*\squarex, 0.3*\squarey) rectangle (\y*\squarex+0.05*\squarex, 0.35*\squarey);
}

\foreach \y in {0, 0.05, 0.1, ..., 0.3} {
    \draw[draw=black,fill=yellow, ultra thin] 
        (8.2*\squarex, \y*\squarey) rectangle (8.25*\squarex, \y*\squarey + 0.05*\squarey);
}

\foreach \y in {-0.2, -0.15, -0.1, ..., 0} {
    \draw[draw=black, ultra thin] 
        (8.2*\squarex, \y*\squarey) rectangle (8.25*\squarex, \y*\squarey + 0.05*\squarey);
}

\draw[draw=black, ultra thin] 
        (8.25*\squarex, -0.2*\squarey) rectangle (8.3*\squarex, -0.2*\squarey + 0.05*\squarey);
\draw[draw=black, ultra thin] 
        (8.3*\squarex, -0.2*\squarey) rectangle (8.35*\squarex, -0.2*\squarey + 0.05*\squarey);

\foreach \y in {-1.3, -1.25, -1.2, ..., -0.2} {
    \draw[draw=black, ultra thin] 
        (8.3*\squarex, \y*\squarey) rectangle (8.35*\squarex, \y*\squarey + 0.05*\squarey);
}

\foreach \y in {8.3, 8.35, 8.4, ..., 8.85} {
    \draw[draw=black, ultra thin] 
        (\y*\squarex, -1.3*\squarey) rectangle (\y*\squarex+0.05*\squarex, -1.35*\squarey);
}

\foreach \y in {-2, -1.95, -1.9, ..., -1.35} {
    \draw[draw=black, ultra thin] 
        (8.85*\squarex, \y*\squarey) rectangle (8.9*\squarex, \y*\squarey + 0.05*\squarey);
}

% Add vertical ellipsis to indicate continuation
\node at (8*\squarex + 0.5*\squarex, 3.5*\squarey) {\vdots};
\node at (8*\squarex + 0.5*\squarex, -2.5*\squarey) {\vdots};

% Add label T_{s_1, r_1} at the corner of the center square
\node[anchor=north east, scale=0.6] at (8.5*\squarex, 0*\squarey + \squarey) {\( \mathrm{T}_{s_1, r_1} \)};

%%%
% Draw the main square starting at x = 10
\draw[thick] (10*\squarex, -0.5*\squarey) rectangle (10*\squarex + 2*\squarex, 0*\squarey + 1.5*\squarey);

\draw[draw=black, fill=blue!20,ultra thin] 
        (11.3*\squarex, 1.4*\squarey) rectangle (11.3*\squarex+0.1*\squarex, 1.5*\squarey);

\draw[draw=black, fill=blue!20,ultra thin] 
        (11.3*\squarex, 1.3*\squarey) rectangle (11.3*\squarex+0.1*\squarex, 1.4*\squarey);

\foreach \y in {10.9, 11, 11.1, ..., 11.3} {
    \draw[draw=black, fill=blue!20,ultra thin] 
        (\y*\squarex, 1.3*\squarey) rectangle (\y*\squarex+0.1*\squarex, 1.4*\squarey);
}

\foreach \y in {0.5, 0.6, 0.7, ..., 1.3} {
    \draw[draw=black, fill=blue!20, ultra thin] 
        (10.9*\squarex, \y*\squarey) rectangle (11*\squarex, \y*\squarey + 0.1*\squarey);
}

\draw[draw=black, fill=red, ultra thin] 
        (10.9*\squarex, 0.5*\squarey) rectangle (11*\squarex, 0.5*\squarey + 0.1*\squarey);

\foreach \y in {0.1, 0.2, 0.3, ..., 0.5} {
    \draw[draw=black, fill=green!20, ultra thin] 
        (10.9*\squarex, \y*\squarey) rectangle (11*\squarex, \y*\squarey + 0.1*\squarey);
}

\foreach \y in {10.4,10.5 , 10.6, ..., 10.9} {
    \draw[draw=black,fill=green!20 ,ultra thin] 
        (\y*\squarex, 0.1*\squarey) rectangle (\y*\squarex+0.1*\squarex, 0.2*\squarey);
}

\foreach \y in {-0.5, -0.4, -0.3, ..., 0} {
    \draw[draw=black,fill=green!20, ultra thin] 
        (10.4*\squarex, \y*\squarey) rectangle (10.5*\squarex, \y*\squarey + 0.1*\squarey);
}

\node at (11*\squarex + 0*\squarex, 2*\squarey) {\vdots};
\node at (11*\squarex + 0*\squarex, -0.7*\squarey) {\vdots};

% Add the label T_{s_1, r_1}
\node[anchor=north east, scale=0.6] at (10.55*\squarex, 1.5*\squarey) {\( \mathrm{T}_{s_1, r_1} \)};
%%%

\draw[draw=black, fill=red, ultra thin] 
        (10.5*\squarex, -2*\squarey) rectangle (10.6*\squarex, -1.9*\squarey);
\node[anchor=west] at (10.6*\squarex, -1.95*\squarey) {\( = \mathrm{C_i} \)};

\end{tikzpicture}
    \captionof{figure}{The path of smaller cubes is $\mathrm{V}(\mathrm{C_i})$. The path in yellow is $\mathrm{V}'(\mathrm{C_i})$. In green and blue the two paths $\mathrm{V}'_1(\mathrm{C_i})$ and $\mathrm{V}'_2(\mathrm{C_i})$, respectively, obtained from $\mathrm{V}'(\mathrm{C_i})$ that start from $\mathrm{C_i}$.}
    \label{fig:paths_from_c_i}
\end{minipage}%
\begin{minipage}{0.49\textwidth}
\begin{center}
\begin{tikzpicture}[scale=0.8]
\def\squarex{1.5} % Square width
\def\squarey{1.5}
% Draw the main square starting at x = 10
\draw[thick] (10*\squarex, -0.5*\squarey) rectangle (10*\squarex + 2*\squarex, 0*\squarey + 1.5*\squarey);

\draw[thick] (12*\squarex, -0.5*\squarey) rectangle (12*\squarex + 2*\squarex, 0*\squarey + 1.5*\squarey);

\draw[draw=black, fill=blue!20,ultra thin] 
        (11.3*\squarex, 1.4*\squarey) rectangle (11.3*\squarex+0.1*\squarex, 1.5*\squarey);

\draw[draw=black, fill=blue!20,ultra thin] 
        (11.3*\squarex, 1.3*\squarey) rectangle (11.3*\squarex+0.1*\squarex, 1.4*\squarey);

\foreach \y in {10.9, 11, 11.1, ..., 11.3} {
    \draw[draw=black, fill=blue!20,ultra thin] 
        (\y*\squarex, 1.3*\squarey) rectangle (\y*\squarex+0.1*\squarex, 1.4*\squarey);
}

\foreach \y in {0.5, 0.6, 0.7, ..., 1.3} {
    \draw[draw=black, fill=blue!20, ultra thin] 
        (10.9*\squarex, \y*\squarey) rectangle (11*\squarex, \y*\squarey + 0.1*\squarey);
}

\draw[draw=black, fill=red, ultra thin] 
        (10.9*\squarex, 0.5*\squarey) rectangle (11*\squarex, 0.5*\squarey + 0.1*\squarey);

\foreach \y in {0.1, 0.2, 0.3, ..., 0.5} {
    \draw[draw=black, fill=green!20, ultra thin] 
        (10.9*\squarex, \y*\squarey) rectangle (11*\squarex, \y*\squarey + 0.1*\squarey);
}

\foreach \y in {10.4,10.5 , 10.6, ..., 10.9} {
    \draw[draw=black,fill=green!20 ,ultra thin] 
        (\y*\squarex, 0.1*\squarey) rectangle (\y*\squarex+0.1*\squarex, 0.2*\squarey);
}

\foreach \y in {-0.5, -0.4, -0.3, ..., 0} {
    \draw[draw=black,fill=green!20, ultra thin] 
        (10.4*\squarex, \y*\squarey) rectangle (10.5*\squarex, \y*\squarey + 0.1*\squarey);
}

\foreach \y in {-0.5, -0.4, -0.3, ..., 0} {
    \draw[draw=black,fill=green!20, ultra thin] 
        (10.4*\squarex, \y*\squarey) rectangle (10.5*\squarex, \y*\squarey + 0.1*\squarey);
}

\foreach \y in {9.5, 9.6, 9.7, ..., 10.7} {
    \draw[draw=black,fill=black!20, ultra thin] 
        (\y*\squarex, 0.3*\squarey) rectangle (\y*\squarex + 0.1*\squarey, 0.3*\squarey + 0.1*\squarey);
}

\foreach \y in {-0.3, -0.2, -0.1, ..., 0.3} {
    \draw[draw=black,fill=black!20, ultra thin] 
        (10.7*\squarex, \y*\squarey) rectangle (10.7*\squarex + 0.1*\squarey, \y*\squarey + 0.1*\squarey);
}

\foreach \y in {10.7, 10.8, 10.9, ..., 12.3} {
    \draw[draw=black,fill=black!20, ultra thin] 
        (\y*\squarex, -0.3*\squarey) rectangle (\y*\squarex + 0.1*\squarey, -0.3*\squarey + 0.1*\squarey);
}

\foreach \y in {10.7, 10.8, 10.9, ..., 12.3} {
    \draw[draw=black,fill=black!20, ultra thin] 
        (\y*\squarex, -0.3*\squarey) rectangle (\y*\squarex + 0.1*\squarey, -0.3*\squarey + 0.1*\squarey);
}

\foreach \y in {-0.3, -0.2, -0.1, ..., 0.8} {
    \draw[draw=black,fill=black!20, ultra thin] 
        (12.3*\squarex, \y*\squarey) rectangle (12.3*\squarex + 0.1*\squarey, \y*\squarey + 0.1*\squarey);
}

\foreach \y in {12.3, 12.4, 12.5, ..., 13.4} {
    \draw[draw=black,fill=black!20, ultra thin] 
        (\y*\squarex, 0.8*\squarey) rectangle (\y*\squarex + 0.1*\squarey, 0.8*\squarey + 0.1*\squarey);
}

\foreach \y in {0.4, 0.5, 0.6, ..., 0.9} {
    \draw[draw=black,fill=black!20, ultra thin] 
        (13.4*\squarex, \y*\squarey) rectangle (13.4*\squarex + 0.1*\squarey, \y*\squarey + 0.1*\squarey);
}

\foreach \y in {13.4, 13.5, 13.6, ..., 14.2} {
    \draw[draw=black,fill=black!20, ultra thin] 
        (\y*\squarex, 0.4*\squarey) rectangle (\y*\squarex + 0.1*\squarey, 0.4*\squarey + 0.1*\squarey);
}

\draw[thick] (10*\squarex, -3.5*\squarey) rectangle (10*\squarex + 2*\squarex, -3*\squarey + 1.5*\squarey);

\draw[thick] (12*\squarex, -3.5*\squarey) rectangle (12*\squarex + 2*\squarex, -3*\squarey + 1.5*\squarey);

\draw[draw=black, fill=red, ultra thin] 
        (10.9*\squarex, -2.5*\squarey) rectangle (11*\squarex, -2.5*\squarey + 0.1*\squarey);

\foreach \y in {-2.9, -2.8, -2.7, ..., -2.5} {
    \draw[draw=black, fill=green!20, ultra thin] 
        (10.9*\squarex, \y*\squarey) rectangle (11*\squarex, \y*\squarey + 0.1*\squarey);
}

\foreach \y in {10.7, 10.8,10.9} {
    \draw[draw=black, fill=green!20, ultra thin] 
        (\y*\squarex, -2.9*\squarey) rectangle (\y*\squarex+0.1*\squarex, -2.8*\squarey);
}

\foreach \y in {-3.3, -3.2, -3.1,-3,-2.9} {
    \draw[draw=black, fill=black!20, ultra thin] 
        (10.7*\squarex, \y*\squarey) rectangle (10.7*\squarex + 0.1*\squarex, \y*\squarey + 0.1*\squarey);
}

\foreach \y in {10.7, 10.8, 10.9, ..., 12.3} {
    \draw[draw=black, fill=black!20, ultra thin] 
        (\y*\squarex, -3.3*\squarey) rectangle (\y*\squarex + 0.1*\squarex, -3.2*\squarey);
}

\foreach \y in {-3.3, -3.2, -3.1, ..., -2.2} {
    \draw[draw=black, fill=black!20, ultra thin] 
        (12.3*\squarex, \y*\squarey) rectangle (12.3*\squarex + 0.1*\squarex, \y*\squarey + 0.1*\squarey);
}

\foreach \y in {12.3, 12.4, 12.5, ..., 13.4} {
    \draw[draw=black, fill=black!20, ultra thin] 
        (\y*\squarex, -2.2*\squarey) rectangle (\y*\squarex + 0.1*\squarex, -2.2*\squarey + 0.1*\squarey);
}

\foreach \y in {-2.6, -2.5, -2.4, ..., -2.1} {
    \draw[draw=black, fill=black!20, ultra thin] 
        (13.4*\squarex, \y*\squarey) rectangle (13.4*\squarex + 0.1*\squarex, \y*\squarey + 0.1*\squarey);
}

\foreach \y in {13.4, 13.5, 13.6, ..., 14.2} {
    \draw[draw=black, fill=black!20, ultra thin] 
        (\y*\squarex, -2.6*\squarey) rectangle (\y*\squarex + 0.1*\squarex, -2.6*\squarey + 0.1*\squarey);
}

% Add the label T_{s_1, r_1}
\node[anchor=north east, scale=0.6] at (10.55*\squarex, 1.5*\squarey) {\( \mathrm{T}_{s_1, r_1} \)};
\node[anchor=north east, scale=0.6] at (10.55*\squarex, -3*\squarey) {\( \mathrm{T}_{s_1, r_1} \)};
%%%

\node at (14.5*\squarex + 0*\squarex, 0.7*\squarey) {\dots};
\node at (14.5*\squarex + 0*\squarex, -2.5*\squarey) {\dots};
\end{tikzpicture}
\end{center}
\captionof{figure}{The gray path is $\mathrm{H}_{r_1}^{\ell}$. $\mathcal{Q}_1$ is the path that is left after concatenation in the picture below, adding the green and gray parts.}
    \label{fig:concatenation_of_paths}
\end{minipage}
\end{figure}  
    We continue to concatenate in this way for each pair $(a_i,b_i),(a_{i+1},b_{i+1})$, for $i\le i_0$, where $i_0$ is the largest integer $i > 1$ for which $a_i<\frac{s_2+s_1+1}{2}$. Intuitively, $i_0$ is the largest integer $i > 1$ for which $(a_i,b_i)$ still did not cross the \textit{middle} of the red path. We remark that the above definition fails when $\mathcal L_X$ is a straight path between $T_{s_1,r_1}$ and $T_{s_2,r_2}$, but in this case we stop the iteration step after $\mathcal Q_1$.

    We now continue with the process for $i\ge i_0$.
    Here we continue interpreting the horizontal crossings as paths from left to right but change the orientation of the vertical crossings of good cubes to be from top to bottom. Furthermore, now instead of taking the $\ell$-th vertical crossing of good cubes we take the $(\lceil c\log n  \rceil-\ell)$-th vertical crossing of good cubes to create the concatenations (the way of enumerating here is different so that the concatenations are disjoint for different $\ell$ if we delete the good cubes from $\mathrm{V}'(\mathrm{C_i}) \cup \mathrm{V}'(\mathrm{C_j})$, see Figure ~\ref{fig:paths_of_random_cubes}). To finish the construction, once we arrive at $\mathrm{T}_{s_2,r_2}$, what is left is the concatenation to the vertical crossing of good cubes $\mathrm{V}'(\mathrm{C_j})$. Here we proceed just as we did for $\mathrm{C_i}$. Select one path that starts from $\mathrm{C_j}$ and is made from good cubes of $\mathrm{V}'(\mathrm{C_j})$ and intersects $\mathcal{Q}_{k-1}$ (we consider the last concatenation available and that passes through $\mathrm{T}_{s_2,r_2}$; there exists at least one such path of good cubes, and the last part of the path comes from a horizontal crossing). Concatenate $\mathcal{Q}_{k-1}$ to the last path selected to obtain $\mathcal{Q}_{k}$. Note that the resulting path is self-avoiding by construction, and also note that for different $\ell$, the paths are disjoint if we delete the initial and final good cubes that come from $\mathrm{V}'(\mathrm{C_i})$ and $\mathrm{V}'(\mathrm{C_j})$. We denote $\mathcal{Q}_k$ by $\mathcal{Q}_X^{\ell}$. Delete from $\mathcal{Q}_X^\ell$ the good cubes that belong to $\mathrm{V}'(\mathrm{C_i})$ and $\mathrm{V}'(\mathrm{C_j})$, except for the cubes where the concatenations with $\mathrm{V}'(\mathrm{C_i})$ and $\mathrm{V}'(\mathrm{C_j})$ were made at the beginning and the end, and denote the resulting random path by $\mathcal{L}^{\ell}_X$ (see Figure ~\ref{fig:paths_of_random_cubes}). In particular, note that the path $\mathcal{L}_X^{\ell}$ is a subpath of $\mathcal{Q}_X^{\ell}$ since the latter has $\mathrm{C_i}$ and $\mathrm{C_j}$ as endpoints (in Figure ~\ref{fig:paths_of_random_cubes}, $\mathcal{Q}_X^{\ell}$ corresponds to the blue path together with the green paths in Figure ~\ref{fig:paths_of_random_cubes} while the path $\mathcal{L}_X^{\ell}$ is only the blue path). We point out that the construction of $\mathcal{Q}_X^{\ell}$ will be used only in the proof of the upper bound of the cover time above the connectivity threshold in Section~\ref{sec:upper_bound_covTime_rgg_nonFixed_r} and Section ~\ref{sec:upper_bound_above_connectivity}. %\cmc{To check this last part after new structure is done.} 

\begin{figure}
\begin{center}
\begin{tikzpicture}[scale=1.3]

% Define square size
\def\squarex{1} % Square width
\def\squarey{1} % Square height
draw a red line below the blue line that goes through the same filled squares as the blue line
% Number of squares to the middle
\def\middle{4} % Number of squares to reach the middle from the first level
\draw[draw=transparent, fill=black!15, thick] (0*\squarex, 0*\squarey) rectangle (1*\squarex, 1*\squarey); 
\draw[draw=transparent, fill=black!15, thick] (1*\squarex, 1*\squarey) rectangle (2*\squarex , 2*\squarey);
\draw[draw=transparent, fill=black!15, thick] (1*\squarex, 0*\squarey) rectangle (2*\squarex, 1*\squarey); 
\draw[draw=transparent, fill=black!15, thick] (2*\squarex, 1*\squarey) rectangle (3*\squarex, 2*\squarey); 
\draw[draw=transparent, fill=black!15, thick] (3*\squarex, 1*\squarey) rectangle (4*\squarex , 2*\squarey); 
\draw[draw=transparent, fill=black!15, thick] (3*\squarex, 2*\squarey) rectangle (4*\squarex , 3*\squarey);
 \draw[draw=transparent, fill=black!15, thick] (4*\squarex, 2*\squarey) rectangle (5*\squarex , 3*\squarey);
\draw[draw=transparent, fill=black!15, thick] (5*\squarex, 2*\squarey) rectangle (6*\squarex , 3*\squarey);
\draw[draw=transparent, fill=black!15, thick] (6*\squarex, 2*\squarey) rectangle (7*\squarex , 3*\squarey);
\draw[draw=transparent, fill=black!15, thick] (6*\squarex, 1*\squarey) rectangle (7*\squarex , 2*\squarey);
\draw[draw=transparent, fill=black!15, thick] (7*\squarex, 1*\squarey) rectangle (8*\squarex , 2*\squarey);
\draw[draw=transparent, fill=black!15, thick] (8*\squarex, 1*\squarey) rectangle (8*\squarex , 2*\squarey);
\draw[draw=transparent, fill=black!15, thick] (8*\squarex, 0*\squarey) rectangle (9*\squarex , 1*\squarey);
\draw[draw=transparent, fill=black!15, thick] (9*\squarex, 0*\squarey) rectangle (10*\squarex , 1*\squarey);
\draw[draw=transparent, fill=black!15, thick] (8*\squarex, 1*\squarey) rectangle (9*\squarex , 2*\squarey);
% Coordinates for the first square (bottom-left of the staircase)
\coordinate (start_x) at (0, 0);

% Draw the increasing part of the staircase
\foreach \level in {0,...,4} {
    \foreach \i in {0,...,\level} {
        \draw[thick] (\level*\squarex, \i*\squarey) rectangle (\level*\squarex + \squarex, \i*\squarey + \squarey);
    }
}
\draw[thick] (5*\squarex, 0*\squarey) rectangle (5*\squarex + \squarex, 0*\squarey + \squarey);
\draw[thick] (5*\squarex, 1*\squarey) rectangle (5*\squarex + \squarex, 1*\squarey + \squarey);
\draw[thick] (5*\squarex, 2*\squarey) rectangle (5*\squarex + \squarex, 2*\squarey + \squarey);
\draw[thick] (5*\squarex, 3*\squarey) rectangle (5*\squarex + \squarex, 3*\squarey + \squarey);
\draw[thick] (5*\squarex, 4*\squarey) rectangle (5*\squarex + \squarex, 4*\squarey + \squarey);

\draw[thick] (6*\squarex, 0*\squarey) rectangle (6*\squarex + \squarex, 0*\squarey + \squarey);
\draw[thick] (6*\squarex, 1*\squarey) rectangle (6*\squarex + \squarex, 1*\squarey + \squarey);
\draw[thick] (6*\squarex, 2*\squarey) rectangle (6*\squarex + \squarex, 2*\squarey + \squarey);
\draw[thick] (6*\squarex, 3*\squarey) rectangle (6*\squarex + \squarex, 3*\squarey + \squarey);

\draw[thick] (7*\squarex, 0*\squarey) rectangle (7*\squarex + \squarex, 0*\squarey + \squarey);
\draw[thick] (7*\squarex, 1*\squarey) rectangle (7*\squarex + \squarex, 1*\squarey + \squarey);
\draw[thick] (7*\squarex, 2*\squarey) rectangle (7*\squarex + \squarex, 2*\squarey + \squarey);

\draw[thick] (8*\squarex, 0*\squarey) rectangle (8*\squarex + \squarex, 0*\squarey + \squarey);
\draw[thick] (8*\squarex, 1*\squarey) rectangle (8*\squarex + \squarex, 1*\squarey + \squarey);

\draw[thick] (9*\squarex, 0*\squarey) rectangle (9*\squarex + \squarex, 0*\squarey + \squarey);

% First square at the bottom-left
\coordinate (A1) at (0, 0);
\coordinate (B1) at (\squarex, 0);
\coordinate (C1) at (\squarex, \squarey);
\coordinate (D1) at (0, \squarey);
% Label the first square
\node[anchor=north west, scale=0.8] at (D1) {$\mathrm{T}_{s_1,r_1}$};
% Second labeled square (at the middle of the stair)
\coordinate (A2) at (4*\squarex, 0);
\coordinate (B2) at (4*\squarex + \squarex, 0);
\coordinate (C2) at (4*\squarex + \squarex, \squarey);
\coordinate (D2) at (4*\squarex, \squarey);
% Label the second square
\node[anchor=north west, scale=0.8] at (D2) {$T_{s_1, \frac{r_2+r_1+1}{2}}$};
% Third labeled square (at the far right)
\coordinate (A3) at (8*\squarex, 0);
\coordinate (B3) at (9*\squarex + \squarex, 0);
\coordinate (C3) at (8*\squarex + \squarex, \squarey);
\coordinate (D3) at (9*\squarex, \squarey);

% Label the third square
\node[anchor=north west, scale=0.8] at (D3) {$\mathrm{T}_{s_2,r_2}$};

% Fill the squares intersected by the red lines with translucent red

% Blue line

\draw[green, ultra thick] (0.5,0.2) -- (0.5,0.7);
\draw[blue, ultra thick] (0.5,0.7) -- (1.2,0.7);
\draw[blue, ultra thick] (1.2,0.7) -- (1.2,1.8);
\draw[blue, ultra thick] (1.2,1.8) -- (3.3,1.8);
\draw[blue, ultra thick] (3.3,1.8) -- (3.3,2.5);
\draw[blue, ultra thick] (3.3,2.5) -- (6.7,2.5);
\draw[blue, ultra thick] (6.7,2.5) -- (6.7,1.3);
\draw[blue, ultra thick] (6.7,1.3) -- (8.4,1.3);
\draw[blue, ultra thick] (8.4,1.3) -- (8.4,0.6);
\draw[blue, ultra thick] (8.4,0.6) -- (9.5,0.6);
\draw[green, ultra thick] (9.5,0.1) -- (9.5,0.6);

\draw[red, ultra thick] (0.5,0.5) -- (1.6,0.5);
\draw[red, ultra thick] (1.6,0.5) -- (1.6,1.3);
\draw[red, ultra thick] (1.6,1.3) -- (3.5,1.3);
\draw[red, ultra thick] (3.5,1.3) -- (3.5,2.3);
\draw[red, ultra thick] (3.5,2.3) -- (6.3,2.3);
\draw[red, ultra thick] (6.3,2.3) -- (6.3,1.15);
\draw[red, ultra thick] (6.3,1.15) -- (8.2,1.15);
\draw[red, ultra thick] (8.2,1.15) -- (8.2,0.4);
\draw[red, ultra thick] (8.2,0.4) -- (9.5,0.4);

\node[draw, fill=red, shape=rectangle, minimum size=8pt, scale=0.6, label=right:{$\mathrm{C_j}$}] at (9.5,0.1) {};
\node[draw, fill=red, shape=rectangle, minimum size=8pt, scale=0.6, label=right:{$\mathrm{C_i}$}] at (0.5,0.2) {};

\end{tikzpicture}
\end{center}
\captionof{figure}{The red path is $\mathcal{L}_X^{\ell_1}$, and the blue path is $\mathcal{L}_X^{\ell_2}$, they are disjoint as a consequence of the enumeration introduced for the long crossings. The green part (subset of $\mathrm{V}'(\mathrm{C_i}) \cup \mathrm{V}'(\mathrm{C_j})$) allows us to reach $\mathrm{C_i}$ and $\mathrm{C_j}$.}
\label{fig:paths_of_random_cubes}
\end{figure}

\section{Lower bound on the cover time}\label{sec:lower_bound_fixed_radius}

Recall that $\mu=\mu(d)$ is the largest constant so that any two vertices of $\mathcal{G}_n$ in adjacent cubes $\mathrm{C}$ (that is, two cubes sharing a $d-1$-dimensional face) of sidelength $r/\mu$ are connected by an edge. We are going to use the tessellation into cubes $\mathrm{C_i}$ as described in Section~\ref{sec:renormalization} of side length $Mr$ and also use the aforementioned finer tessellation of $\Lambda_n$ into cubes $\mathrm{Q_i}$, of sidelength $\frac{r}{\mu}$. The cube with index $\mathrm{i} \in \mathbb{Z}^d$ is given by
\begin{equation}
    \mathrm{Q_i} := \mathrm{i} \frac{r}{\mu} + \left[ 0, \frac{r}{\mu} \right)^d. \hspace{2mm}
\end{equation}

Let $\phi=\phi(d)$ be the smallest integer such that any pair of vertices of $\cG_n$, contained in cubes $\mathrm{Q_i}$ and $\mathrm{Q_j}$, is not connected by an edge if $|\mathrm{i}-\mathrm{j}|_1>\phi$. For example,  $\phi(2)=3$. Let $x \in \cG_n$, with $x\in \mathrm{Q_j}$. Recall that we use the notation $\mathrm{j}=(\mathrm{j}_1,...,\mathrm{j}_d)$.
We say that $x$ is \emph{$m$-special} if the following two conditions hold:
\begin{itemize}
    \item Each of the cubes $\mathrm{Q_k}$, with $\mathrm{k}=\mathrm{j}+e_1s$ and $1\leq s\leq m$, has exactly one point of the PPP. We call the set of cubes $\{\mathrm{Q_k}\}$ by slight abuse of notation \emph{$m$-pending path}  with root at $x$ (note that the graph induced by a pending path is not necessarily a path, since vertices in two cubes at distance $2$ in the cube metric might also be connected by an edge, see Figure~\ref{fig:especialpath}).
    \item All cubes at $1$-norm at most $\phi$ from some cube $Q_k$ of the previous condition, excluding the cubes $Q_k$ themselves, have zero vertices of $\cG_n$.
    
\end{itemize}

\begin{figure}[htbp!]
\centering
    \begin{tikzpicture}
        \draw (0, 0) -- (6, 0);
        \draw (0, 1) -- (6, 1);
        \draw (1, 2) -- (6, 2);
        \draw (1, -1) -- (6, -1);
        \draw (1, -2) -- (6, -2);
        \draw (1, 3) -- (6, 3);

        \draw (8, 3) -- (12, 3);
        \draw (8, 2) -- (12, 2);
        \draw (8, 1) -- (12, 1);
        \draw (8, 0) -- (12, 0);
        \draw (8, -1) -- (12, -1);
        \draw (8, -2) -- (12, -2);

        \draw (8, 3) -- (8, -2);
        \draw (9, 3) -- (9, -2);
        \draw (10, 3) -- (10, -2);
        \draw (11, 3) -- (11, -2);
        \draw (12, 3) -- (12, -2);

        \foreach \x in {0,...,6} {
            \draw (\x, 0) -- (\x,1);
        }
        \foreach \x in {1,...,6} {
            \draw (\x, 1) -- (\x,2);
            \draw (\x, 0) -- (\x,-1);
            \draw (\x, -1) -- (\x,-2);
            \draw (\x, 2) -- (\x,3);
        }
        \node[red, circle, fill=red, scale=0.4] at (0.2,0.3) {};
        \node[blue, circle, fill=blue, scale=0.4] at (1.7,0.75) {};
        \node[blue, circle, fill=blue, scale=0.4] at (2.35,0.32) {};
        \node[blue, circle, fill=blue, scale=0.4] at (3.12,0.92) {};
        \node[blue, circle, fill=blue, scale=0.4] at (4.1,0.78) {};
        \node[blue, circle, fill=blue, scale=0.4] at (5.5,0.21) {};
        \node[blue, circle, fill=blue, scale=0.4] at (9.72,0.21) {};
        \node[blue, circle, fill=blue, scale=0.4] at (8.2,0.71) {};

        \node[scale=2] at (6.75, 1.5) {$\cdot$};
        \node[scale=2] at (7, 1.5) {$\cdot$};
        \node[scale=2] at (7.25, 1.5) {$\cdot$};

        \node[scale=2] at (6.75, -0.5) {$\cdot$};
        \node[scale=2] at (7, -0.5) {$\cdot$};
        \node[scale=2] at (7.25, -0.5) {$\cdot$};

        \draw[dash pattern=on 4pt off 4pt, thick,blue] (5.5,0.21) -- (6.31,0.55);
        \draw[dash pattern=on 4pt off 4pt, thick,blue] (5.5,0.21) -- (4.1,0.78);
        \draw[dash pattern=on 4pt off 4pt, thick,blue] (5.5,0.21) -- (4.1,0.78);
        \draw[dash pattern=on 4pt off 4pt, thick,blue] (3.12,0.92) -- (4.1,0.78);
        \draw[dash pattern=on 4pt off 4pt, thick,blue] (2.35,0.32) -- (4.1,0.78);
        \draw[dash pattern=on 4pt off 4pt, thick,blue] (2.22,0.32) -- (1.7,0.75);
        \draw[dash pattern=on 4pt off 4pt, thick,blue] (1.7,0.75) -- (3.12,0.92);
        \draw[dash pattern=on 4pt off 4pt, thick,blue] (1.7,0.75) -- (0.2,0.3);
        \draw[dash pattern=on 4pt off 4pt, thick,blue] (8.2,0.71) -- (9.72,0.21);
        \draw[dash pattern=on 4pt off 4pt, thick,blue] (8.2,0.71) -- (7.8,0.1);

        \draw[dash pattern=on 4pt off 4pt, thick,blue] (0.2,0.3) -- (2.35,0.32);
        \draw[dash pattern=on 4pt off 4pt, thick,blue] (2.35,0.32) -- (3.12,0.92);

        \node at (9.72, 0.45) {$y$};
        \node at (0.2, 0.5) {$x$};
    \end{tikzpicture}
    \caption{A pending path in 2 dimensions with root at $x$ ($x$ is thus $m$-special for some $m$); all  cubes surrounding the path are empty; $y$ is the vertex associated to $x$ that belongs to $\mathcal{R}$ (as defined in the proof). Observe that a vertex in the path could be connected by an edge to more than 2 vertices.}
    \label{fig:especialpath}

\end{figure}

Now, tessellate $\Lambda_{n}$ into cubes $\mathrm{C}$ of side length $Mr$ as in Section~\ref{sec:renormalization}, such that $Mr/2$ is large enough and also an integer multiple of $r/\mu$. Recall that by definition of the corresponding tessellation $\mathcal{K}_n$ we only consider cubes of sidelength $Mr$ fully contained in $\Lambda_n$. $M$ will be chosen large enough (depending on the dimension); it might happen that close to the boundary some part of $\Lambda_n$ remains uncovered by the tessellation, but this causes no problem with the proofs of this section.

\begin{comment}
 Note that in this way, any cube $\mathrm{Q}$ is either totally contained in some cube $\mathrm{C}$, or, by construction, the center of $\mathrm{C}$ is at distance smaller than $Mr$ from the boundary of $\Lambda_n$ \dmc{Is this obvious to see, for me it is only a bit larger than $Mr$. Imagine the large cube is just a bit outside $\Lambda_n$, and the last cube contained is basically at distance $Mr+r/\mu$. Please explain}\cmc{I think this makes reference to the center of the cube $\mathrm Q$. In that case, what you say is correct. The center could be at distance $Mr+r/2\mu-\eps$. It is also not clearn to me why this is needed}\dmc{Indeed, I will recheck, but then leave out}($M$ depends on the dimension). \dmc{The next two sentences should be eliminated, I think}\cmc{I think it could be clarifying to keep them, but I have no strong opinion}Such a choice of $M$ only depends on the dimension \cm{once we require $M>M'$ for some large $M'$}. Recall that, by definition of the tessellation $\mathcal{K}_n$ into cubes $\mathrm{C}$, we only consider cubes of sidelength $Mr$ fully contained in $\Lambda_n$; it might happen that close to the boundary some part of $\Lambda_n$ remains uncovered by the tessellation, but this causes no problem with the proofs of this section.
\end{comment}

The following lemma tells that for $c_1>0$ small enough, a.a.s.\ we will find many $\left\lfloor c_1 \frac{\log n}{r^d} \right\rfloor$-pending paths, whose root is a vertex of the giant of $ \cG_n$.
\begin{lemma}\label{lem:special_paths}
    Let $d\geq 2$, $\eps>0$. There exist small positive constants $c_1, \xi, \gamma$ ($\xi$ depends only on $d,c_1, \gamma$), such that if $(1+\eps)r_g \le r < \gamma r_c$, the event 
    \begin{equation}
        \left|\left\{x\in L_1: x \text{ is } \left(\left\lfloor c_1 \frac{\log n}{r^d} \right\rfloor\right)-\text{special}\geq n^{\xi}\right\}\right|
    \end{equation}
    occurs a.a.s.
\end{lemma}
\begin{proof} 
In order to make sure that we can define a $\left\lfloor c_1 \frac{\log n}{r^d} \right\rfloor$-pending path in $\Lambda_n$, we would like its root to be inside $\Lambda_{m_n}$, where
    \begin{equation}
    m_n=n^{1/d}-\frac{r}{\mu}\left\lfloor c_1\frac{\log n}{r^d} \right\rfloor. 
\end{equation}
    Let $\mathcal{K}_{m_n}^2$ be the 2-dimensional slice where the first two coordinates are not fixed and all others are equal to 0 (see Figure~\ref{fig:bigcube}).
     Observe that by Theorem~\ref{thm:giantcomponents}, a.a.s. all vertices of $\mathcal{G}_n$ that belong to a horizontal crossing of $\mathcal{K}_{m_n}^2$ belong to $L_1$: indeed, each of these crossings has Euclidean diameter $\Omega(m_n/r_c)$, and the latter is much larger than $c\log^2n$. %\cm{\st{Since we are interested in a result that holds a.a.s., we may condition on this event.}}} \cmc{We can only work conditioning on the fact that there a lot of crossings, as this does not change the fact that what happens inside and outside $\mathcal K$ are independent. I would remove this and mention it at the end.}\dmc{If you have a better strategy than conditioning, please go ahead. We could just say that bad events all do not happen, you mean without conditioning?}\cmc{Yes, the last that you mention is what I had in mind. I already implemented it.}
     
     By Lemma~\ref{lem:crossingsstrips}, a.a.s., there are at least $c_2 m_n$ horizontal disjoint crossings of good cubes inside $\mathcal{K}_{m_n}^2$ for some $c_2 > 0$, and we may assume this below. Each horizontal crossing of good cubes contains a crossing of $\mathcal G_n$ inside $\Lambda_{m_n}$ (see Remark \ref{rem:domination_largest_component_good_cubes}). Recall also that the right endvertices of the crossing in $\mathcal G_n$ can be chosen such that they are at distance less or equal than $r/\mu$ from the right boundary of the slice $\mathcal{K}_{m_n}^2$. We denote by $\mathcal E$ the set of vertices in $\mathcal K^2_{m_n}$ at distance at most $r/\mu$ from the right boundary of the latter set and that are an endpoint of some horizontal crossing.

    Let $x\in \mathrm{Q_i}$ and $y \in \mathrm{Q_j}$ be vertices of $ \mathcal{G}_n$, with $\mathrm{i}_1=\mathrm{j}_1$ (that is, with first coordinates being equal). By definition of $\phi$ and special vertex, if $|\mathrm{i}-\mathrm{j}|> 2\phi$ (see the definition of $\phi$ above) then the events of the vertices $x,y$ being $\left(\left\lfloor c_1\frac{\log n}{r^d} \right\rfloor\right)$-special are independent. Also, note that $\mathcal{G}_n \cap \mathcal K_{m_n}^2$ is independent of $\mathcal{G}_n\cap (\Lambda_n\setminus \mathcal{K}_{m_n}^2)$. In particular, the event that a vertex $x \in \cG_n$ that is inside of $\mathcal{K}^{2}_{m_n}$, is at distance less than $r/\mu$ from the right boundary of $\mathcal{K}^2_{m_n}$ and is an endvertex of some horizontal crossing in $\mathcal K^2_{m_n}$ does not have an effect on the probability to have a $\left\lfloor c_1 \frac{\log n}{r^d} \right\rfloor$-pending path in $\Lambda_n \setminus \mathcal{K}^{2}_{m_n}$ attached to it. 
    Using the fact that $\lfloor c_1 \log n / r^d \rfloor \ge 1$, the probability for the latter event to happen is equal to 
    \begin{equation}
        (e^{-(r/\mu)^d} (r/\mu)^d)^{\left\lfloor c_1\frac{\log n}{r^d}\right\rfloor} (e^{- (r/\mu)^d})^{c'\left\lfloor c_1\frac{\log n}{r^d}\right\rfloor}(e^{- (r/\mu)^d})^{c''} \ge n^{-\rho},
    \end{equation}
    where $c', c''$ are constants that depend only on the dimension and $\rho$ is a constant that depends on $c_1, d, \eps, \gamma$. Observe that we can choose $c_1$ small enough such that $\rho < 1/d$, and then we can choose $\gamma$ small enough such that $c_1\frac{\log n}{r^d}$ is always greater than 1.
    Hence, recalling that the number of horizontal crossings in $\mathcal K^2_{m_n}$ is a.a.s.\ at least $c_2m_n$, and noting that a  constant fraction (for each element in $\mathcal{E}$ we may have to eliminate at most a dimension-dependent constant number of elements that are too close) of the vertices in $\mathcal E$ gives rise to independent events of being $\lfloor c_1 \log n / r^d \rfloor$-special, we have under the event that there are at least $c_2 m_n$ crossings in $\mathcal K^2_{m_n}$,
    \begin{equation}
    \begin{split}
        \mathbb{P} \left( |\{x \in \mathcal E: x\text{ is }(\lfloor c_1\log n/r^d \rfloor)-\mbox{special}\} |< n^{\xi}  
\right) \leq \mathbb{P} \left( Bin\left(c_4 n^{1/d},n^{-\rho}\right) < n^{\xi}\right).
    \end{split}
    \end{equation}
   By choosing $\xi =(1/d-\rho)/2$ and applying Chernoff bounds, and noting that $\mathcal E\subset L_1$ a.a.s.\ (each horizontal crossing in $\mathcal K^2_{m_n}$ belongs to $L_1$ a.a.s.\ and  each vertex in $\mathcal E$ is connected to some crossing by construction), the theorem follows.
\end{proof}

\begin{theorem}\label{thm:lower_bound_cov_r_not_fixed}
    Let $\eps$ be an arbitrarily small positive constant. Let $(1+\eps)r_g \le r \le (1-\eps) r_c$. Then a.a.s.,
    \begin{equation}
        \tau_{cov}(L_1) \geq \delta n \log^2 n,
    \end{equation}
    where $\delta$ is a positive constant that only depends on $\eps$.
\end{theorem}
\begin{proof}
\textbf{Case 1: }$r<\gamma' r_c$, where $\gamma'$ is a small enough constant to be fixed later. Assume also that $\gamma' \leq \gamma$, where $\gamma$ is the same constant from Lemma~\ref{lem:special_paths}. Define as before
\begin{equation}
    m_n=n^{1/d}-\frac{r}{\mu}\left\lfloor \frac{c_1 \log n}{r^d} \right\rfloor.
\end{equation}
Let $\mathcal{P}$ be the set of $\left\lfloor \frac{c_1 \log n}{r^d}\right\rfloor$-special vertices of $\Lambda_{m_n}$ that are contained within $L_1$ (restricted to $\Lambda_{m_n})$. Denote as before by $\mathcal{E}$ the set of vertices of $L_1$ that are in the rightmost cube of a pending path with root at a vertex in $\mathcal{P}$ (see Figure~\ref{fig:especialpath}):
\begin{equation}
    \mathcal{E}=\{y \in L_1: y\in \mathrm{C_j}, \, \mathrm{j}=\mathrm{i}+\lfloor c_1\log n/r^d \rfloor, x\in  \mathrm{Q_i} \text{ and } x\in \mathcal{P}\}
\end{equation}
Note that $|\mathcal{P}|=|\mathcal{E}|$. By Lemma~\ref{lem:special_paths}, $|\mathcal{E}|\geq n^{\xi}$ a.a.s., and by construction, $\mathcal{E}\subseteq L_1$. Now, we give a lower bound on the effective resistance between any pair of vertices of $\mathcal{E}$. Take $x,y \in \mathcal{E}$, and let us construct edge-disjoint edge-cutsets separating $x$ from $y$. Let $x\in \mathrm{Q_i}$. Let $\sigma(d)=\sigma$ be the
smallest integer such that when taking out $\sigma$ consecutive cubes of a pending path, any path of vertices in $\mathcal{G}_n$ connecting the endpoints of the path must have at least one edge with both endpoints inside the $\sigma$ many consecutive cubes. For $1 \leq k\leq \left\lfloor\frac{1}{\sigma}\left\lfloor \frac{c_1 \log n}{r^d} \right\rfloor\right\rfloor - 1$, let $\Pi_k$ be the set of edges of $L_1$  with both endpoints inside two different cubes $\mathrm{Q_j}, \mathrm{Q_j'}$ for some $j,j'$ satisfying $\mathrm{j}=\mathrm{i}-e_1s$ and $\sigma (k-1)\leq s< \sigma k$ and $\mathrm{j'}=\mathrm{i}-e_1s'$ and $\sigma (k-1)\leq s'< \sigma k$. Choose $\gamma'$ small enough such that $\left\lfloor\frac{1}{\sigma}\left\lfloor \frac{c_1 \log n}{r^d} \right\rfloor\right\rfloor$ is greater or equal to 2. Note that $|\Pi_k|\leq \binom{\sigma}{2}$, as there is only one vertex per cube $\mathrm{Q_j}$. Moreover, observe that $\Pi := \{\Pi_k\}_k$ is a family of disjoint edge-cutsets between $x,y$, as any path between both vertices has to contain an edge belonging to $\Pi_k$ (by definition of $\sigma$), for any $k$. By Lemma~\ref{lem:Nash-Williamas} we then have
    \begin{equation}
        \mathcal{R}(x\longleftrightarrow y,L_1)\geq \sum_{k=1}^{\left\lfloor\frac{1}{\sigma}\left\lfloor \frac{c_1 \log n}{r^d} \right\rfloor\right\rfloor-1} \frac{1}{\binom{\sigma}{2}} \geq \delta_1 \frac{\log n}{r^d},
    \end{equation}
    for some $\delta_1>0$. Using Lemma ~\ref{lem:sizergg} together with the formula of Theorem~\ref{thm:lowerboundcov}  we obtain the desired result for this case.

%\dmc{Esta bien esta demostracion, si.}
    \textbf{Case 2: $\gamma' r_c \leq r\leq (1-\eps)r_c$.} In this case we will use a different strategy: we will show that a.a.s.\ the number of vertices of $L_1$ with degree 1 is $\Omega(n^{\xi'})$ for some $\xi' > 0$. Denote by $R_n$ the number of vertices of $\mathcal G_n$ that do not belong to the giant $L_1$, and by $S_n$ the number of isolated vertices of $\mathcal G_n$. Denote also by $D_1,D_{1,\mathrm{in}},D_{1,\mathrm{out}}$, the number of vertices with degree 1 of $\mathcal G_n$, the number of vertices of degree 1 in $L_1$, and the number of vertices with degree 1 in $\mathcal G_n$ but not in $L_1$, respectively. Clearly, we have
    \[
    D_{1,\mathrm{out}} \le R_n-S_n.
    \]
    By \cite[Theorem 1.1]{penrose2026components}, a.a.s., $R_n-S_n=o(S_n)$ and  $S_n=\Theta(ne^{-V_dr^d})$, with $V_d$ being the volume of the $d$-dimensional ball with unit radius.
    Furthermore, \[
    \mathbb E[D_1] =\Theta(nr^de^{-\theta_d r^d}), 
    \] and by a standard second moment method, we also have a.a.s.
     $D_{1}=\Theta(\mathbb E[D_1])$. Note that thus a.a.s.\  $S_n=o(D_1)$. Hence, a.a.s.,
    \[D_{1,\mathrm{in}}=D_1-D_{1,\mathrm{out}} \ge D_1-o(S_n) =\Theta(D_1).\]
    Recalling that $r_c= (1+o(1))(\frac{\log n}{\theta_d})^{1/d}$, by our assumption on $r$,  we deduce that $\theta_d r^d<c \log n$ for some $c < 1$. Hence, a.a.s.\ $D_{1,\mathrm{in}}\ge n^{\xi'}$ for some $\xi' > 0$. 
    To conclude, let $V'$ then be the set of vertices of $L_1$ of degree exactly $1$. For any $x,y \in V', x\neq y$, taking out the only edge incident to say $x$, we obtain an edge-cutset of size $1$ between $x$ and $y$, and thus, by Lemma~\ref{lem:Nash-Williamas}, 
    $$\mathcal{R}(x\longleftrightarrow y,L_1)\geq 1.$$ Since this holds for any different $x,y \in V'$, using Lemma~\ref{lem:sizergg}, by Theorem~\ref{thm:lowerboundcov}, a.a.s.\ there exist some $\delta_2, \delta_3> 0$, such that 
    $$
    \tau_{cov} \ge \delta_2 nr^d \log n \ge \delta_3 n \log^2 n, 
    $$
    and this case is finished as well, and by choosing $\delta=\min\{\delta_1, \delta_3\}$, the proof of the theorem is finished.
    \end{proof}

\section{Proof of the upper bound on the cover time with fixed radius}
\label{sec:upperboundfixedradius}
The goal of this section is to prove the following theorem:
\begin{theorem}\label{thm:upperboundcov_fixed_r}
Let $d \ge 2$, and let $\eps > 0$ be arbitrarily small. Assume that $(1+\eps)r_g \le r \le r_0$, for some fixed $r_0$. A.a.s., there is a constant $\gamma(d,r_0,\eps)>0$ such that 
\begin{equation}
\tau_{cov}(L_1) \leq \gamma n\log^2 n. 
\end{equation}
\end{theorem}
\begin{proof}
    Since we use the construction given in Section~\ref{sec:randompaths}, we recall our initial assumptions made on the good cubes $\mathrm{C_i}$ and $\mathrm{C_j}$: both cubes belong to the same horizontal strip and they both belong to parallelepipeds not at the boundary of the slice and separated by an even number of parallelepipeds. We will lift these assumptions later.

    Consider any two vertices $x,y \in \cG_n$ that belong to the good cubes $\mathrm{C_i}$ and $\mathrm{C_j}$ respectively, and both belong to the largest connected component of the renormalized graph of good cubes. As the good cubes $\mathrm{C_i}$ and $\mathrm{C_j}$ belong to long crossings of good cubes (in the renormalized graph) which at the same time contain long crossings of $\cG_n$ of length $\Theta(n^{1/d})$ (see item 2) of Remark~\ref{rem:domination_largest_component_good_cubes}), by Theorem~\ref{thm:giantcomponents}, a.a.s.\ $x,y$ then also belong to $L_1$. 

    As explained before, without loss of generality we may suppose that $\mathrm{V}_1'(\mathrm{C_i})$ intersects $\mathcal{L}_X^{\ell}$. Denote by $\mathrm{C}^1_{\ell}$ the only cube of $\mathcal{L}_X^{\ell}$ that intersects $\mathrm{V}_1'(\mathrm{C_i})$. In the same way, we denote by $\mathrm C^2_\ell$ the intersection between $\mathcal L^\ell_X$ and $\mathrm V_1' (\mathrm C_j)$ (we assumed without loss of generality that such intersection is non-empty).

    Select an induced path of $\mathcal G_n$ that starts at $x$, is contained in $\mathrm{V}_1'(\mathrm{C_i})$ and has its final vertex at the last cube of $\mathrm{V}_1'(\mathrm{C_i})$.  Denote this path by $\mathcal{V}^1_{x,\ell}$. We do the same with $\mathrm{V}_2'(\mathrm{C_i})$ and denote the path by $\mathcal{V}^2_{x,\ell}$. We proceed in the same way for the vertex $y$, denoting the respective paths by $\mathcal V^1_{y,\ell}$ and $\mathcal V^2_{y,\ell}$, respectively.
    
    Since $\mathcal{L}_X^{\ell}$ is a path of good cubes that starts at $\mathrm{C}^1_\ell$ and ends at $\mathrm{C}^2_\ell$, we may obtain an induced path, totally contained within $\mathcal{L}_X^{\ell}$,  of $\cG_n$ that starts at some vertex of $\mathcal{V}^1_{x,\ell}$ inside $\mathrm{C}^1_\ell$ and finishes at some vertex of $\mathcal{V}^1_{y,\ell}$ inside $\mathrm{C}^2_\ell$. Note that since all cubes are good,  we can use the giant components of the good cubes, and we may choose a path that intersects $\mathcal{V}^1_{x,\ell}$ and $\mathcal{V}^1_{y,\ell}$ only at its initial and final vertices, respectively. We denote this path by $J^{\ell}_{X}$, see Figure ~\ref{fig:path_from_random_path_of_cubes}.

    We concatenate $\mathcal{V}^1_x$, $J^{\ell}_{X}$ and $\mathcal{V}^1_y$ in this order, such that the resulting path is self-avoiding (see Figure ~\ref{fig:path_from_random_path_of_cubes} that explains the first part of the concatenation of blue and red). Note that the vertices of concatenation could be different to the initial and final vertices of $J^\ell_X$, but by construction there is at least one point of intersection between the paths. We denote the latter path by $\mathcal J^\ell_X$. 
\begin{figure}[htb]
\begin{center}
\begin{tikzpicture}[scale=0.8]
\def\squarex{3.5} % Square width
\def\squarey{3.5} % Square height
    \draw[thick] (10*\squarex, -0.5*\squarey) rectangle (10*\squarex + 2*\squarex, 0*\squarey + 1.5*\squarey);

\draw[draw=black, fill=blue!20,ultra thin] 
        (11.3*\squarex, 1.4*\squarey) rectangle (11.3*\squarex+0.1*\squarex, 1.5*\squarey);

\draw[draw=black, fill=blue!20,ultra thin] 
        (11.3*\squarex, 1.3*\squarey) rectangle (11.3*\squarex+0.1*\squarex, 1.4*\squarey);

\foreach \y in {10.9, 11, 11.1, ..., 11.3} {
    \draw[draw=black, fill=blue!20,ultra thin] 
        (\y*\squarex, 1.3*\squarey) rectangle (\y*\squarex+0.1*\squarex, 1.4*\squarey);
}

\foreach \y in {0.5, 0.6, 0.7, ..., 1.3} {
    \draw[draw=black, fill=blue!20, ultra thin] 
        (10.9*\squarex, \y*\squarey) rectangle (11*\squarex, \y*\squarey + 0.1*\squarey);
}

\draw[draw=black, fill=red, ultra thin] 
        (10.9*\squarex, 0.5*\squarey) rectangle (11*\squarex, 0.5*\squarey + 0.1*\squarey);

\foreach \y in {0.1, 0.2, 0.3, ..., 0.5} {
    \draw[draw=black, fill=green!20, ultra thin] 
        (10.9*\squarex, \y*\squarey) rectangle (11*\squarex, \y*\squarey + 0.1*\squarey);
}

\foreach \y in {10.4,10.5 , 10.6, ..., 10.9} {
    \draw[draw=black,fill=green!20 ,ultra thin] 
        (\y*\squarex, 0.1*\squarey) rectangle (\y*\squarex+0.1*\squarex, 0.2*\squarey);
}

\foreach \y in {-0.5, -0.4, -0.3, ..., 0} {
    \draw[draw=black,fill=green!20, ultra thin] 
        (10.4*\squarex, \y*\squarey) rectangle (10.5*\squarex, \y*\squarey + 0.1*\squarey);
}

\draw[fill=black] (10.95*\squarex + 0*\squarex, 0.55*\squarey) circle (2pt);
\node at (11*\squarex + 0*\squarex + 0.3, 0.55*\squarey) {$x$};

% Add the label T_{s_1, r_1}
\node[anchor=north east, scale=1] at (10.4*\squarex, 1.5*\squarey) {\( \mathrm{T}_{s_1, r_1} \)};
%%%

\draw[draw=black, fill=red, ultra thin] 
        (10*\squarex, -0.9*\squarey) rectangle (10.1*\squarex, -0.8*\squarey);
\node[anchor=west] at (10.1*\squarex, -0.85*\squarey) {\( = \mathrm{C_i} \)};

\draw[draw=black, fill=green!20, ultra thin] 
        (10.65*\squarex, -0.9*\squarey) rectangle (10.75*\squarex, -0.8*\squarey);
\node[anchor=west] at (10.75*\squarex, -0.85*\squarey) {\( \in \mathrm{V}_1'(\mathrm{C_i}) \)};

\draw[draw=black, fill=black!20, ultra thin] 
        (11.5*\squarex, -0.9*\squarey) rectangle (11.6*\squarex, -0.8*\squarey);
\node[anchor=west] at (11.6*\squarex, -0.85*\squarey) {\( \in \mathcal{L}_X^{\ell} \)};

 \draw[ultra thick,blue] plot [smooth] coordinates {(10.5*\squarex, -1.1*\squarex)  (10.65*\squarex, -1.1*\squarex) };
 \node[anchor=west] at (10.65*\squarex, -1.1*\squarey) {\( = \mathcal{V}_x^1 \)};

 \draw[ultra thick,red] plot [smooth] coordinates {(11*\squarex, -1.1*\squarex)  (11.15*\squarex, -1.1*\squarex) };
 \node[anchor=west] at (11.15*\squarex, -1.1*\squarey) {\( = J^{\ell}_{X} \)};

\foreach \y in {-0.3, -0.2, -0.1, ..., 0.1} {
    \draw[draw=black,fill=black!20, ultra thin] 
        (10.7*\squarex, \y*\squarey) rectangle (10.7*\squarex + 0.1*\squarey, \y*\squarey + 0.1*\squarey);
}

\foreach \y in {10.7, 10.8, 10.9, ..., 12.3} {
    \draw[draw=black,fill=black!20, ultra thin] 
        (\y*\squarex, -0.3*\squarey) rectangle (\y*\squarex + 0.1*\squarey, -0.3*\squarey + 0.1*\squarey);
}

\foreach \y in {10.7, 10.8, 10.9, ..., 12.3} {
    \draw[draw=black,fill=black!20, ultra thin] 
        (\y*\squarex, -0.3*\squarey) rectangle (\y*\squarex + 0.1*\squarey, -0.3*\squarey + 0.1*\squarey);
}

 \draw[ultra thick,blue] plot [smooth] coordinates {(10.95*\squarex, 0.55*\squarex)  (10.9*\squarex, 0.12*\squarex)  (10.45*\squarex, 0.13*\squarey)  (10.47*\squarex, -0.5*\squarex)};

\draw[ultra thick, red] plot [smooth] coordinates {
    (10.75*\squarex, 0.11*\squarex)
    (10.75*\squarex, -0.23*\squarex)  
};

\draw[ultra thick, red] plot [smooth] coordinates {
    (10.75*\squarex, -0.23*\squarex)  
    (12.5*\squarex, -0.23*\squarex) 
};

\draw[->, thick, bend left=45] 
    (10.5*\squarex,0.4*\squarex) to node[pos=0.2, left] {$\mathrm{C}^1_\ell$} (10.73*\squarex,0.2*\squarex);

\end{tikzpicture}
\captionof{figure}{$\mathcal{V}^1_{x,\ell}$ and $J^{\ell}_X$ have exactly one vertex in common. We use this vertex to concatenate both paths.}
    \label{fig:path_from_random_path_of_cubes}
\end{center}
\end{figure}
    
    We have the following claim that collects properties of $\mathcal{J}_X^{\ell}$:
    \begin{claim}\label{claim:properties}
    \begin{enumerate}
        \item Let $\ell_1 \neq \ell_2$. The set of vertices of $\mathcal{J}^{\ell_1}_X$ that are within good cubes that belong to $\mathcal{L}_X^{\ell_1}$ is disjoint from the set of vertices of $\mathcal{J}_X^{\ell_2}$ that are within the good cubes that belong to $\mathcal{L}_X^{\ell_2}$. If $\mathcal{J}_X^{\ell_1}$ and $\mathcal{J}_X^{\ell_2}$ have common vertices, these are inside of $\mathrm{V}'(\mathrm{C_i})$ or $\mathrm{V}'(\mathrm{C_j})$.  
        \item The number of edges that have some endpoint inside of $\mathrm{V}'(\mathrm{C_i}) \cup \mathrm{V}'(\mathrm{C_j})$  and belong to some path $\mathcal{J}_X^{\ell}$ is $\Theta(\log n)$. 
    \end{enumerate}
\end{claim}
\textit{Proof of claim.}
    Recall that if $\ell_1 \neq \ell_2$, then $\mathcal{L}_X^{\ell_1}$ does not intersect $\mathcal{L}_X^{\ell_2}$, and as $\mathcal{J}^{\ell}_X$ is totally contained within $\mathcal{L}_X^{\ell} \cup \mathrm{V}'(\mathrm{C_i}) \cup \mathrm{V}'(\mathrm{C_j})$, part (a) follows.
    
    Regarding item (b), recall that $\mathrm{V}'(\mathrm{C_i})$ has logarithmic length in the renormalized lattice where we see cubes as points of $\mathbb{Z}^d$. Since the path $\mathcal{V}^1_{x,\ell}$ is chosen to be an induced path it has a logarithmic number of edges. Note that the other edges of $\mathcal{J}_X^{\ell}$ with at least one endpoint that is inside of $\mathrm{V}'(\mathrm{C_i})$ must have this endpoint inside of $\mathrm{C}^1_\ell$. Since the path is induced and the cube $\mathrm{C}^1_\ell$ has constant volume by our assumption on $r$, the number of edges from some realization $\mathcal{J}_X^{\ell}$ and with at least one endpoint that is inside of $\mathrm{C}^1_\ell$ is $O(1)$. Finally, since there is a logarithmic number of such cubes $\mathrm{C}^1_\ell$ (one for each $\ell$) the number of such vertices is logarithmic. The same is true for $\mathrm{V}'(\mathrm{C_j})$ and we obtain property (b). The claim follows.\qed

    We continue with the proof of the theorem. To this end, let $Y$ be a discrete uniform random variable with values in $\{1,...,\lfloor c\log n\rfloor\}$ and independent of $X$. Define the random path $\mathcal{J}_{X}^{Y}$ and for any $u, v \in \cG_n$, we define the function
    \[\psi(u,v)=\begin{cases}
    \text{1}, & \text{if $\{u,v\} \in \mathcal{J}_X^Y$ and the path visits $u$ first, }\\
    \text{-1}, & \text{if $\{u,v\} \in \mathcal{J}_{X}^Y$ and the path visits $v$ first,}\\
    \text{0}, & \text{otherwise.}
    \end{cases}\]
    We define a unit flow $\theta$ between $x$ and $y$ by setting (to see that this is indeed a flow, see the discussion just before \cite[Proposition 3.2]{LyonsPeres2016})
    \begin{equation}\label{eq:definition_flow}
        \theta(u,v)=\mathbb{E}_{X,Y}(\psi(u,v))
    \end{equation}
    for any $u,v\in \mathcal G_n$.
    Note that 
    \begin{equation}
    |\theta(u,v)|\leq \mathbb{P}_{X,Y}(\{u,v\}\in \mathcal{J}_{X}^{Y})=:p(u,v).
    \end{equation}
     We next group the parallelepipeds $\mathrm{T}$ that contain some realization of $\mathcal{J}^{Y}_X$ according to their level $t$ as follows (see Figure~\ref{fig:supercubes_levels}):
    \[D_t=\begin{cases}
    \{\mathrm{T}_{i,j}:i= s_1+t, \, r_1\leq j\leq r_1+t\}, & \text{if $0\leq t\leq \frac{s_2-s_1+1}{2}$, }\\
    \{\mathrm{T}_{i,j}:i=s_1+t,\, r_1\leq j\leq r_1+s_2-s_1-t-1\}, & \text{if $\frac{s_2-s_1+1}{2}+1$} \leq t \leq s_2-s_1+1.
    \end{cases}\]
\begin{figure}[htb]
    \begin{minipage}{0.5\textwidth}
\begin{tikzpicture}[scale=0.8]
% Define square size
\def\squarex{1} % Square width
\def\squarey{1} % Square height
\draw[draw=transparent, fill=red!30, thin] (3*\squarex, 0*\squarey) rectangle (4*\squarex, 1*\squarey); 
\draw[draw=transparent, fill=red!30, thin] (3*\squarex, 1*\squarey) rectangle (4*\squarex, 2*\squarey);
\draw[draw=transparent, fill=red!30, thin] (3*\squarex, 2*\squarey) rectangle (4*\squarex, 3*\squarey);
\draw[draw=transparent, fill=red!30, thin] (3*\squarex, 3*\squarey) rectangle (4*\squarex, 4*\squarey);
% Number of squares to the middle
\def\middle{4} % Number of squares to reach the middle from the first level

% Coordinates for the first square (bottom-left of the staircase)
\coordinate (start_x) at (0, 0);

% Draw the increasing part of the staircase
\foreach \level in {0,...,4} {
    \foreach \i in {0,...,\level} {
        \draw[thick] (\level*\squarex, \i*\squarey) rectangle (\level*\squarex + \squarex, \i*\squarey + \squarey);
    }
}
\draw[thick] (5*\squarex, 0*\squarey) rectangle (5*\squarex + \squarex, 0*\squarey + \squarey);
\draw[thick] (5*\squarex, 1*\squarey) rectangle (5*\squarex + \squarex, 1*\squarey + \squarey);
\draw[thick] (5*\squarex, 2*\squarey) rectangle (5*\squarex + \squarex, 2*\squarey + \squarey);
\draw[thick] (5*\squarex, 3*\squarey) rectangle (5*\squarex + \squarex, 3*\squarey + \squarey);
\draw[thick] (5*\squarex, 4*\squarey) rectangle (5*\squarex + \squarex, 4*\squarey + \squarey);

\draw[thick] (6*\squarex, 0*\squarey) rectangle (6*\squarex + \squarex, 0*\squarey + \squarey);
\draw[thick] (6*\squarex, 1*\squarey) rectangle (6*\squarex + \squarex, 1*\squarey + \squarey);
\draw[thick] (6*\squarex, 2*\squarey) rectangle (6*\squarex + \squarex, 2*\squarey + \squarey);
\draw[thick] (6*\squarex, 3*\squarey) rectangle (6*\squarex + \squarex, 3*\squarey + \squarey);

\draw[thick] (7*\squarex, 0*\squarey) rectangle (7*\squarex + \squarex, 0*\squarey + \squarey);
\draw[thick] (7*\squarex, 1*\squarey) rectangle (7*\squarex + \squarex, 1*\squarey + \squarey);
\draw[thick] (7*\squarex, 2*\squarey) rectangle (7*\squarex + \squarex, 2*\squarey + \squarey);

\draw[thick] (8*\squarex, 0*\squarey) rectangle (8*\squarex + \squarex, 0*\squarey + \squarey);
\draw[thick] (8*\squarex, 1*\squarey) rectangle (8*\squarex + \squarex, 1*\squarey + \squarey);

\draw[thick] (9*\squarex, 0*\squarey) rectangle (9*\squarex + \squarex, 0*\squarey + \squarey);

\coordinate (red_node) at (5*\squarex, 2.5);
\node[red, circle, fill=red, scale=0.5, label=above:{$y_X$}] at (red_node) {}; 

\coordinate (red_node_second) at (3*\squarex, 1.5);
\node[red, circle, fill=red, scale=0.5, label=above:{$y_X^t$}] at (red_node_second) {};

\draw[red, thick] (red_node) -- (0,0); % Segment to bottom-left corner of the first cube

\draw[blue, thick, dotted] (4.9, 2.5) rectangle (5.1, 0);
\draw[blue, thick, dotted] (2.9, 1.5) rectangle (3.1, 0);
% First square at the bottom-left
\coordinate (A1) at (0, 0);
\coordinate (B1) at (\squarex, 0);
\coordinate (C1) at (\squarex, \squarey);
\coordinate (D1) at (0, \squarey);
% Label the first square
\node[anchor=north west, scale=0.7] at (D1) {$\mathrm{T}_{s_1,r_1}$};
% Second labeled square (at the middle of the stair)
\coordinate (A2) at (4*\squarex, 0);
\coordinate (B2) at (4*\squarex + \squarex, 0);
\coordinate (C2) at (4*\squarex + \squarex, \squarey);
\coordinate (D2) at (4*\squarex, \squarey);
% Label the second square
% Third labeled square (at the far right)
\coordinate (A3) at (8*\squarex, 0);
\coordinate (B3) at (9*\squarex + \squarex, 0);
\coordinate (C3) at (8*\squarex + \squarex, \squarey);
\coordinate (D3) at (9*\squarex, \squarey);

% Label the third square
\node[anchor=north west, scale=0.7] at (D3) {$\mathrm{T}_{s_2,r_2}$};
\coordinate (red_node) at (3.5,-1);
\node[red, circle, scale=0.5, label=above:{$\textbf{Level t}$}] at (red_node) {}; 
\draw[->, black, thick] (red_node) -- ++(0, 1); % 1 unit upward
\coordinate (A) at (-0.2,-0.1);
\coordinate (B) at (2.8,-0.1);
\coordinate (C) at (4.8,-0.1);
\node[anchor=north west, scale=0.7] at (A) {A};
\node[anchor=north west, scale=0.7] at (B) {B};
\node[anchor=north west, scale=0.7] at (C) {C};

\coordinate (cubo_T) at (3.35,1.45);
\node[anchor=north west, scale=0.7] at (cubo_T) {$\mathrm{T}_{i,j}$};

\end{tikzpicture}

\captionof{figure}{Parallelepipeds at level $t$ are in red (i.e., their projections onto the first $2$ dimensions). See the projected probability over the right line of $D_t$ inside the blue dotted lines.}
    \label{fig:supercubes_levels}
\end{minipage}%
\begin{minipage}{0.5\textwidth}
\begin{center}
\begin{tikzpicture}[scale=0.5]
\draw[thick] (0, 0) rectangle (10, 10);
   \draw[thick] (8, 0) rectangle (8.5, 10);
   \draw[thick] (0, 2) rectangle (10, 2.5);
   \node[red, circle, fill=red, scale=0.6, label=right:{$x$}] at (2,2.25) {};
   \node[red, circle, fill=red, scale=0.6, label=right:{$y$}] at (8.25,9) {};
   \node[red, circle, fill=red, scale=0.6, label=right:{$z$}] at (8.25,2.2) {};
\end{tikzpicture}
\end{center}
\captionof{figure}{Using triangle inequality for vertices in different strips and in the same 2-dimensional slice.}
    \label{fig:resistance_different_strips}
\end{minipage}
\end{figure}
    Now, we are ready to bound the effective resistance between $x,y$. Recall that $\mathfrak{r}(e)=1$, and so, since $\theta^2(u,v) \le p^2(u,v)$, by Definition~\ref{def:effectiveresistance}, in order to bound $\mathcal{R}(x\longleftrightarrow y)$, it suffices to bound $\sum_{e \in E} p^2(e)$. More precisely, we may split the edge set according to levels of parallelepipeds and get:
    \begin{equation}\label{eq:upper_bound_formula_cubes}
        \mathcal{R}(x\longleftrightarrow y) \leq \sum_{u \in \mathrm{V}'(\mathrm{C_i}) \cup \mathrm{V}'(\mathrm{C_j})}
        \sum_{\substack{v \in L_1\\v\sim u}} p^2 (u,v) + \sum_{t}\sum_{\mathrm{T}_{i,j}\in D_t} \sum_{\substack{u \in \mathrm{T}_{i,j} \\ u \notin \mathrm{V}'(\mathrm{C_i}) \cup \mathrm{V}'(\mathrm{C_j})}} \sum_{\substack{v \in L_1\\v\sim u}} p^2(u,v). 
    \end{equation}

    The first double sum on the right hand side can be bounded by $O (\log n)$, as the number of edges that have positive probability of belonging to the path is $\Theta (\log n)$ (see item b) of Claim~\ref{claim:properties}). 
    
    Consider then the remaining term. Let $u\in \mathrm{T}_{i,j}$ with $\mathrm{T}_{i,j}\in D_t$ and $\{u,v\} \in E$ and $u \notin \mathrm{V}'(\mathrm{C_i}) \cup \mathrm{V}'(\mathrm{C_j})$. We now claim that the probability that $\mathrm{T}_{i,j}\in D_t$, with $1\leq t\leq \frac{s_2-s_1+1}{2}$, is part of the path of parallelepipeds $\mathcal{L}_X$ is bounded above by $\frac{c'}{t}$ for some constant $c'$ that can be chosen independent of $n$ and $t$ (for $t=0$ we bound this probability by 1). To show this, denote by $y_X^t$ the intersection of the segment between the left bottom corner point of $\mathrm{T}_{s_1,r_1}$ and $y_X$, and the left line of level $t$ (see Figure ~\ref{fig:supercubes_levels}). By similarity of triangles, $y_X^t$ is uniform over the left line of level $t$ (see for example the triangles $Ay^t_X B$ and $Ay_X C$ in Figure~\ref{fig:supercubes_levels}). Now, observe that in order for $\mathrm{T}_{i,j} \in D_t$ to be part of $\mathcal{L}_X$ the random point $y_X^t$ has to belong to the left face of $\mathrm{T}_{i,j}$ or to $\mathrm{T}_{i,j-1}$ (the left face of the parallelepiped immediately below, in case such parallelepiped exists): indeed, if $y_X^t$ were above $\mathrm{T}_{i,j}$, then the parallelepiped does not belong to the path, and if $y^{t}_X$ were below $\mathrm{T}_{i,j-1}$, as the angle of the red line is at most $\pi/4$, the red line in Figure~\ref{fig:supercubes_levels} does not intersect $\mathrm{T}_{i,j}$. Therefore, the probability that $\mathrm{T}_{i,j} \in D_t$ belongs to $\mathcal{L}_X$ is bounded from above by the probability that $y_X^t$ belongs to the left faces of two cubes at level $t$. Since $y_X^t$ is uniform in its range, this probability is at most $\frac{2}{t}$. By symmetry, the same is true for $ \frac{s_1-s_1+1}{2}+1 \le t \le s_2-s_1$ (for $t=s_2-s_1+1$ we bound this probability by $1$); in this case we obtain an upper bound of $\frac{2}{s_2-s_1-t-1}$ for the probability that $\mathrm{T}_{i,j} \in D_t$.

    Consider a fixed realization of $\mathcal{L}_X$,
    %(path of parallelepipeds, think this as the value of $X$ being fixed),
    and suppose first that $\{u,v\} \in E$ such that $u$ belongs to some parallelepiped $\mathrm{T}_{i,j}$ of $\mathcal{L}_X$. We may also assume that the vertices $u,v$ do not belong to the good cubes of $\mathrm{V}'(\mathrm{C_i}) \cup \mathrm{V}'(\mathrm{C_j}),$ as otherwise the contribution is already counted in the first double sum. 
    By independence of $X,Y$, and since by Claim~\ref{claim:properties} part a), $\{u,v\}$ belongs to one unique path $\mathcal{J}_X^Y$, and as the path is chosen uniformly from a set of $\lfloor c\log n\rfloor$ possible realizations, we have 
    \begin{equation}
        \mathbb{P}_{X,Y} (\{u,v\} \in \mathcal{J}^{Y}_{X}|X) \leq \frac{1}{\lfloor c\log n \rfloor}. 
    \end{equation}
    Furthermore, recall that $u$ belongs to the parallelepiped of $\mathrm{T}_{i,j} \in D_t$. Thus, by the above, since the choice of the parallelepiped in $D_t$ (with $1\leq t\leq \frac{s_2-s_1+1}{2}$) and the choice of the crossing are independent, $p(u,v) \le \frac{c''}{t \log n}$ for some $c'' > 0$. The same occurs by symmetry for the other levels.

    Fix then a parallelepiped $\mathrm{T}$ that belongs to some level $t$. Recall that the volume of the parallelepiped $\mathrm{T}$ is $\Theta(\log^2 n)$, since $r$ is constant in this section. Hence, the number of edges that have at least one endpoint inside of $\mathrm{T}$ is by Lemma~\ref{lem:edgedeviation} a.a.s.\ $O(\log^2 n)$. We substitute these observations into the right hand expression of~\eqref{eq:upper_bound_formula_cubes} and obtain  
\begin{equation}
\begin{aligned}
\sum_{1 \leq t \leq \frac{s_2-s_1+1}{2}} \sum_{\mathrm{T}_{i,j} \in D_t} \sum_{\substack{u \in \mathrm{T}_{i,j} \\ u \notin \mathrm{V}'(\mathrm{C_i}) \cup \mathrm{V}'(\mathrm{C_j})}} \sum_{v \in L_1} p^2(u,v) 
&%\underset{\text{~\ref{eq:bound_on_p_for_level_r}}}{\leq} 
\leq
\sum_{1 \leq t \leq \frac{s_2-s_1+1}{2}} \sum_{\mathrm{T}_{i,j} \in D_t} \sum_{\substack{u \in \mathrm{T}_{i,j} \\ u \notin \mathrm{V}'(\mathrm{C_i}) \cup \mathrm{V}'(\mathrm{C_j})}} \sum_{v \in L_1} \frac{c'''}{t^2\log^2 n} \\
&%\underset{\text{\# of vertices inside $\mathrm{T}_{i,j}$}}{\leq} 
\leq
\sum_{1 \leq t \leq \frac{s_2-s_1+1}{2}} \sum_{\mathrm{T}_{i,j} \in D_t} \log^2 n \frac{c''''}{t^2\log^2 n} \\
&%\underset{\text{\# of parallelepipeds in $D_r$}}
\leq 
\sum_{1 \leq t \leq \frac{s_2-s_1+1}{2}} (t+1) \log^2 n \frac{c''''}{t^2\log^2 n} \\
&= \sum_{1 \leq t \leq \frac{s_2-s_1+1}{2}} O\left(\frac{1}{t}\right) 
= O(\log n).
\end{aligned}
\end{equation}
It is clear that the contribution of $D_0$ to the inequality is $O(1)$. By symmetry around $y_X$, one obtains the same bound for the remaining levels $\frac{s_2-s_1+1}{2}<t\leq s_2-s_1+1$. Then $\mathcal{R}(x\longleftrightarrow y)=O(\log n)$.

We now finally lift the assumptions on the parallelepipeds, namely that $r_1=r_2$ and that $s_2-s_1$ is odd, plus the fact that neither $\mathrm {C_i},\mathrm {C_j}$ belong to the boundary: consider first two vertices $x,y$ of $L_1$ that are in the same strip of a 2-dimensional slice, but now separated by an odd number of parallelepipeds that are not at the boundary. Let $x\in \mathrm{T}_{s_1',r_1}$ and $y\in \mathrm{T}_{s_2',r_1}$ with $s_2'-s_1'$ even such that neither parallelepiped is at the boundary. By the construction given above there exists $x' \in L_1$ with $x' \in \mathrm{T}_{s_1'+1,r_1}$ such that $\mathcal{R}(x'\longleftrightarrow y)=O(\log n)$. Note that $d_E(x,x')=O(\log n)$, since $x$ and $x'$ are in adjacent parallelepipeds, and any parallelepiped has logarithmic side length in the first 2 dimensions and fixed side length in the other dimensions. By Corollary ~\ref{cor:chemical_distance_rgg}, a.a.s., the graph distance between $x$ and $x'$ is then also $O(\log n)$, and thus a.a.s.\ $\mathcal{R}(x\longleftrightarrow x') = O(\log n)$. Then, by application of the triangle inequality (see Lemma~\ref{lem:triangleinequality}), also a.a.s. $\mathcal{R} (x \longleftrightarrow y) = O(\log n)$.  
Also, since any two vertices in the same parallelepiped (not necessarily in good cubes, but still in $L_1$) satisfy that by Corollary~\ref{cor:chemical_distance_rgg}, their graph distance is a.a.s.\ $O(\log n)$, by another application of the triangle inequality, also a.a.s.\ the effective resistance between them is $O(\log n)$. Hence, for any $x',y' \in L_1$ belonging to the largest connected component of the renormalized graph inside the same strip of a $2$-dimensional slice not at the boundary, we have
\begin{equation}
    \mathcal{R}(x \longleftrightarrow y) = O (\log n).
\end{equation}
Using the same idea, we may extend our bounds to vertices in parallelepipeds of a strip on the boundary of the 2-dimensional slice, since for these vertices there exists a vertex at graph distance $O(\log n)$ in a parallelepiped not at the boundary.

    We now extend the result to any two vertices $x,y$ of $L_1$ in the same 2-dimensional slice, but possibly in different strips: to do so, suppose that $x,y \in \mathcal{K}^{2}_n$. Let $x \in \mathrm{H}_r$ ($x$ is in the $r$-the horizontal strip) and let $y \in \mathrm{V}_s$ ($y$ in the $s$-th vertical strip). Note that this election is always possible,  as any vertex belongs to one vertical strip and one horizontal strip. Pick $z \in L_1$ such that $z \in \mathrm{V}_s \cap \mathrm{H}_r = \mathrm{T}_{s,r}$ (see Figure ~\ref{fig:resistance_different_strips}). Note that such $z$ must exist as $\mathrm{T}_{s,r}$ is the intersection of two strips that contain long crossings of good cubes, which at the same time contain long crossings of $L_1$. As $x$ and $y$ are in the same strip, by the above, a.a.s.\ $\mathcal{R}(x\longleftrightarrow z) = O (\log n)$. Similarly, a.a.s.\ $\mathcal{R}(y\longleftrightarrow z) = O (\log n)$, and applying again the triangle inequality, the result holds as well.

    Now, in general, suppose that there is no 2-dimensional slice containing $x, y$; while always a plane containing $x,y$ exists, we cannot ensure the existence of a slice, which by our definition is assumed to be axis-parallel. Now, in case two different slices have non-empty intersection, this intersection is a line of cubes such that only one coordinate is not fixed. In this case, as before, there exists at least a vertex $z \in \cG_n$ that belongs to $L_1$ and that is in the intersection of both slices. Then, for any two vertices $x',y'$ such that $x'$ is in the first slice and $y'$ is in the second slice, by the above, a.a.s.\ $\mathcal{R}(x' \longleftrightarrow z)=O(\log n)$ and  $\mathcal{R}(y'\longleftrightarrow z)=O(\log n)$, and thus, by the triangle inequality, also $\mathcal{R}(x'\longleftrightarrow y')=O(\log n)$. To conclude, observe that for any two vertices $x,y \in L_1$, there exists a sequence of slices $(\mathcal{K}_n^2)_1,...,(\mathcal{K}_n^2)_s$, with $s \le d-1$, such that for any $1 \le i \le s-1$,  $(\mathcal{K}_n^2)_s$ has non-empty intersection with $(\mathcal{K}_n^2)_{s+1}$, and moreover, $x \in (\mathcal{K}_n^2)_1$, and $y \in (\mathcal{K}_n^2)_s$. By applying the triangle inequality iteratively, we then have a.a.s.\ $\mathcal{R}(x\longleftrightarrow y)=O(\log n)$.
    
    By this result, together with Lemma~\ref{lem:sizergg}, we then immediately conclude that a.a.s.
    \begin{equation}
        \tau_{cov} \leq \gamma n \log^2 n,
    \end{equation}
    for some $\gamma(d,r) > 0$ (uniformly bounded from below), and the theorem follows.
\end{proof} 

\section{Upper bound for non-fixed radius}\label{sec:upper_bound_covTime_rgg_nonFixed_r}
In this section we prove the upper bound for RGGs with radii tending to infinity but below the threshold of connectivity. We begin by noting that the proof given in Section~\ref{sec:upperboundfixedradius} is not sufficient, as it gives as an upper bound on the resistance of $O(\log n)$, whereas we need an upper bound of magnitude $O(\log(n)/r^d)$. With this in mind, we therefore introduce a refined construction that allows for a better flow dissipation. This structure will then also be used in the proof of the upper bound right above the connectivity threshold.

\subsection{Resistance bound in the backbone of dense cubes}\label{subsec:nonfixed_r_upper_bound_backbone}
Recall the definition of $\mu$ introduced in the tessellation of cubes $(\mathrm{Q_i})_i$ in Section $\ref{sec:lower_bound_fixed_radius}$ of cubes of side length $r/\mu$. In this section, we choose $\mu(n) =\mu +O(1/n)$ to be the smallest real number larger or equal to $\mu$ such that any cube $\mathrm{Q}$ is either completely contained in $\Lambda_n$ or it shares lower-dimensional faces with the boundary of $\Lambda_n$. With abuse of notation, we denote below $\mu(n)=\mu$. Let $\beta$ be a sufficiently small positive constant. We say that $\mathrm{Q_i}$ is $\emph{dense}$ if it has at least $\beta r^d$ vertices of $\cG_n$, otherwise we say that it is $\emph{non-dense}$. 
 
    We see cubes $\mathrm{Q}$ as sites of $\mathbb{Z}^d$. By Chernoff's inequality, the probability that a cube is non-dense is bounded from above by $e^{-\kappa r^d}$ for some $\kappa=\kappa(\beta)>0$ (assuming $\beta$ small enough). We note that $\kappa$ is a fixed constant and does not depend on $r$. Considering dense cubes as occupied sites of a Bernoulli site percolation model, for $r \ge r_0$ with $r_0$ large enough, we are thus in the supercritical phase of Bernoulli site percolation. That is, a.a.s., there exists a giant component corresponding to dense cubes. We denote by $\mathcal{C}^{\text{site}}_{\mathrm{Q}}(n)$ the giant (largest) cluster of the associated Bernoulli process of dense cubes $\mathrm{Q}$. We prove the following regarding the maximum of the effective resistance in the giant cluster:
    \begin{theorem}\label{thm:nonfixed_r_upper_bound_backbone}
        Let $r \ge r_0$ for a sufficiently large constant $r_0 > 0$ and let also $r^d\le C\log n$ for some $C > 0$. 
        Then there is $C'=C'(d,r_0,C)$, such that a.a.s., for every $x,y \in  \cG_n$ such that $x \in \mathrm{Q_i}$ and $y \in \mathrm{Q_j}$ and $\mathrm{Q_i},\mathrm{Q_j}\in \mathcal{C}^{\text{site}}_{\mathrm{Q}}(n)$, it holds that 
        \[
        \mathcal{R}(x\longleftrightarrow y) = C' \frac{\log n}{r^{2d}}.
        \]
    \end{theorem}
    \begin{proof}
    First, we may choose $r_0$ large enough such that $\beta r^d$ is at least 3 and that the dense cubes dominate a supercritical Bernoulli site percolation process with a parameter that only depends on $r_0$.
        Divide $\Lambda_n$ into disjoint $\alpha$-logarithmic strips of cubes $\mathrm{Q_i}$. We use the structure introduced in Section~\ref{sec:randompaths}, with the difference that we consider cubes $\mathrm{Q_i}$ instead of cubes $\mathrm{C_i}$, and dense cubes play the role of good cubes. 

        In the following we fix two vertices $x,y$ belonging to a fixed $2$-dimensional slice (we generalize this later). Analogously as before, define $\mathcal{Q}^{Y}_X$,  where $Y$ is a uniform discrete random variable independent from $X$ with values in $\{1,\dots, \lfloor c \log n\rfloor\}$ (see the construction of Section~\ref{sec:randompaths} for the definition of $c$ and Theorem~\ref{thm:upperboundcov_fixed_r} for details of the construction of $\mathcal{Q}^{\ell}_X$ with a fixed value $\ell$ of $Y$): note that $\mathcal{Q}^{\ell}_X$ is a random path of dense cubes between two fixed dense cubes $\mathrm{Q_i}$ and $\mathrm{Q_j}$, such that both belong to the giant component of dense cubes in $\mathcal{C}^{\text{site}}_{\mathrm{Q}}(n)$. As in Section~\ref{sec:randompaths}, we may assume that the two dense cubes belong to some parallelepipeds $\mathrm{Q_i} \in \mathrm{T}_{s_1,r_1}$ and $\mathrm{Q_j} \in \mathrm{T}_{s_2,r_2}$ such that $r_1=r_2$, $s_2>s_1$, $s_2-s_1$ is odd, and that both parallelepipeds do not belong to a strip that is at the boundary of the 2-dimensional slice. Pick vertices $x,y$ of $\cG_n$ such that $x\in \mathrm{Q_i}$ and $y\in \mathrm{Q_j}$. For each dense cube that is part of some realization of $\mathcal{Q}^{Y}_X$ select $\lfloor \beta r^d \rfloor$ vertices, except for $\mathrm{Q_i}$ and $\mathrm{Q_j}$ where we may choose $\lfloor \beta r^d \rfloor - 1$ vertices distinct from $x,y$ (or $\lfloor \beta r^d \rfloor-2$ in case $\mathrm{Q_i}=\mathrm{Q_j}$). 
    
    Conditionally on $X$ and $Y$, we construct a random path of vertices in $\cG_n$  that starts at $x$, finishes at $y$, and is entirely contained within the path of dense cubes $\mathcal{Q}_X^Y$ as follows. First, select uniformly at random and independently from $X,Y$ one of the $\lfloor \beta r^d\rfloor-1$  vertices that belong to the dense cube $\mathrm{Q_i}$ (one out of the $\lfloor \beta r^d \rfloor-1$ vertices previously selected), and connect it by an edge to the vertex $x$: this is the first edge of our path. For subsequent steps we proceed similarly: suppose that the previously selected vertex belongs to the dense cube $\mathrm{Q}_s$. We then choose uniformly and independently at random from the $\lfloor \beta r^d\rfloor$ vertices that belong to the dense cube that follows in the path $\mathcal{Q}^{Y}_X$. We continue, until the last selected vertex belongs to the cube just before $\mathrm{Q_j}$. In this case we select uniformly at random between the $\lfloor \beta r^d\rfloor -1$ selected vertices in $\mathrm{Q_j}$ and finally connect all these last vertices by an edge to $y$ (if we had $\mathrm{Q_i}=\mathrm{Q_j}$, then we would pick $\lfloor \beta r^d \rfloor-2$ vertices inside the cube and form paths of length $2$ from $x$ to $y$ through these vertices). We denote this random path of $\cG_n$ by $\mathcal{J}_{X,Z}^Y$ and denote its law by $\mathbb{P}_{X,Y,Z}$, where $Z$ denotes the randomness that comes from the path selection conditioned on $X,Y$. Observe that the number of edges that have some endpoint inside of $\mathrm{V}'(\mathrm{Q_i})\cup \mathrm{V}'(\mathrm{Q_j})$ is $O(r^{2d} \log n)$, since $\mathrm{V}'(\mathrm{Q_i})$ (see also Section ~\ref{sec:randompaths}) is made of $\Theta(\log n)$ dense cubes, each of which has $O(r^d)$ selected vertices, and any such vertex has $\Theta(r^{d})$ neighbors among selected vertices inside its own cube and inside a constant number of cubes around it.   
Since between any two consecutive dense cubes there are $\Theta(r^{2d})$ edges, and we choose uniformly at random one of them, we obtain in the case $u \neq x, u \neq y, v \neq x, v \neq y$, $u\neq v$,
    $$
        \mathbb{P}_{X,Y,Z}(\{u,v\} \in \mathcal{J}_{X,Z}^Y|X,Y) =O\left(\frac{1}{r^{2d}}\right), 
   $$ 
whereas in the case $|\{u,v\}\cap \{x,y\}|\in \{1,2\}$, we have
    $$
    \mathbb{P}_{X,Y,Z}(\{u,v\} \in \mathcal{J}_{X,Z}^Y|X,Y) =O\left(\frac{1}{r^d}\right) .
    $$
   As before, denote by $p(u,v)$ the probability that the edge $\{u,v\}$ belongs to the random path $\mathcal{J}_{X,Z}^Y$. Using the random path $\mathcal{J}_{X,Z}^Y$, we define a flow in the same way as in the proof of Theorem ~\ref{thm:upperboundcov_fixed_r} (see \eqref{eq:definition_flow}), and we obtain
    \begin{align}\label{eq:upper_bound_formula_dense_cubes}
\mathcal{R}(x \longleftrightarrow y) \leq 
&\sum_{\substack{v\in L_1\\ y\sim v }} p^2(y, v) +\sum_{\substack{v\in L_1\\ x\sim v }} p^2(x, v) \notag \\
&+ \sum_{\substack{u \in \mathrm{V}'(\mathrm{Q}_i) \cup \mathrm{V}'(\mathrm{Q}_j) \\ u \neq x, u \neq y}} 
\sum_{\substack{v \in L_1 \\ v \neq x, v \neq y \\u\sim v}} p^2(u, v) \notag \\
&+ \sum_{t} 
\sum_{\mathrm{T}_{i,j} \in D_t} 
\sum_{\substack{u \in \mathrm{T}_{i,j} \\ u \notin \mathrm{V}'(\mathrm{Q}_i) \cup \mathrm{V}'(\mathrm{Q}_j) \\ u \neq x, u \neq y}} 
\sum_{\substack{v \in L_1\\ u\sim v\\v \neq x, v \neq y}} p^2(u, v).
\end{align}
    To bound the first two sums,  since $p(x,v)=O\left( \frac{1}{r^d}\right)$, and since $x$ and $y$ send flow to different $O (r^d)$ vertices, we have
    $\sum_{v \in L_1} \big(p^2(y, v) + p^2(x, v)\big)=O\left(1/r^d \right)$.
For the sum in the second row, since the number of edges with at least one endpoint inside of $\mathrm{V}'(\mathrm{Q_i}) \cup \mathrm{V}'(\mathrm{Q_j})$ is $O(r^{2d}\log n)$, and since $p(u,v)=O(1/r^{2d})$,$$
    \sum_{\substack{u \in \mathrm{V}'(\mathrm{Q}_i) \cup \mathrm{V}'(\mathrm{Q}_j) \\ u \neq x, u \neq y\\}} 
\sum_{\substack{v \in L_1 \\ v \neq x, v \neq y\\u \sim v}} p^2(u, v)   
    = O \left(\frac{\log n}{r^{2d}}\right).
    $$
    To bound the term in the last row, recall that, as in the proof of Theorem~\ref{thm:upperboundcov_fixed_r} in Section~\ref{sec:upperboundfixedradius}, the probability that a dense cube inside of some parallelepiped $\mathrm{T} \in D_t$ belongs to $\mathcal{Q}_X^Y$ is $O\left(\frac{1}{t\log n }\right)$ for $1\leq t\leq \frac{s_2-s_1+1}{2}$ and $O\left(\frac{1}{\log n}\right)$ for $t=0$ (if the parallelepiped does not belong to $\mathrm{V}'(\mathrm{Q_i})\cup \mathrm{V}'(\mathrm{Q_j})$),  since there are $\Omega(t)$ parallelepipeds at level $t \ge 1$, and we choose one of the $\Omega(\log n)$ crossings of dense cubes, uniformly at random. Now, observe that we have the additional randomness coming from the fact for any pair of consecutive dense cubes, we choose uniformly at random a vertex inside each cube, giving another factor $O(1/r^{2d})$, as mentioned above. All three choices are independent of each other, and thus,  for edges $\{u,v\}$ appearing in the last term, we have
   $$ 
        p(u,v) = O \left(\frac{1}{t r^{2d}\log n} \right).
    $$

    The number of edges with positive probability $p$ and with at least one endpoint inside any parallelepiped $\mathrm{T}$ is a.a.s. $O(r^{2d}\log^2 n)$, as there are $O(\log^2 n)$ dense cubes inside $\mathrm{T}$ (the volume of $\mathrm{T}$ is $\Theta(r^{d}\log^2 n)$, and each cube has volume $\Theta(r^d)$), each dense cube contains $O(r^d)$ vertices, and the selected vertices are connected by an edge to another $O(r^d)$ selected vertices by construction. Thus
    \begin{equation}
\begin{aligned}
\sum_{1\leq t\leq \frac{s_2-s_1+1}{2}} 
&\sum_{\mathrm{T}_{i,j} \in D_t} 
\sum_{\substack{u \in \mathrm{T}_{i,j} \\ u \notin \mathrm{V}'(\mathrm{Q}_i) \cup \mathrm{V}'(\mathrm{Q}_j)\\ u \neq x, u\neq y}} 
\sum_{\substack{v\in L_1 \\ u\sim v \\ v \neq x, v\neq y}} 
p^2(u,v) \\
&\leq
\sum_{1\leq t\leq \frac{s_2-s_1+1}{2}} 
\sum_{\mathrm{T}_{i,j} \in D_t} 
\sum_{\substack{u \in \mathrm{T}_{i,j} \\ u \notin \mathrm{V}'(\mathrm{Q}_i) \cup \mathrm{V}'(\mathrm{Q}_j)\\ u\neq x,u\neq y}} 
\sum_{\substack{v\in L_1 \\ u\sim v\\ v\neq x,v\neq y}} 
\frac{c'''}{t^2r^{4d}\log^2 n} \\
&\leq 
\sum_{1\leq t\leq \frac{s_2-s_1+1}{2}} 
\sum_{\mathrm{T}_{i,j} \in D_t} 
r^{2d}\log^2 n \frac{c''''}{t^2r^{4d}\log^2 n} \\
&\leq 
\sum_{1\leq t\leq \frac{s_2-s_1+1}{2}} 
(t+1) \frac{c''''}{t^2 r^{2d}} \\
&= O\left( \frac{1}{r^{2d} }\right)
\sum_{1\leq t\leq \frac{s_2-s_1+1}{2}} 
O\left(\frac{1}{t}\right) 
\\ 
&= O \left(\frac{\log n}{r^{2d}}\right).
\end{aligned}
\end{equation}
It is clear that the contribution of $D_0$ to the formula is $O \left(\frac{1}{r^{2d}}\right)$. By symmetry, the contribution of $D_t$ with $\frac{s_2-s_1}{2}< t\leq s_2-s_1+1$ is also $O \left(\frac{1}{t r^{2d}}\right)$,
and thus
   $$
        \mathcal{R}(x \longleftrightarrow y) = O\left(\frac{\log n}{r^{2d}} \right).$$
    Recall that we assumed that $x,y$ belong to some dense cubes $\mathrm{Q_i}$ and $\mathrm{Q_j}$,  that at the same time are inside of some parallelepipeds in the same horizontal strip and separated by an even number of parallelepipeds. Before we lift these conditions, we first provide a useful construction in the following claim:
    \begin{claim}\label{claim:construction} Suppose $x',y'\in L_1$ with $x'\in \mathrm{Q_i}'$ and $y' \in \mathrm{Q_j}'$ with both cubes being dense, and such that $\mathrm{Q_i}'$ and $\mathrm{Q_j}'$ are connected by a path $\mathcal{L}$ of dense cubes of length $O(\log n)$. Then $$\mathcal{R}(x'\longleftrightarrow y')=O\left(\frac{\log n}{r^{2d}}\right).$$
    \end{claim}
    \textit{Proof of claim.} Let $x'$ send an amount of $\frac{1}{\lfloor \beta r^d \rfloor -1 }$ flow to the $\lfloor \beta r^d \rfloor-1$ previously selected vertices of $\mathrm{Q_i}'$. From there on, let the previously selected vertices send a uniformly distributed flow to the selected $\lfloor \beta r^d \rfloor$ vertices that belong to the dense cube that follows in the path connecting $\mathrm{Q_i}'$ and $\mathrm{Q_j}'$. We continue distributing in a uniform way the flow until we arrive until $\mathrm{Q_j}'$. Once the flow reached the $\lfloor \beta r^d \rfloor-1$ selected vertices of $\mathrm{Q_j}'$, send all the flow back to $y'$. Denote this flow by $\theta$. We then have 
    \begin{equation}
        \mathcal{R}(x'\longleftrightarrow y') \leq \sum_{\substack{z \in \mathrm{Q_i}'\\x'\sim z}}\theta^2(x',z) +\sum_{\substack{z\in \mathrm{Q_j}'\\y'\sim z}} \theta^2(y',z) + \sum_{\mathrm{Q} \in \mathcal{L}} \sum_{\substack{u\in \mathrm{Q} \\ v \in L_1\\u\sim v}} \theta^2(u,v).
    \end{equation}
    As before, the first two sums contribute $O (r^d) O(1/r^{2d})=O(1/r^d)$. For the third sum, note that there are $O(\log n)$ dense cubes in $\mathcal{L}$, and between two consecutive dense cubes there are $O(r^{2d})$ edges with flow $O\left(\frac{1}{r^{2d}}\right)$ on each of them. Thus
$$
        \mathcal{R}(x'\longleftrightarrow y') = O\left( \frac{1}{r^{d}}\right) + \sum_{|\mathcal{L}|} O\left(r^{2d}\right) O\left(\frac{1}{r^{4d} }\right) =O\left(\frac{\log n}{r^{2d}}\right),
    $$
    and the claim is proved.\qed
    
    We now continue with the proof of the theorem and show how Claim~\ref{claim:construction} allows us to prove the bound on the effective resistance for any $x\in \mathrm{C_i}' \in \mathrm{T}_{s_1,r_1}$ and $y\in \mathrm{C_j}' \in \mathrm{T}_{s_2,r_2}$ in the same strip. Assume first that both parallelepipeds do not belong to strips on the boundary. First, if $x$ and $y$ are separated by an odd number of parallelepipeds (that is, $s_2-s_1$ even), we find $x'$ in a parallelepiped adjacent to $x$, at graph distance in the renormalized lattice of cubes $O(\log n)$ from $x$. Thus, by Claim~\ref{claim:construction}, $\mathcal{R}(x\longleftrightarrow x') = O\left( \frac{\log n}{r^{2d}}\right)$, and since $x'$ and $y$ are separated by an even number of parallelepipeds, by the above, $ \mathcal{R}(x'\longleftrightarrow y) = O\left( \frac{\log n}{r^{2d}}\right)$, and finally by Lemma~\ref{lem:triangleinequality}, also $ \mathcal{R}(x\longleftrightarrow y) = O\left( \frac{\log n}{r^{2d}}\right)$.
    As before, by the same idea, we may extend our bounds also to arbitrary vertices inside the same parallelepiped (in cubes not satisfying all restrictions) and to vertices in parallelepipeds of a strip on the boundary.

To generalize the inequality between any pair of strips we use the same  idea as in Section~\ref{sec:upperboundfixedradius} from Figure~\ref{fig:resistance_different_strips}: for any pair of strips, we find another vertex $z$ in the intersection of the strips that contain $x$ and $y$ and that intersect at some parallelepiped $\mathrm{T}$ that contains the vertex $z$ such that $\mathcal{R}(x\longleftrightarrow z)$ and $\mathcal{R}(y\longleftrightarrow z)$ are $O\left(\frac{\log n}{r^{2d}}\right)$ (such $z$ must exist, as there is a logarithmic amount of crossings of dense cubes in each strip), we apply again the triangle inequality and we are done. 

To generalize to different slices, we proceed once more as in the last step of the proof of Theorem~\ref{thm:upperboundcov_fixed_r}, where we saw that we may go from one slice to another in at most a fixed amount of steps. Subsequent applications of the triangle inequality then immediately yield the desired result.

We thus finished proving that a.a.s.\ the effective resistance between two vertices $x,y \in \cG_n$ that belong to dense cubes that form part of the giant component of dense cubes is $O\left(\frac{\log n}{r^{2d}}\right)$. 

To conclude, observe that the a.a.s.\ statement holds simultaneously for all $x,y$, as the number of edges with positive probability $p$ and with at least one endpoint inside a fixed parallelepiped is $O(r^{2d}\log^2 n)$ a.a.s.\ simultaneously for all parallelepipeds (it holds inside one fixed $T$ with probability $1-o(n^{-2})$).
    \end{proof}

\begin{lemma}\label{lem:upper_bound_nonfixedr_second_case}
    Let $r_0$ be a large enough constant, and let $r_0\leq r\leq (1-\eps)r_c $. Then
    \begin{equation}
        \tau_{cov}(L_1) = O (n \log^{2} n).
    \end{equation}
\end{lemma}

\begin{proof}
   By Lemma~\ref{lem_Timar}, every maximal connected component of cubes, which might contain dense and non-dense cubes and that has empty intersection with the giant component of dense cubes, has a diagonally connected interior boundary of non-dense cubes. As there are at most $c_d^{2L}$ diagonally connected paths of length $L$ when the starting vertex is fixed (as before, one may fix a spanning tree, traverse each edge twice, and at each step there are at most $c_d=c_d(d)$ choices to select the next vertex), and since the probability that a cube is non-dense is, by standard Chernoff bounds for Poisson random variables,  $e^{-\kappa r^d}$ for some $\kappa=\kappa(\beta) > 0$, the probability that there exists a diagonally connected path of non-dense cubes of length greater or equal than $L$ is, by a union bound over all $(1+o(1))n$ starting vertices, at most
$$
(1+o(1))n e^{-L\kappa r^d}(c_d^{2})^L,
    $$
    which tends to 0 for $L= c\frac{\log n}{r^d}$ with $c > 0$ large enough. 
    Hence, for this choice of $L$, a.a.s.\, in the renormalized lattice of cubes $\mathrm{Q}$ of volume $M^d$, the boundary of any connected component excluding the giant is at most of size $L$. Thus, by a standard isoperimetric inequality, a.a.s. every component different from the giant of the associated Bernoulli site percolation has volume at most $O(L^{d/(d-1)})$. In other words, for every site, at distance $L$ one finds a site belonging to the giant component (of the renormalized lattice of dense cubes), and for any vertex $x \in \cG_n$, at Euclidean distance at most $O(Lr/\mu) = O(\log n/r^{d-1})$ we find a vertex $y \in \cG_n$ belonging to a cube of the giant component of the Bernoulli site percolation.

    Then, by Corollary~\ref{cor:chemical_distance_rgg}, a.a.s.\ we have $d_G(x,y)=O\left( \log n/r^d  \right)$, and thus $\mathcal{R}(x \longleftrightarrow y) = O(\log n/r^d)$. If $z \in \cG_n$ is another vertex that belongs to some dense cube in the backbone $\mathcal{C}^{\text{site}}_{\mathrm{Q}}(n)$ then by~ Theorem~\ref{thm:nonfixed_r_upper_bound_backbone} we have $\mathcal{R}(y\longleftrightarrow z) = O(\log n/r^{2d})$, and hence by the triangle inequality, also
    \[
    \mathcal{R}(x \longleftrightarrow z) = O(\log n /r^d) + O(\log n/r^{2d}) = O(\log n/r^d).
    \]
   Since by Lemma~\ref{lem:sizergg}, the number of edges of the giant component satisfies $|E| = \Theta(nr^d)$ a.a.s., by Theorem~\ref{thm:upperboundcov_chandra}, we have that
    \[
    \tau_{cov} (L_1) = O(n \log^2 n),
    \]
 and the lemma follows.
\end{proof}
\textbf{Proof of Theorem~\ref{ref:main1}.} By combining Theorem~\ref{thm:lower_bound_cov_r_not_fixed} and Theorem~\ref{thm:upperboundcov_fixed_r} with Lemma~\ref{lem:upper_bound_nonfixedr_second_case}, Theorem~\ref{ref:main1} follows immediately. \qed

 \section{Upper bound right above the threshold for connectivity for \texorpdfstring{$d=2$}{d=2}}\label{sec:upper_bound_above_connectivity}

 Perhaps surprisingly, for $\mathcal{G}(n,r,2)$, no upper bound of $O(n \log n)$ on the cover time of the simple random walk had been known right above the connectivity threshold. For $d \ge 3$, in the already mentioned paper of Cooper and Frieze~\cite{cooper2011cover} the result was obtained, but their result does not apply for $d=2$. As mentioned in the introduction, for $d=2$, in \cite{covertimergg} the authors found the asymptotic order of the cover time only for a radius that was by a constant factor larger than the connectivity threshold. In this section, we close the gap left by these papers for $d=2$. We note that it is enough to prove the upper bound, as the lower bound is immediate by the general lower bound on cover times for connected graphs by Feige~\cite{feige1995lower}. Our main result will therefore be the following theorem (since our result is only new for $2$ dimensions close to the connectivity threshold, and we thus state it only for 2 dimensions, although it holds more generally by straightforward generalizations of cells to cubes):

 \begin{theorem}\label{thm:upperboundrg}
     Let $\eps>0$ be an arbitrarily small constant, and consider $d=2$. Suppose that $ (1+\eps)\frac{\log n}{\pi} \le r^2$, and also $r^2=\Theta(\log n)$. Then, a.a.s., for any two vertices $x,y \in \cG_n$,
     \begin{equation}
         \mathcal{R}(x\longleftrightarrow y) = O\left(\frac{1}{\log n}\right).
     \end{equation}
 \end{theorem}
 \textit{Proof of Theorem~\ref{ref:main2} assuming Theorem~\ref{thm:upperboundrg}:} By Lemma~\ref{lem:sizergg}, a.a.s., $|E|=\Theta(nr^d) =\Theta( n\log n)$. Theorem~\ref{ref:main2} then follows immediately by Theorem~\ref{thm:upperboundcov_chandra}. \qed
 
 \par\noindent It thus remains to prove  Theorem~\ref{thm:upperboundrg}. The approach of the proof is similar to the one in Section~\ref{sec:upper_bound_covTime_rgg_nonFixed_r}: first we prove that the upper bound on the effective resistance is true for any pair of vertices inside the \emph{backbone} of dense cubes, and in a second step we then prove that the resistance between any two vertices is of the desired order. We now proceed to the details:

 \textit{Proof of Theorem~\ref{thm:upperboundrg}:} Recall that $\mathcal{C}^{\text{site}}$ denotes the giant component of dense cubes of the auxiliary graph of cubes as in Section~\ref{sec:upper_bound_covTime_rgg_nonFixed_r}. By Theorem~\ref{thm:nonfixed_r_upper_bound_backbone}, the result holds for any $x,y \in \cG_n$ with $x \in \mathrm{Q_i}, y \in \mathrm{Q_j}$, and $\mathrm{Q_i}, \mathrm{Q_j} \in \mathcal{C}^{\text{site}}_{\mathrm{Q}}(n)$. It thus remains to show that the effective resistance between any vertex $x \in \cG_n$   that is in a cube, being dense or non-dense, that is not in the giant component of the auxiliary Bernoulli site process, and some vertex $y \in \cG_n$, belonging to a dense cube in the giant component of the auxiliary Bernoulli site process is at most $O\left(\frac{1}{\log n}\right)$: by the triangle inequality (Lemma~\ref{lem:triangleinequality}) together with Theorem~\ref{thm:nonfixed_r_upper_bound_backbone}, this will then show that for any pair of vertices $x,y \in \cG_n$, we have $\mathcal{R}(x\longleftrightarrow y) =O\left(\frac{1}{\log n}\right)$. 
 
 In order to show the claimed bound on the effective resistance, it suffices to show that for every vertex $x$, we find $y \in \mathrm{Q_j} \in \mathcal{C}^{\text{site}}_{\mathrm{Q}}(n)$ together with a set $\mathcal{P}_{x,y}$ of $\Omega(\log n)$ internally vertex-disjoint paths between $x$ and $y$ of constant length (in terms of graph distance). Indeed, in this case one can split evenly the unit flow emanating from $x$ between all these paths, such that the flow on each such edge of each path is $O(1/\log n)$, and the resistance, by Definition~\ref{def:effectiveresistance}, thus satisfies
$$\mathcal{R}(x\longleftrightarrow y) 
= \sum_{e \in P \in \mathcal{P}_{x,y}} O(1/\log^2 n) = O\left(\frac{1}{\log n}\right).
$$
We will follow this plan in detail now. We first assume that $x$ is at Euclidean distance at least $3C\sqrt{\log n}$ from all boundaries of $\Lambda_n$, where $C$ is a sufficiently large constant.

Recall the definition of $L=c\log n/r^d$ for $c$ large enough from Section~\ref{sec:upper_bound_covTime_rgg_nonFixed_r}, and so in particular $L$ is a large constant here. Recall also that by the proof of Theorem~\ref{thm:nonfixed_r_upper_bound_backbone}, for every $x \in \cG_n$ there exists $y \in \cG_n$ at distance $rL /\sqrt{5}=O(\log n)$, with $y  \in \mathrm{Q_j} \in \mathcal{C}^{\text{site}}_{\mathrm{Q}}(n)$. In particular, by construction of dense cubes inside this giant component, observe that $y \in L_1$.    We thus may find, an axis-parallel square $S$ of side length $3C \sqrt{\log n}$ containing both $x$ and $y$, such that both $x$ and $y$ are at Euclidean distance at least say $C\sqrt{\log n}$ from the boundary of $S$, and such that the left bottom corner of $S$ is an integer coordinate. We now show that we cannot separate $x$ from $y$ by removing less than $\eps' \log n$ many vertices of $\cG_n$ inside $S$ (and therefore we cannot separate $x$ and $y$ in $\cG_n$ either), for some $0 < \eps' < \eps$. Since this will hold a.a.s.\ for all $x,y$ simultaneously (with $S$ changing according to the positions of $x,y$, and also for vertices $x \in \cG_n$ close to the boundary of $\Lambda_n$, as we show below), we conclude that the graph $\cG_n$ is $\eps' \log n$-connected. By Menger's theorem, for every pair of vertices there exist $\eps' \log n$ internally vertex-disjoint paths inside $S$ between any pair of vertices, and we will show then that they can be chosen of constant length. 

To do so, we adapt ideas from a classic $k$-connectivity argument on random geometric graphs, in spirit our argument is similar to the proof of Theorem 1.1, in particular Proposition 5.2 in Penrose's paper~\cite{Pen99} applied with $k=\eps'\log n$. 
First, we show that after removing less than $\eps'\log n$ many vertices, $\cG_n$ a.a.s.\ does not have two connected components $\mathcal{C}_1$ and $\mathcal{C}_2$ both of Euclidean diameter larger than $C \sqrt{\log n}/5$ ($C$ the same as before). In fact, we show that this cannot happen inside $S$:
indeed, if this were the case, then consider the cells of the tessellation with cells of sidelength $r/\sqrt{20}$ forming the external boundary of the larger of the two components inside $\Lambda_n$, say $\mathcal{C}_1$. 
Then, for each component of $S \setminus \mathcal{C}_1$, by Lemma~\ref{lem_Timar}, we know that its interior boundary is diagonally connected; in particular, this holds for $\mathcal{C}_2$	(which is one of these components). 
The part of the exterior boundary of $\mathcal{C}_2$ (viewed as a subset of $\mathbb Z^d$) that lies in the infinite component of $\mathbb Z^d$ is also diagonally connected (once again by Lemma~\ref{lem_Timar}, considering only $\mathcal{C}_2$ and its complement for defining the infinite component). Now, when taking the intersection of this boundary with $S$, since $|\mathcal{C}_2| \le |\mathcal{C}_1|$, it is of the same order (for example Lemma 12 of~\cite{treewidth}).
Now, this set is not necessarily diagonally connected anymore. However, since both $x$ and $y$ are at distance at least $C\sqrt{\log n}$ from the boundary of $S$ (by construction of $S$), there must exist a diagonally connected part inside $S$ of length $C\sqrt{\log n}/5$ as well. The number of cells of this exterior boundary inside $S$ is at least $(C \sqrt{\log n} /20)/(r/\sqrt{20}) \ge C/4$. Observe that  the maximal Euclidean distance between two vertices in any two diagonally adjacent cells is at most $\sqrt{3r^2/20}\le r/2$, and hence, if there were any vertex inside a cell belonging to the external boundary of $\mathcal{C}_1$, such a vertex would be connected by an edge to a vertex in $\mathcal{C}_1$ (indeed, if the neighboring square was hitting only an edge of $\mathcal{C}_1$, its distance to the edge were at most $r/2$, and from there to the closer endpoint of the edge at most $r/2$, so in total at most $r$). Now, denote by $X$ the number of vertices in a fixed set of $C/4$ cells of radius $r/\sqrt{20}$. We have by standard Chernoff bounds for Poisson random variables
$$
\Pr(X \le \eps'\log n)=\Pr(\mathrm{Poi}(Cr^2/80) \le \eps'\log n) \le (eC/(\pi \eps'80))^{\eps'\log n}e^{-Cr^2/80} \le e^{-Cr^2/200},
$$
assuming $C$ large enough and $\eps' < \eps$.
 The number of  connected sets of size $C/4$ is at most $O(n 8^{2C/4})$: first, one has to find the starting cell, for which there are $O(n)$ choices; then, one can find a spanning tree with $C-1$ edges that defines the connected set, and traversing all edges exactly twice, say in depth-first search order, we cover the whole connected set; the result follows by noting that at each step there are at most $8$ choices for choosing the next cell. Hence, by Chernoff bounds together with a union bound over all such animals, the probability of having in one of them less than $\eps'\log n$ many vertices is at most 
$O(n8^{C/4}e^{-r^2 C/200})=O(n8^{C/4}e^{-(1+\eps)C\log n/(200\pi) })=o(n^{-2})$, where the equality follows by choosing $C$ large enough. Hence, for $C$ large, by taking a union bound over all pairs of vertices $x,y \in \cG_n$ sufficiently far away from the boundary, a.a.s.\ such two components cannot exist. 

Next, if after removing less than $\eps'\log n$ vertices there would exist a connected component $\mathcal{C}_1$ of Euclidean diameter at most $C\sqrt{\log n}/5$ inside $S$, this component a.a.s.\ has $O(\log n)$ vertices. If the Euclidean diameter of such a component $\mathcal{C}_1$ were less than $\eps''\sqrt{\log n}$ (for $\eps''< \eps' < \eps$ small enough; we define the Euclidean diameter of an isolated vertex to be zero and include this in this case), then note that the probability to have at most $\eps' \log n$ points in the annulus obtained by taking the difference of the ball of outer radius $r$ and inner radius $\eps''\sqrt{\log n/\pi}$ centered around any of the vertices in $\mathcal{C}_1$, is, once again by Chernoff bounds for Poisson variables, at most 
$$
O(\log n)\Pr(\mathrm{Poi}((1+\eps-\eps'')\log n) \le \eps'\log n)=(e(1+\eps-\eps'')/\eps')^{\eps'\log n}e^{-(1+\eps-\eps''+o(1))\log n}\le n^{-1-\eps'''},
$$
for some $\eps''' > 0$, where we assumed $\eps'' < \eps'$. By a union bound over all at most $n(1+o(1))$ components such a component a.a.s.\ does not exist. 

Otherwise, if the Euclidean diameter were at least $\eps''\sqrt{\log n/\pi}$ (but still at most $C\sqrt{\log n}/5$, and therefore still completely inside $S$), then observe that the smallest axis-parallel rectangle $R$ containing all vertices of such a component has to have without loss of generality horizontal length at least $(\eps''\sqrt{\log n/\pi})/2$. Almost surely, such rectangle (of area $O(\log n)$) is defined by at most $4$ points, and hence, for every fixed $S$, there are at most $O(\log^4 n)$ choices for the rectangle. Therefore, the halfcircle of radius $r$ to the left of the leftmost point of $R$ as well as the halfcircle to the right of the rightmost point of $R$ have to contain at most $\eps'\log n$ points (otherwise by removing less than $\eps'\log n$ vertices $\mathcal{C}_1$ is not disconnected). The probability for this to happen is, by a union bound over all $O(n)$ choices for $S$ and thus in total over all at most a.a.s.\ $O(n \log^4 n)$ choices of rectangles, by another Chernoff  bound for Poisson random variables, at most 
$$
O(n \log^4 n)\Pr(\mathrm{Poi}((1+\eps)\log n) \le \eps'\log n)
\le (e(1+\eps)/\eps')^{\eps'\log n}n^{-\eps+o(1)}=o(1), 
$$
where we again assumed $\eps' < \eps$. Hence, a.a.s.\ such a component also cannot exist. Thus, we cannot separate $x \in \cG_n$ from $y \in \cG_n$ by eliminating less than $\eps' \log n$ vertices inside $S$. By Menger's theorem, there are thus $\eps'\log n$ internally vertex-disjoint paths between $x$ and $y$, and since the paths are all inside $S$, they can be shortened if needed to be of constant length in terms of graph distance. This case is done.

We now explain the modifications when $x\in \cG_n$ is close to the boundary: first, we claim that there exists $\eps'' > 0$ sufficiently small such that for any vertex $x$ at Euclidean distance at most $3C\sqrt{\log n}$ from at least one boundary of $\Lambda_n$, there are at least $\varepsilon'' \log n$ neighbors of $x$, all of them at distance at least $d_E(x,\partial\Lambda_n)+\delta r$ from $\partial \Lambda_n$. Indeed, suppose first that $x$ is not in a corner. Denoting by $B(x,r)$ the Euclidean ball centered at $x$ of radius $r$ and by $\mathrm{vol}(\cdot)$ its area, we then have $\mathrm{vol}(B(x,r)\cap \Lambda_n) \ge  \mathrm{vol}(B(x,r))/2=\pi r^2/2 $. Consider the restricted area $B(x,r)\cap \{z\in \Lambda_n: d_E(z,\partial \Lambda_n)>d_E(x,\partial \Lambda_n)+\eta r\}$ for some $\eta>0$. For any $\eta > 0$, there is $\alpha=\alpha(\eta)>0$ such that the area of this region is, by assumption on $r$, larger than $\mu:=(\frac{1}{2} + \alpha) \log n$. By Chernoff bounds, for $c > 0$ small enough,
$$
\mathbb{P}(Po(\mu)\ge (1-\delta)\mu) \ge 1-e^{-\mu(\delta+(1-\delta)\log(1-\delta))}.
$$
Choosing $\delta=(\mu-c)/\mu$ we get that this probability
is for small enough $c > 0$ at least
$$
1-\exp(-(\mu-c)-\mu(1-\delta)\log(1-\delta)).
$$
Since $-(1-\delta)\log(1-\delta) \le (1-\delta)^{0.9}$ for $\delta$ close to $1$ and $(1-\delta)^{0.9}\mu = c^{0.9}\mu^{0.1}$, this probability is at least
$$
1-\exp\left(-(\frac12+\alpha-c+c^{0.9}(\frac12+\alpha)^{0.1})\log n\right) \ge 1-\exp\left(-(\frac12+\alpha/2)\log n)\right)=1-n^{-\frac12-\frac{\alpha}{2}}.
$$
 Since there are a.a.s.\ $O(n^{1/2}\log^{1/2}n)$ at distance at most $O(r)$ from the boundary of $\Lambda_n$, by a union bound, this holds a.a.s.\ for all such vertices. The adaptation to the corners is similar, now noting that instead of half the area of a circle we get a fourth, but the union bound is now over $O(r^2)$ vertices at most. The claim is proved.  

Now, starting from a vertex $x$ with $d_E(x,\partial \Lambda_n)<3C\sqrt{\log n}$, for $\eps''' < c$, we may thus find $\eps''' \log n$ neighbors all at distance at least $\eta r$ further away from the boundary from $\Lambda_n$ than $x$. For each of these neighbors $x_i$, by the previous claim, we find one more neighbor, being again at least $\eta r$ further away from the boundary of $\Lambda_n$ than $x_i$, and by choosing $\eps'''$ sufficiently small, we may assume all of them disjoint from all previously found neighbors. Iterating at most $3C/\eta=O(1)$ many times, we find a total of $(3C/\eta)\eps''' \log n$ vertices, which can be chosen to be all disjoint by choosing $\eps'''$ small enough so that $ (3C/\eta)\eps''' < c$. Note that the final vertices of this construction then are at distance at least $3C\sqrt{\log n}$ away from the boundary of $\Lambda_n$. For these we apply the previous result, so that the latter vertices are all connected to $y$ via $\Omega(\log n)$ internally vertex-disjoint paths, which are inside some $S$, and can therefore be chosen to be of constant length. Concatenating these with the original paths starting from $x$ (which by construction are all of constant length as well), we find thus $\Omega(\log n)$ internally vertex-disjoint paths from $x$ to $y$, and this case is done as well. The theorem follows.

 \section{Concluding remarks and future work}\label{sec:conclusion}

We showed that the cover time undergoes a jump at the threshold of connectivity, decreasing from $\Theta(n \log^2 n)$ to $\Theta(n \log n)$ a.a.s. It is therefore natural to ask what happens right at the threshold of connectivity: to this end, assume that $r^d=(\log n+ f(n))/V_d$, where $f(n)$ may be positive or negative, and $|f(n)|=o(\log n)$. If $f(n)$ is negative and $|f(n)|=\omega(1)$, the same argument as in Case 2 of Section~\ref{sec:lower_bound_fixed_radius}, 
can be used to show that the number of vertices of degree $1$ inside the giant component is a.a.s. $\Omega(nr^d V_d e^{-r^d V_d})=\Omega(e^{-f(n)})$ (which is positive since $f(n)$ is negative). Hence, we obtain a set $V'$ of those vertices of degree 1, $x,y \in \cG_n$ for which $\log(|V'|)=-f(n)$, and any pair of vertices $x,y \in V'$ satisfies  $\mathcal{R}(x\longleftrightarrow y) \geq 1$. Thus, by Theorem~\ref{thm:lowerboundcov}, a.a.s. $\tau_{\cov} \ge \gamma n \log n |f(n)|$. We do not know for which values of $f(n)$ this lower bound is tight. On the other hand, the upper bound of Theorem~\ref{thm:nonfixed_r_upper_bound_backbone} shows that for any two vertices $x,y$ belonging to a dense cube of the auxiliary giant component (as defined there), $\mathcal{R}(x\longleftrightarrow y) =O(1/\log n)$. The value of the cover time is thus determined by the maximal resistance between pairs of vertices, one of which being outside this auxiliary giant, but still inside the giant. It is subject of further work to obtain results on these pairs.
%We do not know for which values of $f(n)$ this lower bound is tight: for some values of $f(n)$ a corresponding upper bound could perhaps be obtained by splitting the cubes into certain good cubes where the flow can be nicely dissipated and bad cubes where not, and combining them in an optimal way. We do not know what happens if $f(n)$ is positive.

More generally, in the same spirit as the very precise works on the cover time for the random walk on $\mathbb{Z}^d$, one could ask for more precise estimates on the cover time for any $r \le r_c$. However, since we do not have very precise estimates on the Green's function, this seems to be out of reach with the current techniques.

%\subsection*{Acknowledgment} The authors would like the anonymous referees for pointing out minor mistakes in a previous version.
\begin{comment}
\cmc{Me gusta lo que está escrito, no agregaría más.}

\dmc{feel free if you want to add a bit your idea}

believe that 
 \cmc{The following is an idea. First just below the threshold for connectivity: The correct bound should be the one given by the number of edges of degree 1 (this works in general, no formal proof just a conjecture, as it works for a lot of grpahs, even for LRP, we won't mention this here). Now, to prove this I think that we should tessellate $\Lambda_n$ into cubes big enough (maybe polynomial in $n$, no idea on how big the cubes should be), what should happen is that some cubes are good and some are bad, in the sense that the good cubes should allow to build a "good flow" and the bad ones another flow that dissipates energy in a bad way. The combination of flows should give the correct bound. We talked about this idea while walking in the forest.}
\end{comment}
 
 \bibliographystyle{abbrv}
    \bibliography{ref}

@book{levinperes,
title={Markov chains and mixing times},
author={Levin, D. A. and Peres, Y.},
year={2017},
volume={107},
publisher={American Mathematical Society}
}

@article{aberesistance,
author={Abe, Y.},
year={2015}, 
title={Effective resistances for supercritical percolation clusters in boxes},
journal={Annales de l'IHP Probabilités et Statistiques},
volume={51},
number={3},
pages={935--946}
}

@book{kesten,
author={Kesten, H.},
year={1982},
title={Percolation theory for mathematicians (Vol. 2)}, 
publisher={Boston: Birkhäuser}
}

@article{pisztorapenrose,
author={Penrose, M. D. and Pisztora, A.},
year={1996},
title={Large deviations for discrete and continuous percolation}, 
journal={Advances in Applied Probability},
volume={28},
number={1},
pages={29--52}
}

@article{liggettdomination,
author={Liggett, T. M. and Schonmann, R. H. and Stacey, A. M.},
year={1997}, 
title={Domination by product measures}, 
journal={The Annals of Probability}, volume={25},
number={1},
pages={71--95}
}

@book{bollobaspercolation,
  title={Percolation},
  author={Bollob{\'a}s, B{\'e}la and Riordan, Oliver},
  year={2006},
  publisher={Cambridge University Press}
}

@article{friedrich2013diameter,
  title={Diameter and broadcast time of random geometric graphs in arbitrary dimensions},
  author={Friedrich, Tobias and Sauerwald, Thomas and Stauffer, Alexandre},
  journal={Algorithmica},
  volume={67},
  pages={65--88},
  year={2013},
  publisher={Springer}
}

@Book{Pen03,
 Author = {Penrose, M.},
 Title = {{Random geometric graphs}},
 fseries = {{Oxford Studies in Probability}},
 series = {{Oxf. Stud. Probab.}},
 Volume = {5},
 Year = {2003},
 Publisher = {Oxford: Oxford University Press},
}

@InCollection{Pen16,
 Author = {Penrose, M.},
 Title = {{Lectures on random geometric graphs}},
 BookTitle = {{Random graphs, geometry and asymptotic structure}},
 Editor = {Fountoulakis, Nikolaos and Hefetz, Dan},
 Pages = {67--101},
 Year = {2016},
 Publisher = {Cambridge: Cambridge University Press},
 DOI = {10.1017/CBO9781316479988.004},
}

@article{Gil61,
  title={Random Plane Networks},
  author={Gilbert, E. N.},
  journal={J. Soc. Ind. Appl. Math.},
  fjournal={Journal of The Society for Industrial and Applied Mathematics},
  year={1961},
  volume={9},
  pages={533-543},
  DOI={10.1137/0109045},
}

@article{DPRZ04,
author={Dembo, Amin and Peres, Yuval and Rosen, Jay and Zeitouni, Ofer},
title={Cover times for {B}rownian motion and
random walks in two dimensions},
journal={Annals of Mathematics},
volume={160},
pages={433--464}, 
year={2004},
}

@article{BeliusKistler,
title={The subleading order of two dimensional cover times},
author={Belius, D. and Kistler, N.}, journal={Probability
Theory and Related Fields},
volume={167},
number={1-2},
pages={461--552},
year={2017},
}

@article{GuptaKumar,
author={Gupta, P. and Kumar, P.R.}, title={Critical power for asymptotic connectivity in wireless networks},
journal={In: Stochastic Analysis, Control, Optimization and Applications: A Volume
in Honor of W. H. Fleming},
year={1998},
pages={547--566},
}

@article{Aky02,
author={Akyildiz, I.~F. and Su, W. and Sankarasubramaniam, Y. and Cayirci, E.},
title={Wireless sensor networks: a survey},
fjournal={Computer Networks},
journal={Comput. Netw.},
volume={38},
number = {4},
pages={393-–422}, 
year={2002},
DOI={10.1016/S1389-1286(01)00302-4},
}

@article{FLMPSSSvL19,
  title={Communication-free massively distributed graph generation},
  author={Funke, Daniel and Lamm, Sebastian and Meyer, Ulrich and Penschuck, Manuel and Sanders, Peter and Schulz, Christian and Strash, D.arren and von Looz, Moritz},
  fjournal={Journal of Parallel and Distributed Computing},
  journal={J. Parallel Distrib. Comput.},
  volume={131},
  pages={200--217},
  year={2019},
  publisher={Elsevier},
  DOI={10.1016/j.jpdc.2019.03.011},
}

@article{Nek07,
  title={Worm epidemics in wireless ad-hoc networks},
  author={Nekovee, Maziar},
  fjournal={New Journal of Physics},
  journal={New J. Phys.},
  volume={9},
  pages={189},
  year={2007},
  publisher={IOP Publishing},
  DOI={10.1088/1367-2630/9/6/189},
}

@inproceedings{chandra1989electrical,
  title={The electrical resistance of a graph captures its commute and cover times},
  author={Chandra, Ashok K and Raghavan, Prabhakar and Ruzzo, Walter L and Smolensky, Roman},
  booktitle={Proceedings of the twenty-first annual ACM {S}ymposium on {T}heory of {C}omputing},
  pages={574--586},
  year={1989}
}

@inproceedings{kahn2000cover,
  title={The cover time, the blanket time, and the {M}atthews bound},
  author={Kahn, Jeff and Kim, Jeong Han and Lovász, László and Vu, Van H},
  booktitle={Proceedings 41st Annual Symposium on Foundations of Computer Science},
  pages={467--475},
  year={2000},
  organization={IEEE}
}

@article{covertimergg,
  title={On the cover time and mixing time of random geometric graphs},
  author={Avin, Chen and Ercal, Gunes},
  journal={Theoretical Computer Science},
  volume={380},
  number={1-2},
  pages={2--22},
  year={2007},
  publisher={Elsevier}
}

@article{Pen99,
title={On k-connectivity for a geometric random graph},
author={Penrose, Matthew},
journal={Random Structures and Algorithms},
year={1999},
volume={15},
number={2},
pages={145--164}
}

@article{Kiwi,
title={Cover and hitting times of hyperbolic random graphs},
author={Kiwi, Marcos and Schepers, Markus and Sylvester, John},
journal={Random Structures and Algorithms},
year={2024},
volume={65},
number={4},
pages={915--978}
}

@article{cooper2011cover,
  title={The cover time of random geometric graphs},
  author={Cooper, Colin and Frieze, Alan},
  journal={Random Structures and Algorithms},
  volume={38},
  number={3},
  pages={324--349},
  year={2011},
  publisher={Wiley Online Library}
}

@article{feige1995upper,
  title={A tight upper bound on the cover time for random walks on graphs},
  author={Feige, Uriel},
  journal={Random Structures and Algorithms},
  volume={6},
  number={1},
  pages={51--54},
  year={1995},
  publisher={Citeseer}
}

@article{DLP12,
  title={Cover times, blanket times, and
majorizing measures},
  author={Ding, Jian and Lee, James R. and Peres, Yuval},
  journal={Annals of {M}athematics},
  volume={175},
  pages={1409--1471},
  year={2012},
}

@article{FPT22,
  title={On the Cover Time of the Emerging Giant},
  author={Frieze, Alan and Pegden, Wesley and Tkocz, Tomasz},
  journal={SIAM Journal on Discrete Mathematics},
  volume={36},
number={3},
  pages={10.1137/21M1441468},
  year={2022},
}

@article{CFL,
title={Cover time of a random graph with given degree sequence II: Allowing vertices of degree two},
author={Cooper, Colin and Frieze, Alan and Lubetzky, Eval}, 
journal={Random Structures and Algorithms}, 
volume={45},
year={2014},
pages={627--674},
}

@article{BDNP,
title={The evolution of the cover time},
author={Barlow, Martin and Ding, Jian and Nachmias, Asaf and Peres, Yuval},
journal={Combin. Probab. Comput.},
volume={20},
year={2011},
pages={331--345},
}

@article{DingAlg,
title={Asymptotic of cover times via {G}aussian {F}ree {F}ields: Bounded degree graphs and general trees},
author={Ding, Jian},
journal={Ann. Probab.}, 
volume={42}, 
year={2014}, 
pages={464--496}
}

@article{Belius,
title={Gumbel fluctuations for cover times in the discrete torus},
author={Belius, David},
journal={Probability Theory
and Related Fields},
volume={157},
pages={635--689}, 
year={2013}
}

@article{Abe2,
title={Second-order term of cover time for planar simple random walk},
author={Abe, Y.},
journal={Journal of
Theoretical Probability}, volume={34},
pages={1689--1747},
year={2021},
}

@article{Aldous89,
author={Aldous, David},
title={Threshold limits for cover times},
journal={Journal of Theoretical Probability}, volume={4},
pages={197--211},
year={1991},
}

@article{CLS21,
author={Cortines, Aser and Louidor, Oren and Saglietti, Santiago},
title={A scaling limit for the cover time of the binary tree}, 
journal={Advances in Mathematics},
volume={391},
pages={107974}, 
year={2021},
}

@article{LS24,
author={Louidor, Oren and Saglietti, Santiago},
title={Tightness for the Cover Time of Wired Planar Domains},
journal={Preprint available at https://arxiv.org/pdf/2406.11034},
year={2024},
}

@article{LS26,
author={Louidor, Oren and Saglietti, Santiago},
title={A Limit in Law for the Cover Time and Last Visited Vertex of Wired Planar Domains},
journal={Preprint available at https://arxiv.org/abs/2602.21277},
year={2026},
}

@article{treewidth,
author={Mitsche, Dieter and Perarnau, Guillem},
title={On Treewidth and Related Parameters of Random Geometric Graphs},
journal={SIAM Journal on Discrete Mathematics},
volume={31},
number={2},
year={2017},
pages={1328--1354}
}

@book{AF,
  author    = {Aldous, David J. and Fill, James A.},
  title     = {Reversible Markov Chains and Random Walks on Graphs},
  journal       = {Preprint available at https://www.stat.berkeley.edu/~aldous/RWG/book.html}
}

@article{Ding,
author={Ding, Jian}, 
title={On cover times for 2d lattices}, 
journal={Electronic Journal of Probability}, volume={17},
pages={1--18},
year={2012},
}

@article{Lawler92,
author={Lawler, G.F.},
title={On the covering time of a disc by simple random walk in two dimensions},
journal={Seminar on Stochastic Processes},
pages={189--207},
year={1992},
}

@article{ZhaiAlg,
title={Exponential concentration of cover times},
author={Zhai, Alex}, journal={Electron. J. Probab.},
volume={23},
year={2018}, 
pages={1--22},
}

@article{feige1995lower,
  title={A tight lower bound on the cover time for random walks on graphs},
  author={Feige, Uriel},
  journal={Random Structures and Algorithms},
  volume={6},
  number={4},
  pages={433--438},
  year={1995},
  publisher={Citeseer}
}

@article{pisztora1996surface,
  title={Surface order large deviations for Ising, Potts and percolation models},
  author={Pisztora, Agoston},
  journal={Probability Theory and Related Fields},
  volume={104},
  pages={427--466},
  year={1996},
  doi={10.1007/BF01198161}
}

@book{LyonsPeres2016,
  author    = {Lyons, Russell and Peres, Yuval},
  title     = {Probability on Trees and Networks},
  publisher = {Cambridge University Press},
  address   = {Cambridge},
  year      = {2016},
  note      = {See Chapter 3, Section 1, ``Flows, Cutsets, and Random Paths,'' especially Eq. (2.17)}
}

@misc{penrose2026components,
  author        = {Penrose, Mathew D. and Yang, Xiaochuan},
  title         = {On the components of random geometric graphs in the dense limit},
  year          = {2026},
  eprint        = {2501.02676},
  archivePrefix = {arXiv},
  primaryClass  = {math.PR},
  doi           = {10.48550/arXiv.2501.02676}
}

@article{timar2013boundary,
  author  = {Tim{\'a}r, {\'A}d{\'a}m},
  title   = {Boundary-connectivity via graph theory},
  journal = {Proceedings of the American Mathematical Society},
  volume  = {141},
  number  = {2},
  pages   = {475--480},
  year    = {2013},
  doi     = {10.1090/S0002-9939-2012-11333-4}
}
\end{document}